\documentclass[oneside,reqno,10pt]{amsart}

\usepackage[margin=1in]{geometry}
\usepackage{mathtools}
\usepackage{amsmath}
\usepackage{amsthm}
\usepackage{amssymb}
\usepackage{tikz}
\usepackage{tikz-cd}
\usepackage[all,cmtip]{xy}
\usepackage{float}
\usepackage{indentfirst}
\usepackage[mathscr]{euscript}
\usepackage{stmaryrd}
\usepackage{hyperref}
\usepackage{csquotes}
\usepackage{enumitem}

\DeclareFontFamily{U}{mathx}{}
\DeclareFontShape{U}{mathx}{m}{n}{<-> mathx10}{}
\DeclareSymbolFont{mathx}{U}{mathx}{m}{n}
\DeclareMathAccent{\widehat}{0}{mathx}{"70}
\DeclareMathAccent{\widecheck}{0}{mathx}{"71}

\newcommand{\Hom}{\mathrm{Hom}}

\newcommand{\cone}{\mathrm{Cone}}
\newcommand{\supp}{\mathrm{Supp}}
\newcommand{\Mod}{\mathrm{Mod}}
\newcommand{\HMod}{\mathrm{HMod}}

\newcommand{\cW}{\mathcal{W}}
\newcommand{\cY}{\mathcal{Y}}
\newcommand{\Prop}{\mathrm{Prop}}

\newcommand{\cC}{\mathcal{C}}
\newcommand{\Ob}{\mathrm{Ob}}
\newcommand{\Tw}{\mathrm{Tw}}
\newcommand{\Perf}{\mathrm{Perf}}
\newcommand{\RHom}{\mathrm{RHom}}
\newcommand{\cD}{\mathcal{D}}
\newcommand{\Z}{\mathbb{Z}}
\newcommand{\op}{^\mathrm{op}}
\newcommand{\vb}{\,|\,}

\newcommand{\cF}{\mathcal{F}}
\newcommand{\Cone}{\mathrm{Cone}}
\newcommand{\C}{\mathbb{C}}
\newcommand{\R}{\mathbb{R}}

\newcommand{\CW}{\mathrm{CW}}

\makeatletter
\def\blfootnote{\gdef\@thefnmark{}\@footnotetext}
\makeatother

\theoremstyle{plain}
\newtheorem{thm}{Theorem}[section]
\newtheorem{prop}[thm]{Proposition}
\newtheorem{lem}[thm]{Lemma}
\newtheorem{cor}[thm]{Corollary}

\newtheorem{assumption}[thm]{Assumption}

\newcommand{\s}{\mathrm{sgn}}

\theoremstyle{definition}
\newtheorem{dfn}[thm]{Definition}
\newtheorem{exa}[thm]{Example}
\newtheorem*{ack}{Acknowledgements}
\newtheorem{notation}[thm]{Notation}

\theoremstyle{remark}
\newtheorem{rmk}[thm]{Remark}

\numberwithin{equation}{section}

\title[Proper modules over Ginzburg dg algebras and compact Fukaya categories of plumbings]{Proper modules over Ginzburg dg algebras and compact Fukaya categories of plumbings}

\author{Wonbo Jeong}
\address[Wonbo Jeong]{Department of Mathematics and Center for Nano Materials\\ Sogang University\\ 35 Baekbeom-ro\\ Mapo-gu\\ Seoul 04107\\ Republic of Korea}
\email{wonbo.jeong@gmail.com}

\author{Dogancan Karabas}
\address[Dogancan Karabas]{Undergraduate Program, Temple University Japan Campus, Tokyo, Japan}
\email{dogancan.karabas@tuj.temple.edu}

\author{Sangjin Lee}
\address[Sangjin Lee]{Korea Institute for Advanced Study, 85 Hoegiro Dongdaemun-gu, Seoul 02455, Republic of Korea}
\email{sangjinlee@kias.re.kr}

\begin{document}

\begin{abstract}
	 We study {\em Ginzburg dg algebras} which appear at the intersection of representation theory and symplectic topology.
	First, we provide a collection of proper modules that generates all proper modules over a Ginzburg dg algebra, without assuming the Jacobi-finite condition. 
	Using this generation result, we study the immersed compact Fukaya category of a general plumbing space. 
	In particular, we prove a generation result for the compact Fukaya category and show that it is equivalent to the category of proper modules over the wrapped Fukaya category, and hence to the category of microlocal sheaves on the Lagrangian skeleton.
\end{abstract}

\maketitle

\tableofcontents

\section{Introduction}
\label{section introduction}

One of the main objects of study in this paper is the {\em Ginzburg dg algebra}, which appears at the intersection of {\em representation theory}, {\em symplectic topology}, and {\em related fields} such as {\em mathematical physics}.
A main goal of this paper is to prove that the (immersed) compact Fukaya category of a plumbing space is equivalent to the category of proper modules over its wrapped Fukaya category, which in turn is equivalent to the category of proper modules over the associated Ginzburg dg algebra (see Theorem \ref{thm intro compact Lagrangian}).
Moreover, this yields an equivalence between the compact Fukaya category of a plumbing space and the category of microlocal sheaves on its Lagrangian skeleton (see Corollary \ref{cor microlocal sheaves}).
To prove this, we establish a generation result for proper modules over Ginzburg dg algebras (see Theorem \ref{thm intro generation}).

In representation theory, Ginzburg \cite{gin06} introduced the notion of Ginzburg dg algebra as an example of a 3-Calabi--Yau algebra, also defined in \cite{gin06}. 
Since then, Ginzburg dg algebras have been studied extensively in representation theory (for example, see \cite{gin06, kel-yan11, kel11, Hermes16}).
They also appear in mathematical physics and algebraic geometry:
As the term ``3-Calabi--Yau algebra'' suggests, these algebras are closely related to the study of Calabi--Yau threefolds that mathematical physicists and algebraic geometers investigate.
Thus, it is not surprising that Ginzburg dg algebras, as natural examples of 3-Calabi--Yau algebras, arise in these areas of mathematics.

Ginzburg dg algebras also appear naturally in symplectic topology. 
Since symplectic topology has strong connections with mathematical physics and algebraic geometry, one could expect such a relation.
However, it is worth emphasizing that Ginzburg dg algebras arise in a purely symplectic-topological context as well.
More precisely, it is known that Ginzburg dg algebras are equivalent to the {\em wrapped Fukaya categories} of plumbing spaces, see \cite{asp21, kar-lee24-2, etg-lek17,etg-lek19,lek-ued20, Hu-Lau-Tan24}.
Thus, a deep understanding of Ginzburg dg algebras could contribute to the study of mathematical physics, algebraic geometry, and symplectic topology.

In the current paper, we first study proper modules (or equivalently, derived modules of finite total dimensions) over Ginzburg dg algebra, motivated from a symplectic topology problem. 
Especially, we characterize a collection of proper modules that generate all proper modules, see Theorem \ref{thm intro usual Ginzburg dg algebra}.
(After doing simple linear algebra, a corollary of Theorem \ref{thm intro usual Ginzburg dg algebra} would be a characterization of indecomposable modules, see examples in Section \ref{section applications and examples}.) 
We note that such generation results are known for {\em some} Ginzburg dg algebras (i.e., ones having finite dimensional Jacobi algebra), but not all. 
Theorem \ref{thm intro usual Ginzburg dg algebra} covers unknown cases.

After that, we solve the motivational symplectic topology problem by employing the generation result. 
The motivating problem concerns the {\em compact Fukaya category} of plumbing spaces. 
Here, the compact Fukaya category is the category of compact, exact, (possibly) immersed Lagrangians. 
(See Definition \ref{dfn compact Fukaya category} for more details.) 
Let us recall that compact Fukaya category is a subcategory of wrapped Fukaya category, as implied by \cite{Jeong-Karabas-Lee26}.
Since wrapped Fukaya categories of plumbings are equivalent to Ginzburg dg algebras, the Yoneda functor sends the compact Fukaya category into the category of modules over Ginzburg dg algebras. 
Moreover, compactness implies that the compact Fukaya category embeds into the category of proper modules.
The second main result of the paper is to show that the collection of generating modules is contained in the Yoneda image of the compact Fukaya category. 
This implies the equivalence between the compact Fukaya category of plumbings and the category of proper modules over the associated Ginzburg dg algebras (Theorem \ref{thm immersed module correspondence}), as well as a generation result for the compact Fukaya category (Theorem~\ref{thm intro compact Lagrangian}). 
Moreover, it implies the equivalence between the compact Fukaya category of plumbings and the category of microlocal sheaves on their Lagrangian skeleta (Corollary \ref{cor microlocal sheaves}).

In the rest of Section \ref{section introduction}, we provide more details on our motivations/ideas/results.
Let $Q$ be a quiver and $W$ be a formal sum of cycles of $Q$. 
The Ginzburg dg algebra $\Gamma_3(Q,W)$ of a quiver with potential $(Q,W)$ was first introduced in \cite{gin06} as an example of $3$-Calabi--Yau algebras. 
It extends to a more general case: 
If $Q$ is a {\em graded} quiver and $W$ is a formal sum of cycles of degree $(3-n)$, the same definition provides an $n$-Calabi--Yau algebra, which is also called Ginzburg dg algebra $\Gamma_n(Q,W)$.
See \cite{kel11}.
(We note that for $n=3$, the ungraded quiver $Q$ can be seen as the graded quiver with a simple graded quiver such that every edge is of degree $0$.)

As mentioned above, Ginzburg dg algebras (especially as $3$-Calabi--Yau algebras/categories) have appeared in many works in representation theory, algebraic geometry, mathematical physics, and related fields.
Some examples are as follows:
\cite{Amiot09, Keller11, kel-yan11, Derksen-Weyman-Zelevinsky08} studied cluster structures on the category of proper modules over $\Gamma_3(Q,W)$, denoted by $\Prop \Gamma_3(Q,W)$, under the assumption that $\Gamma_3(Q,W)$ is {\em Jacobi-finite}. 
(Recall that $\Gamma_3(Q,W)$ is Jacobi-finite if and only if its zeroth cohomology is finite-dimensional.)
Bridgeland and Smith \cite{Bridgeland-Smith15} related Bridgeland stability conditions on $\Prop \,\Gamma_3(Q,W)$ to the geometry of surfaces and quadratic differentials.
Moreover, a Jacobi-finite $\Gamma_3(Q,W)$ and $\Prop\, \Gamma_3(Q,W)$ could provide an computational example of \cite{Kontsevich-Soibelman08, Kontsecich-Soibelman10} that studied the wall-crossing phenomena on the space of stability conditions on $3$-Calabi--Yau categories and their generalized Donaldson--Thomas invariants.

We note that in the Jacobi-finite case, $\Prop \,\Gamma_3(Q,W)$ has a known collection of generating objects, i.e., a collection of proper modules that generates the category; see \cite{King94}. 
Thus, if a quiver $Q$ is equipped with a {\em generic} potential $W$, then there exists a generating collection of proper modules over $\Gamma_3(Q,W)$, since the genericity of $W$ implies the Jacobi-finiteness of $\Gamma_3(Q,W)$. 
Such generating collections have played an important role in many works, including some of the examples mentioned above.

However, not every pair $(Q,W)$ satisfies the Jacobi-finite condition. 
For example, if $Q$ has a cycle and if $W =0$, then $\Gamma_3(Q,W=0)$ has infinite dimensional zeroth cohomology. 
Thus, it would be a natural research direction to seek a specific collection of generating objects for $\Prop\, \Gamma_3(Q,W)$.

The first main result of the current article is the construction of a specific generating collection for $\Prop \, \Gamma_n(Q,0)$. 

\begin{thm}[Corollary \ref{cor generation of ginzburg dg algebra}]
	\label{thm intro usual Ginzburg dg algebra}
	Let $n \geq 3$, and let $Q$ be a non-positively graded quiver. 
	(Roughly speaking, the non-positively graded condition means that the graded quiver does not contain a positively graded cycle. 
	See Definition \ref{dfn non-positivcely grading condition} for the formal definition.)
	Then, we provide a specific collection of proper modules over $\Gamma_n(Q,0)$ that generates $\Prop\, \Gamma_n(Q,0)$.
\end{thm}

An example illustrating Theorem \ref{thm intro usual Ginzburg dg algebra} is given in Section \ref{subsection example}.

\begin{rmk}
	\label{rmk usual Ginzburg} 
	\mbox{}
	\begin{enumerate}
		\item In the statement of Theorem \ref{thm intro usual Ginzburg dg algebra}, we do not explicitly describe the generating collection. 
		Roughly speaking, our construction characterizes indecomposable proper modules. 
		See Remarks \ref{rmk indecomposable} and \ref{rmk indecomposable modules in an example}.
		
		\item Theorem \ref{thm intro usual Ginzburg dg algebra} treats the case $W=0$.
		The reason for restricting to this potential is that this case admits a geometric model arising from symplectic topology. 
		This geometric model helps us formulate the main idea of the proof of Theorem \ref{thm intro usual Ginzburg dg algebra}. 
		However, we expect that Theorem \ref{thm intro usual Ginzburg dg algebra} can be generalized to arbitrary potentials $W$. 
		See Remark \ref{rmk with potential}.
		
		\item We note that some works, for example \cite{kel-yan11}, use the {\em completed} Ginzburg dg algebra. 
		However, according to \cite[Remark 6.2.1]{Booth-Goodbody-Opper25}, under the Jacobi-finite condition, a Ginzburg dg algebra and its completed version are quasi-equivalent.
	\end{enumerate}
\end{rmk}

Our main idea for proving Theorem \ref{thm intro usual Ginzburg dg algebra} originates from symplectic topology. 
In fact, by employing ideas from symplectic topology, we prove a more general version of Theorem \ref{thm intro usual Ginzburg dg algebra}, namely Theorem \ref{thm intro generation}.
Before stating this generalization and explaining the main idea of the proof, let us first describe the motivation coming from symplectic topology.

As mentioned at the beginning, \cite{kar-lee24-2} showed that for a graded quiver $Q$, $\Gamma_n(Q,0)$ is equivalent to the wrapped Fukaya category of a plumbing space $P$ equipped with a grading structure $\eta$, where $P$ is obtained by plumbing copies of $T^*S^n$ along $Q$ and $\eta$ respects the grading on $Q$. 
In fact, \cite{kar-lee24-2} computed the wrapped Fukaya category in a more general setting, namely for plumbing spaces obtained by plumbing arbitrary cotangent bundles and equipped with grading structures $\eta$.
The result is that the wrapped Fukaya category, denoted by $\cW(P;\eta)$, is (Morita equivalent to) an extended/generalized version of a Ginzburg dg algebra. 
This generalized version contains additional morphisms associated to each vertex of $Q$. 
See Theorem \ref{thm wrapped Fukaya category} and Remark \ref{rmk generalized Ginzburg dg algebra} for more details. 
Motivated by this, we refer to $\cW(P;\eta)$ as a {\em generalized} Ginzburg dg algebra/category.

\begin{rmk}
	\label{rmk generalized Ginzburg dga}
	As stated in Theorem \ref{thm wrapped Fukaya category}, the extra morphisms in a generalized Ginzburg dg algebra, which are added at each vertex, are generators of the chain complexes of based loop spaces. 
	Thus, the term ``generalized Ginzburg dg algebra'' can be replaced by ``usual Ginzburg dg algebra with based loop space coefficients''.
\end{rmk}

As a general version of Theorem \ref{thm intro usual Ginzburg dg algebra}, we prove Theorem~\ref{thm intro generation}.

\begin{thm}[Theorem \ref{thm generation of proper modules}]
	\label{thm intro generation} 
	Let $P$ be a plumbing space of dimension $2n \geq 6$, obtained by plumbing cotangent bundles of arbitrary manifolds and equipped with a non-positive grading structure $\eta$ (see Definition \ref{dfn non-positivcely grading condition}). 
	For the corresponding generalized Ginzburg dg algebra $\cW(P;\eta)$, we provide a specific collection of proper modules that generates $\Prop \,\cW(P;\eta)$.  
\end{thm}

As mentioned in Remark \ref{rmk usual Ginzburg} (1), the description of the generating modules we provide can be viewed as a characterization of indecomposable modules.

Our motivation and main idea for the proof of Theorem \ref{thm intro generation} arise from symplectic topology. 
In the second part of the article, namely Sections \ref{section the compact Fukaya category of plumbings} and \ref{section an alternative proof}, we study the motivating symplectic topology problem. 

Our motivation from symplectic topology is the study of the {\em compact Fukaya category}, especially its generation. 
We note that generation of wrapped Fukaya categories is well understood in the literature (see \cite{cdrgg17, gps2}), whereas generation of compact Fukaya categories are not well understood in general; see \cite{abo-smi12, li19} for some known cases.

This project originated from the idea of studying (generation of) compact Fukaya categories via wrapped Fukaya categories.  
More precisely, we note that the {\em compact Fukaya category} of a Weinstein manifold $X$, denoted by $\cF(X)$, can be viewed as a full subcategory of the wrapped Fukaya category $\cW(X)$. 
Applying the Yoneda functor to $\cF(X)$, {\em compactness} implies that the images are proper modules.
In other words,
$
\cY\left(\cF(X)\right) \subset \Prop\, \cW(X),
$
where $\cY$ denotes the Yoneda functor. 
In particular, for plumbing spaces we have
$
\cY\left(\cF(P;\eta)\right) \subset \Prop\, \cW(P;\eta).
$

Theorem \ref{thm intro generation} provides a collection of proper modules $\{M_i\}$ that generate $\Prop\,\cW(P;\eta)$. 
For each generator $M_i$, if there exists a compact Lagrangian $L_i \in \cF(P;\eta)$ such that $\cY(L_i)$ is isomorphic to $M_i$, then this gives an explicit collection of generators $\{L_i\}$ of $\cF(P;\eta)$. 
Theorem \ref{thm intro compact Lagrangian} establishes precisely this result.

\begin{thm}[Theorems \ref{thm immersed module correspondence} and \ref{thm compact Fukaya}]
	\label{thm intro compact Lagrangian}
	For each generating proper module $M$ given in Theorem~\ref{thm intro generation}, there exists a compact Lagrangian $L$ in $P$ such that the Yoneda image of $L$ is isomorphic to $M$. In particular:
	\begin{enumerate}
		\item The (immersed) compact Fukaya category $\cF(P;\eta)$, the category of proper modules $\Prop\,\cW(P;\eta)$, and the category of microlocal sheaves on the Lagrangian skeleton $\mathbb{L}$ of $P$ are equivalent:
		\[
		\cF(P;\eta) \;\simeq\; \Prop\,\cW(P;\eta) \;\simeq\; \mu\mathrm{Sh}(\mathbb{L}).
		\]
		
		\item The category $\cF(P;\eta)$ is generated by a collection of Lagrangians that we describe explicitly.
	\end{enumerate}
\end{thm}

\begin{rmk}
	Theorem \ref{thm intro compact Lagrangian} (1) is a generalization of \cite{nad-zas09,nad09} which implies the same result for cotangent bundles. 
	See Corollary \ref{cor microlocal sheaves}. Theorem \ref{thm intro compact Lagrangian} (2) recovers and generalizes several well-known generation results in symplectic topology, including generation of the compact Fukaya category of plumbings along trees by Lagrangian cores equipped with local systems due to \cite{abo-smi12}.
	The key new feature is that degree-$0$ cycles in the quiver can produce additional generators. 	
	See Remark~\ref{rmk cores generate compact Lagrangians} for a more detailed discussion.
\end{rmk}

We note that although Theorem \ref{thm intro generation} precedes Theorem \ref{thm intro compact Lagrangian}, the motivation and main idea come from Theorem \ref{thm intro compact Lagrangian}. 
More precisely, we first conjectured that a certain collection of compact Lagrangians generates $\cF(P;\eta)$. 
See Section \ref{subsection motivation and assumptions} for the motivation behind this conjecture.
We then translated the geometric conditions expected of these compact Lagrangian generators into conditions on proper modules, and these translated conditions characterize the generating proper modules appearing in Theorem \ref{thm intro generation}.

We should remark that the expected generating Lagrangians could be {\em immersed} (if the plumbing quiver $Q$ has either a degree $0$ cycle or a self-loop).
Thus, the relevant Floer theory is immersed Lagrangian Floer theory. 
We briefly review the immersed Lagrangian Floer theory in the wrapped setting appearing in the literature:
In \cite[Section 6.2]{cdrgg17}, the authors defined an immersed version of Lagrangian Floer theory using Hamiltonian perturbations. 
This suggests the existence of a Fukaya category whose objects include exact immersed Lagrangians, but the authors did not verify all details required to establish the full $\mathcal{A}_\infty$-structure, so the resulting structure was called a ``Fukaya pre-category''. 
(See \cite[Remark~8.14]{cdrgg17} for more details.)
On the other hand, \cite{gao17} resolved this issue using pearly complexes instead of the Hamiltonian-perturbation approach.
Thus, the remaining problem was to establish the equivalence between these two constructions: one defined using Hamiltonian perturbations and the other using pearly complexes. 
This equivalence is proved in our separate paper \cite{Jeong-Karabas-Lee26}. 
In the same paper, we also extend the generation theorem of \cite{cdrgg17} to immersed Lagrangians equipped with certain bounding cochains (equivalently, augmentations) (see Theorem \ref{thm immersed generation recalled}), showing that such Lagrangians are generated by cocores and hence are objects of the wrapped Fukaya category considered in this paper.

Using this fact, for each generator $M$ appearing in Theorem \ref{thm generation of proper modules}, we construct a (possibly immersed) Lagrangian in the wrapped Fukaya category whose Yoneda image is isomorphic to $M$ in Section \ref{section the compact Fukaya category of plumbings}.
 
 We also note that Section \ref{section an alternative proof} provides an alternative proof of Theorem \ref{thm intro compact Lagrangian} using a slightly different method.
 The argument there uses the setting of \cite{gps2}. 
 We include this alternative proof because it gives a more pictorial explanation of the result. 
 However, the drawback of this approach is that it relies on an ad hoc argument using a specific Weinstein sectorial covering of the plumbing space.
 
 In the last section of the paper, we discuss applications and an example. 
 As an application, in Section \ref{sec:stab} we construct a Bridgeland stability condition on the category of interest, namely the category of proper modules over (generalized) Ginzburg dg algebras. 
 In particular, its stability space is nonempty, and these categories could provide a useful test ground for studying Bridgeland stability conditions and examining related theories such as \cite{Kontsevich-Soibelman08}. 
We also give a construction of cluster categories in Section \ref{sec:CYtri}.
Finally, in Section \ref{subsection example}, we fix a specific Ginzburg dg algebra and carry out explicit computations illustrating our results.

\begin{ack}
During this work, Wonbo Jeong was supported by the National Research Foundation of Korea (NRF) grants funded by the Korean government (MIST, No. RS-2025-24803252) and through the G-LAMP program (MOE, RS-2024-00441954).
Part of this work was carried out while Dogancan Karabas was at the Kavli Institute for the Physics and Mathematics of the Universe (Kavli IPMU), the University of Tokyo, where the author was supported by the World Premier International Research Center Initiative (WPI), MEXT, Japan. Further parts of this work were carried out at the Kyoto University Institute for Advanced Study (KUIAS), in the Hiraoka Laboratory, and at Temple University, Japan Campus. 
Sangjin Lee is supported by a KIAS Individual Grant (MG094401) at Korea Institute for Advanced Study.
\end{ack}

\section{Preliminaries} 
\label{section preliminaries} 

\subsection{Dg categories and their modules}
\label{subsection preliminaries on category theory}
We set notations and introduce basic definitions in category theory.

A differential graded (dg) category is a category enriched over the symmetric monoidal category of complexes over a fixed commutative ring $k$. It can also be viewed as an $A_{\infty}$-category whose higher structures, i.e., compositions of order greater than 2, vanish. For further details on dg categories, readers may refer to \cite{kel08}, and for a review of $A_{\infty}$-categories, one can consult \cite{sei08}. We use $d$ for the differential and $\circ$ for compositions of morphisms, and we omit the latter whenever it is convenient. When introducing a dg category, we follow the convention of providing the following five items:
\begin{enumerate}[label = (\roman*)]
	\item {\em Objects:} We list the objects in the category.
	\item {\em Generating morphisms:} We give a set of generating morphisms. They generate all the morphisms as an algebra, not as a module. We will not explicitly mention the existence of identity morphisms, but every object has the identity morphism, and generation can utilize these identity morphisms.
	\item {\em Degrees:} For each generating morphism, we specify its degree.
	\item {\em Differentials:} For each generating morphism, we specify its differential.
	\item {\em Relations:} We specify the relations between generating morphisms. We will omit this item if the generating morphisms freely generate all other morphisms.
\end{enumerate}

Given a dg category $\cC$, we denote by $\Ob\,\cC$ (or simply $\cC$ when it is clear from the context) the
collection of objects in $\cC$.
We write $\Hom^*_{\cC}(A,B)$ (or simply $\Hom^*(A,B)$) for the cohomology of the morphism space $\hom^*(A,B)$ of $\cC$. 
Given a subset $S$ of closed degree zero morphisms in $\cC$, we denote the dg localization of $\cC$ at the morphisms in $S$ by $\cC[S^{-1}]$.

We write $\Tw\, \cC$ for the dg category of twisted complexes in $\cC$, which is a
pretriangulated envelope of $\cC$, and $\Perf\,\cC$ for the split-closure (i.e., idempotent completion) of $\Tw\,\cC$. 
We say dg categories $\cC$ and $\cD$ are quasi-equivalent if there is a chain of dg categories and dg functors
\[\cC\xleftarrow{\sim}\cC'\xrightarrow{\sim}\cD\]
for some dg category $\cC'$, where each dg functor in the chain is a quasi-equivalence. We say $\cC$ and $\cD$ are pretriangulated (resp.\ Morita) equivalent if $\Tw\,\cC$ and $\Tw\,\cD$ (resp.\ $\Perf\,\cC$ and $\Perf\,\cD$) are quasi-equivalent.

Let $\{L_i\}$ be a collection of objects in $\cC$, and let $\cD$ be the full dg subcategory of $\cC$ whose object set is $\{L_i\}$. Then, $\{L_i\}$ generates (resp.\ split-generates) $\cC$ if $\cD$ is pretriangulated (resp.\ Morita) equivalent to $\cC$.

We define $\Mod\, k$ to be the dg category of $\Z$-graded cochain complexes of $k$-modules, localised at quasi-isomorphisms. Equivalently, it is the dg category of cofibrant complexes of $k$-modules. 
We define \(\Perf\, k\) to be the full dg subcategory of \(\Mod\, k\) consisting of perfect complexes, by which we mean bounded complexes of finitely generated projective k-modules.
When k is Noetherian with finite global homological dimension (e.g., a principal ideal domain), every perfect complex can be represented by a bounded complex of finitely generated k-modules.
We also have the following fact, see, e.g., \cite[Exercise 1.18]{kas-sha02}:

\begin{prop}
	\label{proposition HMod}
	If $k$ is a principle ideal domain, any object $K\in \Mod\,k$ is isomorphic to $\bigoplus_d H^d(K)[-d]$ in $H^0(\Mod\, k)$.
\end{prop}

\begin{rmk}
	\label{rmk HMod}
	Thanks to Proposition \ref{proposition HMod}, if $k$ is a principle ideal domain, up to quasi-equivalence, we can assume that the objects of $\Perf\, k\subset \Mod\, k$ are of the form
	\[K_1[d_1]\oplus\ldots\oplus K_m[d_m] \qquad\text{(without differential)}\]
	for some $m\geq 0$, $d_1<\ldots<d_m$, and finitely generated $k$-modules $K_i$.
	We will assume this throughout the paper.
\end{rmk}

We write $\Mod\, \cC$ (resp.\ $\Prop\,\cC$) for the derived dg category of right $\cC$-modules (resp.\ proper right $\cC$-modules), which is given by the internal hom $\RHom(\cC^{\op},\Mod\, k)$ (resp.\ $\RHom(\cC^{\op},\Perf\, k)$). 

A dg category $\cC$ is {\em semifree} if its morphisms, treated as an algebra, are {\em freely} generated by a set of morphisms $\{f_i\}$ (indexed by an ordinal), with the condition that $df_i$ is generated by the set $\{f_j \vb j < i\}$. In this case, $\{f_i\}$ is called a set of generating morphisms of $\cC$. By \cite{canonaco-ornaghi-stellari-2}, when $\cC$ is a semifree dg category, $\Mod \, \cC$ (resp.\ $\Prop \, \cC$) is equivalent to the dg category of dg functors from $\cC^{\op}$ to $\Mod\, k$ (resp.\ $\Perf\, k$) whose morphisms are $A_{\infty}$-natural transformations, or equivalently to the dg category of dg modules whose morphisms are module homomorphisms in the sense of \cite[Section~(1j)]{sei08}. We will use both perspectives throughout the paper.

A dg category $\cD$ is a semifree extension of a dg category $\cC$ if $\cD$ is obtained from $\cC$ by adding a set of objects to the objects of $\cC$ and a set of morphisms $\{f_i\}$ (indexed by an ordinal) freely to the morphisms of $\cC$ (treated as an algebra), with the condition that $df_i$ is generated by the morphisms of $\cC$ and $\{f_j \vb j <i\}$.

In this paper, the Yoneda functor of \(\cC\) is the functor
\[
\cY\colon \cC \to \Mod\,\cC.
\]
For any object \(E \in \cC\), the object \(\cY(E)\) is the dg functor \(\cC^{op} \to \Mod k\) given by $\cY(E)(-) = \hom_{\cC}(-,E)$.

\subsection{Plumbing space} 
\label{subsection plumbing space}
One of the main purposes of the paper is to study the compact Fukaya category of {\em plumbing spaces}.
Thus, it would reasonable to introduce {\em briefly} the definition/construction of plumbing spaces in this preliminary section. 
In the literature, there exist many good references for more details, for example, see \cite{abo-smi12}.
We refer also the reader to \cite[Section 3]{kar-lee24-2}, because we share the same notation. 

We construct a plumbing space from a {\em set of data}, which we call {\em plumbing data}.
\begin{dfn}
	\label{definition plumbing data}
	Let $n$ be a natural number bigger than $1$. 
	{\em Plumbing data} is a triple $(Q, M, \s)$ such that 
	\begin{itemize}
		\item $Q$ is a (finite) quiver, i.e., a directed graph, 
		\item $M$ is a map from the set of vertices of $Q$, denoted by $V(Q)$, to the collection of $n$-dimensional connected, oriented, closed, smooth manifolds, denoted by $\mathcal{O}_n$, i.e., 
		\[M: V(Q) \to \mathcal{O}_n,\]
		(Note that for simplicity of notation, we set the notation $M_v:= M(v)$ for all $v \in V(Q)$.) 
		\item $\s$ is a map from the set of arrows, denoted by $E(Q)$, to $\{1, -1\}$, i.e., 
		\[\s: E(Q) \to \{1, -1\}. \]
	\end{itemize}
\end{dfn}
We note that even if we allow $M_v$ to be a manifold {\em with} boundary, the plumbing space construction given below (and most of results in the current paper) will still work.
However, the resulting plumbing space should be a Weinstein {\em sector}, not a Weinstein manifold. 
In the paper, for the simplicity of arguments, we do not allow $M_v$ to be a manifold with boundary.

When plumbing data $(Q,M,\s)$ is provided, one can construct a Weinstein manifold, called {\em plumbing space}, from the given data. 
Especially, the vertex set $V(Q)$ and the map $M$ defined on $V(Q)$ provide ingredients, and the edge $E(Q)$ and the map $\s$ on $E(Q)$ specify how to ``plumb'' the ingredients. 

More precisely, the ingredients are cotangent bundles $\left\{T^*M_v | v \in V(Q)\right\}$.
If an edge $e \in E(Q)$ has a tail $v \in V(Q)$ and a head $w \in V(Q)$, then one plumbs $T^*M_v$ and $T^*M_w$ at one point. 
In other words, one arbitrarily chooses points $p \in M_v$ and $q \in M_w$ and their small neighborhood $U_p$ of $p$ and $U_q$ of $q$ which are homeomorphic to the $n$-dimensional disk. 
Then, one can identify $T^*U_p$ and $T^*U_q$, where both are homeomorphic to $T^*\mathbb{D}^n \simeq \mathbb{D}^n \times \mathbb{D}^n$.
(Here $\mathbb{D}^n$ denotes the open disk of dimension $n$.)
In order to identify them, we need to set a map 
\begin{align*}
	f: T^*\mathbb{D}^n = \mathbb{D}^n \times \mathbb{D}^n &\to  T^*\mathbb{D}^n =  \mathbb{D}^n \times \mathbb{D}^n\\
	(x_1, \dots, x_n,y_1, \dots, y_n) &\mapsto (-y_1, \dots, -y_n, x_1, \dots, x_n).
\end{align*}
Simply speaking, we identify the base direction of the domain cotangent bundle with the fiber direction of the target. 

Note that the direction of $e$ determines the {\em direction} of the map $f$, in the sense that the domain (resp.\ target) is $T^*U_p$ (resp.\ $T^*U_q$)  $\subset T^*M_v$ and $v$ is the tail of $e$. 
And, we point out that we did not specify how to identify $T^*U_p$ (resp.\ $T^*U_q$) with $T^*\mathbb{D}^n \simeq \mathbb{D}^n \times \mathbb{D}^n$. 
We simply identify $T^*U_p$ and $T^*\mathbb{D}^n$ in the way that the zero sections of are identified with each other as {\em oriented} manifolds. 
For $T^*U_q$, $\s(e)$ determines how to identify. 
If $\s(e) = 1$ (resp.\ $-1$), we identify $T^*U_q$ and $T^*\mathbb{D}$ so that two zero sections are identified with the same (resp.\ opposite) orientation. 
In the end, it determines the intersection sign of $U_p$ and $U_q$ in the identified region, and the intersection sign should coincide with $\s(e)$. 

We note that the above plumbing procedure is {\em local}, since everything happens in a small neighborhood of the chosen plumbing points $p$ and $q$. 
For every edge $e \in E(Q)$, one can perform the above process and the plumbing space is defined to be 
\[\cup_{v \in V(Q)} T^*M_v / \sim,\]
where $\sim$ is the identification according to above plumbing procedure. 

Finally, we point out a well-known fact that if $n \geq 2$, the choice of plumbing points does not effect on (the Weinstein isotopy type of) the resulting plumbing space. 
\begin{dfn}
	\label{definition plumbing space}
	Let $(Q,M,\s)$ be plumbing data. 
	\begin{enumerate}
		\item Let $P(Q,M,\s)$ denote the plumbing space corresponding to the given data set $(Q,M,\s)$. 
		\item Let $\iota_v$ denote the natural map from $T^*M_v$ to $P(Q,M,\s)$ for all $v \in V(Q)$. 
		A {\em Lagrangian core} (resp.\ {\em Lagrangian cocore}), or simply {\em core} (resp.\ {\em cocore}), is $\iota_v(M_v)$ (resp.\ $\iota_v(T^*_xM_v)$, where $x$ is not in any of the identified regions). 
	\end{enumerate}
\end{dfn}

We note that two different plumbing data sets could give Weinstein isotopic plumbing spaces. 
Especially, according to the following proposition, one could come up with multiple plumbing data sets such that the corresponding plumbing spaces are Weinstein isotopic. 

\begin{prop}[Propositions 3.15 and 3.16 of \cite{kar-lee24-2}]
	\label{proposition direction reversing}
	Let $(Q,M,\s)$ be plumbing data and $P(Q,M,\s)$ be the corresponding plumbing space. 
	Let $(Q',M',\s')$ be plumbing data such that 
	\begin{itemize}
		\item $Q$ and $Q'$ are the same quiver except the direction of one arrow $e \in E(Q)$ and $e' \in E(Q')$, i.e., 
		\[V(Q)=V(Q'), E(Q) \setminus \{e\} = E(Q') \setminus \{e'\},\]
		such that $e$ and $e'$ are connecting the same vertices in $V(Q) = V(Q')$, but they have opposite directions,
		\item $M, M': V(Q)=V(Q') \to \mathcal{O}_n$ are the same maps,
		\item $\s(f) = \s'(f)$ for all $f \in E(Q) \setminus \{e\} = E(Q') \setminus \{e'\}$ and 
		\[\s(e) = \begin{cases*}
			-\s'(e') &\text{  if $n$ is an odd integer,}\\
			\s'(e') &\text{  if $n$ is an even integer.}
		\end{cases*}\]
	\end{itemize}
	Then, two plumbing spaces $P(Q,M,\s)$ and $P(Q',M',\s')$ are Weinstein isotopic to each other.
\end{prop}

We end this subsection by introducing a specific Weinstein sectorial covering of a plumbing space $P(Q,M,\s)$, which is associated to the plumbing data $(Q,M,\s)$. 
The covering is important in the current paper, because of the following two reasons: 
One reason is that a plumbing space could be seen as a gluing of each Weinstein sector in the specific covering. 
In other words, the covering provides another construction of a plumbing space, which is the purpose of this subsection. 
The other reason is that \cite{kar-lee24-2} computed wrapped Fukaya categories of plumbing spaces from the Weinstein sectorial covering -- \cite{gps2} proves that if a wrapped Fukaya category of a Weinstein manifold can be computed as a homotopy colimit of those of Weinstein sectors. 
Applying \cite{gps2}, \cite{kar-lee24-2} computed the wrapped Fukaya categories of each sector and their homotopy colimit, i.e., wrapped Fukaya categories of plumbings. 

Recall that a plumbing space $P= P(Q,M,\s)$ is defined to be 
\[\cup_{v \in V(Q)} T^*M_v / \sim,\]
where $\sim$ means the identification described above. 
Then $P$ can be decomposed into subsets of the following two types: the complements of the identified regions in the cotangent bundles $T^*M_v$, and the identified regions themselves.

To be more precise, we point out that the number of edges is the same as the number of plumbing points, and thus, it is also the same as the number of (connected components in) identified regions. 
Moreover, every connected component of the identified regions is the same Weinstein sector.
The Weinstein sector is $T^*\mathbb{D}^n$ with a specific choice of Weinstein structure.
We skip the details, but after convex completion, the Weinstein sector is 
\[\left(\mathbb{C}^n, \partial_\infty \left(\mathbb{R}^n \sqcup i \mathbb{R}^n\right)\right).\]
Let $\Pi_n$ denote the Weinstein sector and let us name it {\em plumbing sector}.
For more details, see \cite[Section 3]{kar-lee24-2}.

Also, let $\widetilde{M}_v$ mean the complement of all identified regions from $M_v$. 
Then, 
\begin{gather}
	\label{eqn Weinstein sectorial covering}
	P = \cup_{v \in V(Q)} T^*\widetilde{M}_v \sqcup \cup_{e \in E(Q)} \Pi_n.
\end{gather}

\subsection{Wrapped Fukaya category of plumbing spaces}
\label{subsection wrapped Fukaya category of plumbing spaces}
Let $\mathcal{W}(W)$ denote (the pretriangulated closure of) the wrapped Fukaya category of a Liouville/Weinstein manifold/sector $W$, which is an $A_{\infty}$-category. 
Given a closed subset (i.e., a stop) $\Lambda \subset \partial_{\infty} W$, where $\partial_{\infty} W$ is the ideal contact boundary of $W$, we define $\cW(W,\Lambda)$ as the partial wrapped Fukaya category associated with the pair $(W, \Lambda)$. When $\Lambda=\emptyset$, we have $\cW(W,\Lambda)=\cW(W)$. 
For precise definitions, we refer to \cite{gps1}. 
The notion of the partial wrapped Fukaya category was originally introduced in \cite{syl19}.

As mentioned above, wrapped Fukaya categories of plumbings are computed in \cite{kar-lee24-2}. 
We introduce the main result in this subsection, see Theorem \ref{thm wrapped Fukaya category}. 

Before starting the introduction, let us explain the role of Theorem \ref{thm wrapped Fukaya category} in the current paper. 
First of all, in the current paper, we study {\em the category of proper modules} over the wrapped Fukaya category of plumbings (in the purpose of answering questions in symplectic topology and representation theory, as presented in Section \ref{section introduction}).
Theorem \ref{thm generation of proper modules} provides a specific representation of wrapped Fukaya categories, which is a helpful tool of our study. 
Moreover, it explains why the current article is involved in not only symplectic topology, but also representation theory, even though we are considering symplectic topological objects, i.e., {\em wrapped Fukaya categories}.
We state more details of the second point, i.e., the paper's relations to representation theory, in Section \ref{subsection Ginzburg dg algebra and wrapped Fukaya category}.

We note that rigorously, wrapped Fukaya categories are invariants of Weinstein manifolds equipped with the extra structures.
In Section \ref{subsubsection grading structure}, we discuss the extra structures of plumbing spaces. 
Then we define the wrapped Fukaya category of plumbings and we compute it in Sections \ref{subsubsection plumbing sector} and \ref{subsubsection wrapped Fukaya category}.
As mentioned before, in order to compute the categories, we need to compute those of each Weinstein sectors.
For the plumbing sector, we introduce the computation in Section \ref{subsubsection plumbing sector}.
The rest of computation will be reviewed in Section \ref{subsubsection wrapped Fukaya category}.

\subsubsection{Grading structure}
\label{subsubsection grading structure}
Note that to define/compute $\mathbb{Z}$-graded wrapped Fukaya categories of Weinstein manifolds $W$, $W$ should have $2 c_1(W) = 0 \in H^2(W; \mathbb{Z})$. 
Also, the definition of wrapped Fukaya category requires to choose $\zeta \in H^1(W; \mathbb{Z})$, called {\em grading structure}, and $b \in H^2(W;\mathbb{Z}/2)$, called {\em background class}. 
For more details, see \cite{sei08}.
Thus, before recalling wrapped Fukaya category computations in \cite{kar-lee24-2}, we review what grading structures and background classes are considered in \cite{kar-lee24-2}.
The reader can find comprehensive arguments about these information in \cite[Sections 6.1 and 6.2]{kar-lee24-2}.
\begin{rmk}
	\label{rmk grading for compact Fukaya}
	We note that the same information, i.e., background class and grading structure, are also necessarily to define/compute compact Fukaya categories. 
\end{rmk}

If our Weinstein manifold is a plumbing space $P$ of dimension $\geq 6$ corresponding to plumbing data $(Q,M,\s)$, one can easily see that 
\[H^2(P;\mathbb{Z}/2) = \bigoplus_{v \in V(Q)} H^2(M_v;\mathbb{Z}/2), \qquad H^1(P;\mathbb{Z}) = H^1(Q;\mathbb{Z}) \oplus \bigoplus_{v \in V(Q)} H^1(M_v;\mathbb{Z}). \]
For the background class $b \in H^2(P;\mathbb{Z}/2)$, \cite{kar-lee24-2} chooses the standard one, i.e., $b$ such that the summand of $b$ in each $H^2(M_v;\mathbb{Z}/2)$ is the standard background class of the cotangent bundle $T^*M_v$ (as explained in \cite{nad-zas09}).

Similarly, to choose an element
$
\zeta \in H^1(P;\mathbb{Z})
=
H^1(Q;\mathbb{Z})
\oplus
\bigoplus_{v \in V(Q)} H^1(M_v;\mathbb{Z}),
$
\cite{kar-lee24-2} takes the standard grading structure on the
$\bigoplus_{v \in V(Q)} H^1(M_v;\mathbb{Z})$ factor, which can be seen as the zero element.
The essential grading data therefore comes from the
$H^1(Q;\mathbb{Z})$ factor.

To encode the effect of $H^1(Q;\mathbb{Z})$ factor, \cite{kar-lee24-2} assigns an integer $d_e$ to each arrow $e$ of $Q$. 
We now define a fixed $\zeta \in H^1(Q;\mathbb{Z})$:
For a loop $\alpha$ in $Q$, one can find a shortest loop $\alpha_0$ homotopic to $\alpha$. 
It is easy to observe that $\alpha_0$ is a concatenation of arrows or reversed arrows of $Q$. 
For convenience, let $\alpha_0 = f_1 \dots f_k$ where $f_i$ is either an arrow or a reversed arrow of $Q$. 
Then, $\zeta \in H^1(Q;\mathbb{Z})$ is defined as 
\begin{gather}
	\label{eqn eta}
	\zeta(\alpha) = \zeta(\alpha_0) := \sum_{f_i = \text{  an arrow  }e_i} d_{e_i} - \sum_{f_i = \text{  a reversed arrow  }  e_i} d_{e_i}.
\end{gather}
It is easy to check that 
\begin{itemize}
	\item if we fix the integers $\{d_e | e \in E(Q)\}$, then $\zeta$ in \eqref{eqn eta} is well-defined, and
	\item for any $\zeta_0 \in H^1(Q;\mathbb{Z})$, one can find a collection $\{d_e | e \in E(Q)\}$ such that $\zeta$ in \eqref{eqn eta} agrees $\zeta_0$. 
\end{itemize}

\begin{rmk}
	We would like to note that for a given $\zeta \in H^1(P;\mathbb{Z})$, the choice of $d_e$ is not unique. 
	For example, one can find different choices of $d_e$ giving the same grading structure and so the same wrapped Fukaya category too. 
	See Corollary \ref{corollary grading changes}.
\end{rmk}

\subsubsection{Wrapped Fukaya category of $\Pi_n$}
\label{subsubsection plumbing sector}
We review the wrapped Fukaya category of the plumbing sector $\Pi_n$. 
We recall that the wrapped Fukaya category of $\Pi_n$ is defined to be the partially wrapped Fukaya category $\cW\left(\mathbb{C}^n, \partial_\infty \left(\mathbb{R}^n \sqcup i \mathbb{R}^n\right)\right)$. 

As known in \cite{gps2, cdrgg17}, two linking disks corresponding to the two stops in $\left(\mathbb{C}^n, \partial_\infty \left(\mathbb{R}^n \sqcup i \mathbb{R}^n\right)\right)$ generate $\cW\left(\mathbb{C}^n, \partial_\infty \left(\mathbb{R}^n \sqcup i \mathbb{R}^n\right)\right)$, because $\mathbb{C}^n$ is a subcritical Weinstein manifold. 
Thus, $\cW\left(\mathbb{C}^n, \partial_\infty \left(\mathbb{R}^n \sqcup i \mathbb{R}^n\right)\right)$ can be described by a quiver representation having two vertices and arrows between them such that two vertices correspond to two linking disks and arrows correspond to morphisms between vertices. 
The quiver is determined in \cite{kar-lee24-2}, and the following is the result. 

\begin{thm}[\cite{kar-lee24-2}]\label{thm:wfuk-plumbing-sector}
	Let $n$ be an integer greater than or equal to $3$.
	\begin{enumerate}
		\item\label{item:plumbing-3} 
		The wrapped Fukaya category of the plumbing sector $\Pi_e$ of dimension $2n$ is given, up to quasi-equivalence, by
		\[\mathcal{W}(\Pi_n)\simeq \mathrm{Tw}\, \mathcal{D}_n \,\]
		where $\mathcal{D}_n$ is the semifree dg category given as follows:
		\begin{enumerate}[label = (\roman*)]
			\item {\em Objects:} $L_s,L_t$ (each representing linking disks corresponding to the stops $\partial_\infty \mathbb{R}^n$ and $\partial_\infty i \mathbb{R}^n$, respectively).
			\item {\em Generating morphisms:}
			\[\begin{tikzcd}
				L_s\ar[r,"x", bend left]& L_t\ar[l,"y",bend left]
			\end{tikzcd}\]
			\item {\em Degrees:} $|x|=0,\quad |y|=2-n$.
			\item {\em Differentials:} $dx=dy=0$.
		\end{enumerate}	
		\item \label{item:inclusion-plumbing-3} 
		We note that there exist natural identifications $S^{n-1} \simeq \partial_\infty \mathbb{R}^n, S^{n-1} \simeq \partial_\infty i \mathbb{R}^n$.
		When equipped with appropriate grading data, they induce the dg functors
		\begin{gather*}
			\Phi_e\colon \mathcal{W}(T^*S^{n-1}) \to\mathcal{W}(\Pi_n)\\
			L \mapsto L_s, \qquad
			z \mapsto yx\\
			\Psi_e\colon \mathcal{W}(T^*S^{n-1}) \to\mathcal{W}(\Pi_n)\\
			L \mapsto L_t, \qquad
			z \mapsto xy,
		\end{gather*}
		where $\mathcal{W}(T^*S^{n-1})$ is quasi-isomorphic to the pretriangulated closure of the dg category with a single object $L$ (representing a cotangent fiber), whose endomorphism algebra is $k[z]$ with $|z|=2-n$ and $dz=0$.
	\end{enumerate}
\end{thm}

\begin{rmk}\label{rmk:plumbing-sector-disks}
	We remark that $\mathbb{R}^n$ and $i \mathbb{R}^n$ are not objects of $\cW(\Pi_n) \simeq \cW\left(\mathbb{C}^n, \partial_\infty \left(\mathbb{R}^n \sqcup i \mathbb{R}^n\right)\right)$ since they are Lagrangians touching the stops. 
	However, their small negative pushoffs by Reeb flows (in the cylindrical end of $\mathbb{C}^n$) are objects in the wrapped Fukaya category. 
	Let $D^n_s$ and $D^n_t$ denote the Lagrangians corresponding to the negative pushoffs of $\mathbb{R}^n$ and $i \mathbb{R}^n$. 
	The Yoneda modules of $D^n_s, D^n_t$, and $D^n_s\cup D^n_t$, where $D^n_s$ and $D^n_t$ are equipped with gradings $d_s, d_t \in \mathbb{Z}$ respectively, are described as
	\begin{gather*}
		\mathcal{Y}\colon\mathcal{W}(\Pi_e)\hookrightarrow\Mod\,\cD_n \quad(\simeq\Mod\,\mathcal{W}(\Pi_e))\\
		D^n_s\mapsto E_s,\quad D^n_t\mapsto E_t,\quad D^n_s\cup D^n_t\mapsto E_{st},
	\end{gather*}
	where
	\begin{alignat*}{2}
		E_s(L_s)&=k[d_s], \quad E_s(L_t)&&=0,\\
		E_t(L_s)&=0, \hspace{1cm} E_t(L_t)&&=k[d_t],\\
		E_{st}(L_s)&=k[d_s], \quad E_{st}(L_t)&&=k[d_t],
	\end{alignat*}
	and they all send $x$ and $y$ to $0$.
\end{rmk}

\subsubsection{Wrapped Fukaya category of plumbings}
\label{subsubsection wrapped Fukaya category}
The specific Weinstein sectorial coverings of plumbings have two types of Weinstein sectors in them.
One is the plumbing sector whose wrapped Fukaya category is introduced in the previous subsubsection. 
The other type is cotangent bundles whose wrapped Fukaya categories have been well-understood, as explained below.

When $M$ is a smooth manifold, \cite{abo12, gps2} proved that we have a quasi-equivalence
\[\mathcal{W}(T^*M)\simeq \mathrm{Tw}\, C_{-*}(\Omega M) ,\]
where the latter can be seen as a semifree dg category with a single object $L$ (representing a cotangent fibre) whose endomorphism algebra is the chains $C_{-*}(\Omega M)$ on the based loop space $\Omega M$. 

For $n\geq 2$, there is an inclusion
\[S^{n-1}\simeq \partial U\hookrightarrow M \setminus \{\mathrm{pt}\}, \]
where $U$ is a small disk neighborhood neighborhood of an interior point $\mathrm{pt}$ in $M$. 
This induces a dg-functor
\begin{equation}\label{eq:eta'}
	\eta\colon k[z]\simeq C_{-*}(\Omega S^{n-1}) \to C_{-*}(\Omega(M \setminus \{\mathrm{pt}\})),
\end{equation}
where $dz=0$ and $|z|=2-n$. 
Moreover, by the homotopy colimit formula in \cite{kar-lee21}, $C_{-*}(\Omega M)$ is the semifree extension of $C_{-*}(\Omega(M \setminus \{\mathrm{pt}\}))$ by a degree $(1-n)$ morphism $h\colon L\to L$ such that $dh=\eta(z)$.

Wrapped Fukaya category of the plumbing space $P$ corresponding to the plumbing data $(Q,M,\s)$ and a grading structure $\zeta$ is described in \cite{kar-lee24-2}. 
\begin{thm}[Theorem 6.5 of \cite{kar-lee24-2}]
	\label{thm wrapped Fukaya category}
	The wrapped Fukaya category $\mathcal{W}(P;d_e)$ is pretriangulated equivalent to $\mathcal{G}$, i.e., 
	\[\mathcal{W}(P;d_e) \simeq \mathrm{Tw}\,\mathcal{G},\]
	where $\mathcal{G}$ is defined as the following data:
	\begin{itemize}
		\item[(i)] {\em Objects:} $L_v$ (representing a Lagrangian cocore dual to $M_v$) for any $v \in V(Q)$.	
		\item[(ii)] {\em Generating morphisms:} There are three types of generating morphisms: 
		\begin{itemize}
			\item For any $v \in V(Q)$, $h_v\colon L_v\to L_v$. 
			\item For any $v\in V(Q)$, the generating morphisms of $ C_{-*}(\Omega_p (M_v\setminus\text{pt}))$, where \\$C_{-*}(\Omega_p (M_v\setminus\text{pt}))$ is considered as a semifree dg algebra.
			\item For any $e=v\to w\in E(Q)$, $x_e\colon L_v \to L_w, y_e\colon L_w \to L_v$.
		\end{itemize}
		\item[(iii)] {\em Degrees:} $|h_v|=1-n,\quad |x_e|=d_e,\quad |y_e|=2-n-d_e$.
		\item[(iv)] {\em Differentials:} $d x_e = d y_e = 0$, the differentials of generating morphisms of $C_{-*}(\Omega(M_v \setminus \{\mathrm{pt}\}))$ are the same as those in $C_{-*}(\Omega(M_v \setminus \{\mathrm{pt}\}))$, and
		\[
		dh_v=\eta_v(z)+\sum_{e= v\to \bullet} (-1)^{n d_e}y_ex_e + \sum_{e= \bullet \to v}(-1)^{n(n-1)/2}\s(e)\, x_ey_e
		\]
	\end{itemize}
\end{thm}

\begin{rmk}
	\label{remark specific representation}
	\mbox{}
	\begin{enumerate}
		\item We note that in the rest of the paper, for simplicity of notation, we let $\cW(P;d_e)$ denote the {\em specific representation} of the wrapped Fukaya category, given in Theorem \ref{thm wrapped Fukaya category}.
		\item For a given set of plumbing data $(Q,M,\s)$, let $P$ be the corresponding plumbing space equipped with a grading information $\{d_e\}$.  
		Let $e \in E(Q)$ start at a vertex $v$ and end at a vertex $w$. 
		Then, we have a plumbing sector $\Pi_n$ as a subset of $P$, which corresponds to the edge $e$. 
		The inclusion induces a functor from $\cW(\Pi_n)$ to $\cW(P)$. 
		Then, the induced functor sends objects $L_s$ to $L_v$, $L_t$ to $L_w[d_e]$, and morphisms $x$ to $x_e$, $y$ to $y_e$.
	\end{enumerate}
\end{rmk}

We set and highlight notations. 
\begin{notation}
	\label{notation morphisms}
	As described in Theorem \ref{thm wrapped Fukaya category}, $x_e, y_e$ denote the specific generating morphisms in $\cW:= \cW(P(Q,M,\s);d_e)$. 
	For a vertex $v \in V(Q)$, we let $\alpha_v^i$ denote the generating morphisms of $ C_{-*}(\Omega_p (M_v\setminus\text{pt}))$.
	Later, we use the overline notation for denoting the morphisms in the opposite category $\cW^{op}$.
	For example, $\overline{x}_e, \overline{y}_e, \overline{h}_v, \overline{\alpha}_v^i$ denote the generating morphisms in $\cW^{op}$ corresponding to $x_e, y_e, h_v, \alpha_v^i$, respectively.
\end{notation}

Proposition \ref{proposition direction reversing} shows that two different sets of plumbing data produce equivalent plumbing spaces. 
Similarly, as a corollary of Theorem \ref{thm wrapped Fukaya category}, one can see that these two equivalent plumbing spaces admit equivalent wrapped Fukaya categories if their grading information are related in the following way.
\begin{cor}[Proposition 6.7 of \cite{kar-lee24-2}]
	\label{corollary direction reversing}
	Let $(Q,M,\s)$ be plumbing data and $P(Q,M,\s)$ be a corresponding plumbing space. 
	Let $(Q',M',\s')$ be plumbing data such that 
	\begin{itemize}
		\item $Q$ and $Q'$ are the same quiver except the direction of one arrow $e \in E(Q)$, i.e., 
		\[V(Q)=V(Q'), \qquad E(Q) \setminus \{e_0\} = E(Q') \setminus \{e_0'\},\]
		$e$ and $e'$ are connecting the same vertices in $V(Q) = V(Q')$, but they have opposite directions,
		\item $M, M': V(Q)=V(Q') \to \mathcal{O}_n$ are the same maps,
		\item $\s(f) = \s'(f)$ for all $f \in E(Q) \setminus \{e_0\} = E(Q') \setminus \{e'_0\}$ and 
		\[\s(e_0) = \begin{cases}
			-\s'(e_0') &\text{  if $n$ is an odd integer,}\\
			\s'(e_0') &\text{  if $n$ is an even integer.}
		\end{cases}\]
	\end{itemize}
	Let $\{d_e| e \in E(Q)\}$ and $\{d'_f | f \in E(Q')\}$ be gradings assigned to the $(2n)$-dimensional plumbing spaces $P(Q,M,\s)$ and $P(Q',M',\s')$, respectively.
	If 
	\[d_e = d'_e \text{  for all  } e \in E(Q) \setminus \{e_0\} = E(Q') \setminus \{e_0'\}, \quad\text{  and  }\quad d_{e_0} = 2-n- d'_{e_0},\]
	then $\mathcal{W}(P(Q,M,\s);d_e)$ and $\mathcal{W}(P(Q',M',\s');d'_e)$ are pretriangulated equivalent to each other. 
\end{cor}

Similar to Corollary \ref{corollary direction reversing}, one can prove that the different choices of $d_e$ can give the same wrapped Fukaya category. 
\begin{cor}
	\label{corollary grading changes}
	Let $(Q,M,\s)$ be a set of plumbing data and let $\{d_e\}_{e \in E(Q)}$ define a grading structure. 
	For a fixed vertex $v \in V(Q)$ and an integer $k \in \mathbb{Z}$, let $d'_e$ be given as 
	\[d'_e = \begin{cases}
		d_e +k & \text{  if  } e = v \to \bullet, \\
		d_e - k & \text{  if  } e = \bullet \to v, \\
		d_e & \text{  otherwise}.
	\end{cases}\]
	Then, $\mathcal{W}(P(Q,M,\s);d_e) \simeq \mathcal{W}(P(Q,M,\s);d'_e)$. 
\end{cor}
\begin{proof}
	There are two different proofs. 
	The first one is to use Equation \eqref{eqn eta}.
	More precisely, Equation \eqref{eqn eta} defines two grading structures $\eta$ and $\eta'$ from $\{d_e\}$ and $\{d'_e\}$.
	It is easy to show that two grading structures $\eta$ and $\eta'$ are the same in $H^1(Q;\mathbb{Z})$.
	Thus, $\mathcal{W}(P(Q,M,\s);d_e)$ and $\mathcal{W}(P(Q,M,\s);d'_e)$ are the wrapped Fukaya category of the same plumbing spaces and the same grading structures by \cite[Theorem 6.5]{kar-lee24-2}.
	
	The second proof is to construct an equivalence from $\mathcal{W}(P(Q,M,\s);d_e)$ to $\mathcal{W}(P(Q,M,\s);d'_e)$. 
	Note that both wrapped Fukaya categories have generating sets $\{L_v\}$ and $\{L'_v\}$.
	Then, there exists a naturally defined functor sending $L_w$ to $L'_w$ for all $w \in V(Q)\setminus \{v\}$ and sending $L_v$ to $L'_v[-k]$. 
\end{proof}

\subsection{Ginzburg dg algebra and wrapped Fukaya category}
\label{subsection Ginzburg dg algebra and wrapped Fukaya category}

At the beginning of the current paper, we said that the main subjects of the article is a category that appears in both {\em symplectic topology} and {\em representation theory}. 
The category in symplectic topology is the wrapped Fukaya category of a general plumbing space, which is defined and explained in Section \ref{subsection wrapped Fukaya category of plumbing spaces}. 
In Section \ref{subsection Ginzburg dg algebra and wrapped Fukaya category}, we discuss how the category appears in representation theory. 

We first introduce the definition of Ginzburg dg algebra.
The definition is standard and there are many references in the literature.
For example, we refer the reader to \cite{gin06, Keller11}. 

Let $Q$ be a finite quiver. 
In other words, $Q$ is a directed graph with finitely many arrows.
We also assume that $Q$ is graded, i.e., for each arrow $e \in E(Q)$, it is assigned a grading.
Let $d_e$ denote the grading of $e \in E(Q)$. 

The Ginzburg dg algebra associated to $Q$ (or equivalently, a quiver $Q$ with the zero potential $W = 0$), is defined as follows: 

\begin{dfn}
	\label{dfn Ginzburg dg algebra} 
	Let $Q$ be a graded quiver, and let $n \geq 3$ be a fixed integer. 
	\begin{enumerate}
		\item Let $\widetilde{Q}$ denote a graded quiver with the same vertices as $Q$ and whose arrows are 
		\begin{itemize}
			\item $x_e = v \to w$ for all $e = v \to w \in E(Q)$, 
			\item $y_e = w \to v$ for all $e = v \to w \in E(Q)$,
			\item $h_v = v \to v$ for all $v \in V(Q)$.
		\end{itemize}
		The grading of each arrow is given as 
		\[|x_e|=d_e,\quad |y_e|=2-n-d_e,\quad |h_v|=1-n.\]
		\item The {\em Ginzburg dg algebra associated to $Q$}, denoted as $\Gamma_n(Q)$ is a differential graded algebra such that 
		\begin{itemize}
			\item whose underlying graded algebra is the graded path algebra $k \widetilde{Q}$, i.e., the $n$-th component of $\Gamma(Q,W)$ consists of elements of the form $\sum_p \lambda_p p$, where $p$ denotes a path of $\tilde{Q}$ of degree $n$,
			\item whose differential $d_{\Gamma(Q,W)}$ (or simply $d$ if there is no chance of confusion) is the unique continuous linear endomorphism satisfying the graded Leibniz rule, which is defined as follows on the arrows of $\tilde{Q}$: For all $v \in V(Q), e \in E(Q)$,
			\[d(x_e) = 0,\quad d(y_e) = 0,\quad d(h_v) = 1_v \left(\sum_{e \in E(Q)} [x_e, y_e] \right) 1_v,\]
			where $1_v$ is the lazy path at $v$.
		\end{itemize}
	\end{enumerate} 
\end{dfn}

One could say that Theorem \ref{thm wrapped Fukaya category} provides/constructs a specific category from a given set of data $(Q,M,\s,d_e)$. 
Considering specific plumbing data, the corresponding wrapped Fukaya category should be equivalent to {\em Ginzburg dg algebra} associated to a quiver $Q$. 
\begin{cor}[Corollary 6.11 of \cite{kar-lee24-2}, Corollary 4.16 of \cite{asp21}]
	\label{cor Ginzburg dg algebra}
	Let $(Q,M,\s,d_e)$ be plumbing data satisfying that
	\[M(v) = S^n \text{  for all  } v \in V(Q), \qquad \s(e) = - (-1)^{\tfrac{n(n-1)}{2}} \text{  for all  } e \in E(Q). \]
	Then, the corresponding wrapped Fukaya category $\cW(P(Q,M,\s,d_e))$ is Morita equivalent to the {\em Ginzburg dg algebra} associated to the quiver $Q$, which is originally defined in \cite{gin06, kel11}. 
\end{cor}

\begin{rmk}
	\label{rmk generalized Ginzburg dg algebra}
	Let $P$ be an arbitrary plumbing space whose plumbing pattern is $Q$. 
	Thanks to Theorem \ref{thm wrapped Fukaya category}, the wrapped Fukaya category of $P$ seems similar to the Ginzburg dg algebra over $Q$. 
	The differences are the generating morphisms of $C_{-*}\left(\Omega_p(M_v \setminus \text{pt})\right)$ on each vertex $v$ of $Q$ and $\eta_v(z)$ in $d h_v$. 
	For simplicity, we remember the similarity by saying that the wrapped Fukaya category of $P$ is a {\em generalized} Ginzburg dg algebra. 
\end{rmk}

In the rest of the current paper, we study categories of proper modules over wrapped Fukaya categories of plumbings. 
However, thanks to Corollary \ref{cor Ginzburg dg algebra}, one can replace the word ``wrapped Fukaya category'' with ``Ginzburg dg algebra''.

\section{Proper modules over wrapped Fukaya categories/Ginzburg dg algebras}
\label{section proper modules over wrapped Fukaya categories}
The main interest of the article is categories appearing in both symplectic topology and representation theory, which are wrapped Fukaya categories of plumbings in symplectic topology and Ginzburg dg algebras/categories in representation theory. 
In this section, we study {\em proper modules} over them.
Especially, we study a generation result of category of proper modules. 
More precisely, we show that a specific collection of proper modules {\em generate} all proper modules. 

We state the generation result, i.e., Theorem \ref{thm generation of proper modules} formally in Section \ref{subsection main result}, after we set notation. 
Section \ref{subsection simpler modules} proves lemmas that simplify Theorem \ref{thm generation of proper modules} and Section \ref{subsection proof of Theorem 1} proves the simplified version of Theorem \ref{thm generation of proper modules}.

We point out that in the section, we are using the term {\em wrapped Fukaya category}, but one can understand most of Section \ref{section proper modules over wrapped Fukaya categories} without expertise in symplectic topology. 

\subsection{Main result}
\label{subsection main result}
Through Section \ref{section proper modules over wrapped Fukaya categories}, we fix plumbing data $(Q,M,\s)$ and an integer $n \geq 3$. 
We also fix a field $k$ as the ground coefficient field for all categories.

We recall that plumbing data $(Q,M,\s)$ defines a plumbing space $P = P(Q,M,\s)$. 
To define Fukaya category of $P$, we need to fix extra data, called a grading structure.
To do this, we fix a collection of integers $\{d_e| e \in E(Q)\}$ encoding a grading structure. 
In other words, a quadruple $(Q,M,\s,d_e)$ is necessarily to define Fukaya category of $P$. 

In the section and the following sections, we require that a given grading structure satisfies the following condition:
\begin{dfn}
	\label{dfn non-positivcely grading condition}
	For plumbing data $(Q,M,\s)$, a grading structure $\{d_e|e \in E(Q)\}$ is {\em non-positively graded} if one could transform it by applying  Corollaries \ref{corollary direction reversing} and \ref{corollary grading changes} (as many as needed) to another set of data $(Q',M',\s',d_e')$ such that 
	\begin{gather}
		\label{eqn non-positively grading condition}
		2-n < d_e' \leq 0.
	\end{gather}
\end{dfn}

\begin{rmk}
	\label{rmk nonpositively grading condition}
	\mbox{}
	\begin{enumerate}
		\item We point that one can replace Equation \eqref{eqn non-positively grading condition} with
		\[2-n \leq d_e' \leq 0.\]
		If $e$ has $d_e' = 2-n$, one can reverse the direction of $e$ by applying Corollary \ref{corollary direction reversing} and the new grading structure satisfies Equation \eqref{eqn non-positively grading condition}. 
		Remark \ref{rmk degree 0 cycle} explains why we use Equation \eqref{eqn non-positively grading condition} instead of the above.
		\item If $(Q,M,\s,d_e)$ is plumbing data equipped with a non-positively grading structure, then the corresponding wrapped Fukaya category $\cW$ in Theorem \ref{thm wrapped Fukaya category} has generating morphisms whose degrees are non-positive. 
		Since its morphisms are non-positively graded, $\cW$ admits various nice properties. 
		For example, it implies that $\cW$ is {\em smooth} (see \cite[Proposition 1.37]{Keller-Wang23}), thus implies that the category of proper modules admits a Calabi--Yau structure. 
		\item One of the motivation (from symplectic topology) of studying proper modules is the expectation that the category of proper modules would be equivalent to the {\em compact Fukaya category}. 
		Note that compact Fukaya category is defined to be a Calabi--Yau category. 
		If the expectation holds, the category of proper modules also has to admit a Calabi--Yau structure, and (2) guarantees it.     
	\end{enumerate}
\end{rmk}

In the rest of the article, we assume that the fixed plumbing data $(Q,M,\s)$ is equipped with $\{d_e\}$ satisfying Equation \eqref{eqn non-positively grading condition}. 

To formally state Theorem \ref{thm generation of proper modules}, we set notation below.
\begin{dfn}
	\label{dfn cycle}
	A {\em cycle} of the given quiver $Q$ is an isotopy class of the image of a map
	\[\gamma : [0,k] \to Q,\]
	where $k>0$ is a natural number and $Q$ is the underlying graph of a quiver $Q$, satisfying that 
	\begin{enumerate}
		\item[(i)] $\gamma$ is injective, 
		\item[(ii)] $\gamma([i-1,i])$ is an edge of $Q$ , and
		\item[(iii)] $\gamma(0) = \gamma(k)$. 
	\end{enumerate}
\end{dfn}
We defined a cycle as an isotopy class, but for simplicity, we always identify a cycle with a specific representative $\mathrm{Im}(\gamma)$, or more simply with a map $\gamma$, such that $\gamma$ satisfies (i)--(iii) of Definition \ref{dfn cycle}. 

We point out technical points on Definition \ref{dfn cycle}:
(iii) justifies the name {\em cycle}.
But the word ``cycle'' in Definition \ref{dfn cycle} does not mean a general element of the first homology group $H_1(Q;\mathbb{Z})$ because of (i).
Note that $Q$ is a graph and thus, for a cycle $\mathrm{Im}(\gamma)$, there are no $\gamma_1, \gamma_2$ such that 
\[\mathrm{Im}(\gamma_1), \mathrm{Im}(\gamma_2) \subset \mathrm{Im}(\gamma) \text{  and  } \gamma = \gamma_1 + \gamma_2 \text{  in  } H_1(Q;\mathbb{Z}).\]

The following defines {\em vertices, edges}, and {\em degrees} of cycles. 
\begin{dfn}
	\label{dfn degree of a cycle}
	Let $\gamma: [0,k] \to Q$ be a cycle of $Q$. 
	\begin{enumerate}
		\item Let $V(\gamma)$ be a subset of $V(Q)$ defined as follows:
		\[V(\gamma) := \{\gamma(i) | i = 1, 2, \dots, k\}.\]
		\item Let $E(\gamma)$ be a subset of $E(Q)$ defined as follows: 
		\[E(\gamma) := \left\{e \in E(Q) | \text{  there exists $i$ such that  } e = \gamma([i-1,i]) \right\}.\]
		\item Let $e_i \in E(\gamma)$ satisfy that $e_i = \gamma([i-1,i])$.
		The {\em degree of $\gamma$}, denoted $|\gamma|$, is defined to be 
		\begin{gather}
		\label{eqn degree of a cycle}
			|\gamma|:= \sum_{\text{the head of  }e_i = \gamma(i)} d_{e_i} + \sum_{\text{the head of } e_i = \gamma(i-1)} \left(2-n - d_{e_i}\right). 
		\end{gather}
	\end{enumerate}
\end{dfn}

\begin{rmk}
	\label{rmk degree 0 cycle}
	For a cycle $\gamma: [0,k] \to Q$ with $E(\gamma) = \{e_i = \gamma([i-1,i]): i =1, \dots, k\}$, if $|\gamma| = 0$, the head of $e_i$ should be $\gamma(i-1)$. 
	If not, $|\gamma|$ contains $(2-n-d_{e_i})$ for some $i$, and then it should be negative since $d_{e_i}$ satisfies Equation \eqref{eqn non-positively grading condition}. 
	In this sense, we can say that a degree $0$ cycle {\em respects} the direction of the quiver $Q$. 
\end{rmk}

We give a name to the set of all degree $0$ cycles.
\begin{dfn}
	\label{dfn Cyc}
	Let $(Q,M,\s)$ be plumbing data equipped with a grading structure $\{d_e\}$ satisfying Equation \eqref{eqn non-positively grading condition}.
	\begin{enumerate}
		\item Let $\overline{\mathrm{Cyc}}$ denote the set of all degree $0$ cycle of $Q$. 
		\item Let $\mathrm{Cyc}$ denote the quotient of a collection of subsets of $\overline{\mathrm{Cyc}}$ 
		\[\left\{\{\gamma_1, \dots, \gamma_k\} \subset \overline{\mathrm{Cyc}} | \cup_{i = 1}^k \mathrm{Im}(\gamma_i) \text{  is a connected subset of  } Q\right\}\]
		by the following equivalence relation: 
		\[\{\gamma_1, \dots, \gamma_k\} \sim \{\gamma'_1, \dots, \gamma'_{k'}\} \text{  if and only if  } \cup_{i=1}^k \mathrm{Im}(\gamma_i) = \cup_{i=1}^{k'} \mathrm{Im}(\gamma'_i).\]
		\item Let $\mathbb{B}$ denote the set 
		\[\mathbb{B}:= V(Q) \sqcup \mathrm{Cyc}.\]
		\item Definition \ref{dfn degree of a cycle} (1) and (2) extend naturally to any $C \in \mathbb{B}$.
		Specifically, for $v \in V(Q) \subset \mathbb{B}$ and $C = \{\gamma_1, \dots, \gamma_m\} \in \mathrm{Cyc} \subset \mathbb{B}$, 
		\[V(v) = \{v\}, E(v) = \varnothing, V(C) = \cup_{i=1}^mV(\gamma_i), E(C) = \cup_{i=1}^mE(\gamma_i).\]
		\end{enumerate}
\end{dfn}

\begin{exa}
	We give an example of a quiver $Q$ and its cycles. 
	Our example quiver $Q$ has two vertices and four arrows as follows: 
	\[\begin{tikzcd}
		\bullet \ar[r,"b", bend left=20] \ar[r, "a", bend left =60] & \bullet \ar[l,"c",bend left=20] \ar[l, "d", bend left=60]
	\end{tikzcd}\]
	We assume that every arrow $e \in \{a, b, c, d\}$ has the same $d_e = 0$.
	Then, $\overline{\mathrm{Cyc}}, \mathrm{Cyc}$ are 
	\[\overline{\mathrm{Cyc}} = \{ca, da, cb, db\}, \mathrm{Cyc} := \Big\{\{ca\}, \{da\}, \{cb\}, \{db\}, \{cb, db\}, \{ca, da\}, \{da, db\}, \{ca, cb\}, \{ca, db\} \sim \{cb, da\}\Big\},\]
	where, for example, $ca$ denotes the cycle obtained by concatenating two arrows $c$ and $a$. 
\end{exa}

Finally, we define a specific collection of proper modules, which appears in Theorem \ref{thm generation of proper modules}.
To do that, we use the usual notation:
For a randomly given quadruple $(Q,M,\s,d_e)$, let $\cW$ denote the corresponding wrapped Fukaya category, i.e., $\cW = \cW(P(Q,M,\s);d_e)$. 
\begin{dfn}
	\label{dfn generators} 
	For every $(C,s) \in \mathbb{B} \times \mathbb{Z}$, we set $\cC_{(C,s)}$ as a collection of proper modules $M$ over $\cW$, i.e., $M \in \Prop \cW$, satisfying the following four conditions:
	\begin{enumerate}
		\item[(a)] If $v \notin V(C)$, then $M(L_v) = 0$.
		\item[(b)] If $v \in V(C)$, then $M(L_v) \in \Mod k$ is nonzero and concentrated at degree $s$. 
		\item[(c)] If $e \notin E(C)$, then $M(\overline{x}_e) = 0$. 
		\item[(d)] If $e = w \to v \in E(C)$ with $v, w \in V(C)$, then 
		\[\sum_{\overline{f} \in \hom_{\cW^{op}}^0(L_w, L_w)} M(\overline{f})\big(\mathrm{Im}M(\overline{x}_e)\big) = M(L_w).\]  
	\end{enumerate}
\end{dfn}
We note that notations used in Definition \ref{dfn generators}, such as $x_e$ or the bar notation, are defined in Notation \ref{notation morphisms}.

\begin{rmk}
	\label{rmk indecomposable}
	Since the condition (d) of Definition \ref{dfn generators} looks technical, we briefly mention the idea behind the condition in the remark. 
	\begin{enumerate}
		\item We note that Definition \ref{dfn generators} is motivated from symplectic topology, and that a module satisfying Definitions \ref{dfn generators} (a)--(d) corresponds to a Lagrangian (equipped with a local systems) with some conditions. 
		One of the geometric conditions that the Lagrangian should satisfy is the {\em connectedness}, and its algebraic counter part is the condition (d). 
		\item It would be natural to expect that the connectedness condition for Lagrangians would correspond to the indecomposable condition for proper modules.
		However, we should take the extra structure, i.e., local system on a Lagrangian, into consideration. 
		Since we do not restrict local systems to be indecomposable, then the corresponding proper modules would not be an indecomposable one. 
		For example, we take a local system that can be written as a direct sum of two non-trivial local systems, then the corresponding proper module is also written as a direct sum of two local systems. 
		See Remark \ref{rmk indecomposable modules in an example} for an example. 
		\item We also note that we use the condition (d) in the proof of Theorem \ref{thm generation of proper modules}, given in Section \ref{subsection proof of Theorem 1}. 
		See Equation \eqref{eqn induction step} and arguments following it. 
	\end{enumerate}
	
\end{rmk}

\begin{thm}
	\label{thm generation of proper modules}
	Let $(Q,M,\s)$ be plumbing data and $\{d_e\}$ satisfy Equation \eqref{eqn non-positively grading condition}.  
	Then, every proper module over $\cW(P(Q,M,\s);d_e)$ is generated by taking cones of modules in 
	\[\bigcup_{(C,s) \in \mathbb{B}\times \mathbb{Z}} \cC_{(C,s)}.\]
	Or equivalently, $\bigcup_{C \in \mathbb{B}} \cC_{(C,0)}$ generates all proper modules by taking cones and shifts.  
\end{thm}
In the rest of the paper, we call modules in $\bigcup_{(C,s) \in \mathbb{B}\times \mathbb{Z}} \cC_{(C,s)}$ {\em the generators}.

Thanks to Section \ref{subsection Ginzburg dg algebra and wrapped Fukaya category}, Theorem \ref{thm generation of proper modules} also proves a generation result for Ginzburg dg algebras. 
\begin{cor}
	\label{cor generation of ginzburg dg algebra}
	Let $Q$ be a graded quiver with a grading $\{d_e\}_{e \in E(Q)}$ satisfying Equation \eqref{eqn non-positively grading condition}. 
	For $n \geq 3$, let $\Gamma_n(Q)$ denote the $n$-Calabi--Yau graded Ginzburg dg algebra over $Q$. 
	Every proper module over $\Gamma_n(Q)$ is generated by 
	\[\bigcup_{(C,s) \in \mathbb{B}\times \mathbb{Z}} \cC_{(C,s)}.\]
\end{cor}

\begin{rmk}
	Before moving onto the next subsection, we would like to note that symplectic topology hinted Definition \ref{dfn generators}, even though Definition \ref{dfn generators} and Theorem \ref{thm generation of proper modules} can be written without using symplectic topology (Theorem \ref{thm generation of proper modules} contains a word ``wrapped Fukaya category'' but the word just means specific categories, which Theorem \ref{thm wrapped Fukaya category} describes). 
	We expected that a collection of compact Lagrangians satisfying some {\em geometric} conditions would generate the compact Fukaya category, and also expect that the compact Fukaya category is equivalent to the category of proper modules over wrapped Fukaya categories of plumbings (or Ginzburg dg algebras).
	In other words, compact Lagrangians with the geometric conditions would generate all proper modules.
	We translated the geometric conditions into {\em algebraic} conditions of proper modules, and the results are conditions (a)--(d) of Definition \ref{dfn generators}. 
\end{rmk}

We prove Theorem \ref{thm generation of proper modules} in Sections \ref{subsection simpler modules} and \ref{subsection proof of Theorem 1}.
Our strategy for the proof is to show that $M \in \Prop \cW$ can be generated by {\em simpler} modules. 
Here, the word ``simpler modules'' means modules satisfying some specific conditions.
Lemmas \ref{lem extra condition 1}, \ref{lem extra condition 2}, and \ref{lem extra condition 3} in Section \ref{subsection simpler modules} will prove that any module $M$ is generated from those simpler modules, thus it will be enough to show that the generators in Theorem \ref{thm generation of proper modules} generate the simpler modules.
It is the main content of Section \ref{subsection proof of Theorem 1}.

\subsection{Simpler modules} 
\label{subsection simpler modules}
As usual, $(Q,M,\s,d_e)$ denotes plumbing data equipped with a non-positively graded $\{d_e\}$. 
Also as we did before, we use $\cW := \cW(P(Q,M,\s);d_e)$ to indicate the corresponding wrapped Fukaya category. 
The purpose of the subsection is to show that every $M \in \Prop \cW$ is an iterated mapping cone of modules with extra conditions --
Lemmas \ref{lem extra condition 1}, \ref{lem extra condition 2}, and \ref{lem extra condition 3} prove it. 

To state the first extra condition and also Lemma \ref{lem extra condition 1}, we introduce the following definition: 
\begin{dfn}
	\label{dfn concentrated at} 
	We say that $M \in \Prop \cW$ is {\em concentrated at degree $s$} if for every $v \in V(Q)$, $M(L_v)$ is concentrated at degree $s$, i.e., $M(L_v)^t = 0$ if $t \neq s$. 
\end{dfn}

\begin{lem}
	\label{lem extra condition 1}
	\mbox{}
	\begin{enumerate}
		\item Every $M \in \Prop \cW$ admits a tower 
		\begin{equation}
			\label{eqn extra condition 1}
			\begin{tikzcd}
				0 \arrow[r] & \ast \arrow[r] \arrow[d] & \dots  \arrow[r] & M \arrow[d] \\
				& M_1 \arrow[lu, dashed] & \dots           & M_k \arrow[lu, dashed]
			\end{tikzcd},
		\end{equation}
		such that 
		\begin{itemize}
			\item $M_i$ is concentrated at degree $s_i$ and 
			\item $s_1 < s_2 < \dots < s_k$. 
		\end{itemize}
		\item Thus, it is enough to prove Theorem \ref{thm generation of proper modules} for modules $M$ concentrated at degree $0$. 
	\end{enumerate}
\end{lem}
\begin{proof}
	The second part (2) is trivially true if (1) holds.
	The following is the proof of (1).
	
	For a given $M$, let us define the following set of integers
	\[\supp(M) := \{i | \text{  there exists $v \in V(Q)$ such that  } M(L_v)^i \neq 0\}.\]
	We note that we are assuming $M$ is a functor from $\cW^{op}$ to $\HMod k$, as mentioned in Remark \ref{rmk HMod}.
	Moreover, $M$ is proper, thus $\supp(M)$ is a finite subset of $\mathbb{Z}$. 
	We prove Lemma \ref{lem extra condition 1} by an induction on the cardinal number $|\supp(M)|$. 
	
	Before starting the proof, we note that if there exists a tower given in Equation \eqref{eqn extra condition 1}, then 
	\[\supp(M) \subset \cup_{i=1}^k \supp(M_i) = \{s_1, \dots, s_k\}.\]
	In the proof, we prove also the equality of both sides of the above equation, i.e., 
	\begin{gather}
		\label{eqn equality} 
		\supp(M) = \cup_{i=1}^k \supp(M_i) = \{s_1, \dots, s_k\}
	\end{gather} 
	
	\noindent{\bf Base step:} 
	The base step is the case of $|\supp(M)| =1$, which is equivalent to say that $M$ is concentrated at one specific degree. 
	Thus, Lemma \ref{lem extra condition 1} (1) and Equation \eqref{eqn equality} hold trivially.
	
	\noindent{\bf Induction step:} 
	To prove the induction step, let us assume that $|\supp(M)| = k$ and that Lemma \ref{lem extra condition 1} (1) and Equation \eqref{eqn equality} hold for modules $N$ with $|\supp(N)| = k-1$. 
	Under the assumption, let 
	\[\supp(M) = \{s_1, \dots, s_k\},\]
	such that $s_1 < s_2 <\dots < s_k$. 
	
	Our strategy is to construct $M_1 \in \Prop \cW$ such that 
	\begin{enumerate}
		\item[(I)] $M_1$ is concentrated at degree $s_1$, 
		\item[(II)] there exists a degree $0$ morphisms $T: M_1 \to M$ in $\Prop \cW$ such that
		\[\supp\left(\Cone(T)\right) = \{s_2, \dots, s_k\}.\]
	\end{enumerate}
	
	If there exists such $M_1$ and $T$, we have 
	\begin{equation*}
		\begin{tikzcd}
			& & M_1 \arrow[r, "T"] & M \arrow[d]\\
			0 \arrow[r] & \ast \arrow[r] \arrow[d] & \dots  \arrow[r] & \ast = \cone(T) \arrow[d] \arrow[lu, dashed] \\
			& M_2 \arrow[lu, dashed] & \dots           & M_k \arrow[lu, dashed]
		\end{tikzcd},	
	\end{equation*}
	Here, the induction hypothesis guarantees the existence of the tower for $\Cone(T)$, so $M_i$ is concentrated at degree $s_i$ for all $i$. 
	
	Applying octahedral axiom iteratively, one has 
	\begin{equation*}
		\begin{tikzcd}
			0 \arrow[r] & M_1 \arrow[r] \arrow[d] & \ast \arrow[r] \arrow[d] & \dots \arrow[r] & \ast \arrow[r] \arrow[d] & M \arrow[d] \\
			& M_1 \arrow[lu, dashed] & M_2 \arrow[lu, dashed] & \dots & M_{k-1} \arrow[lu, dashed] & M_k \arrow[lu, dashed]
		\end{tikzcd}.
	\end{equation*} 
	It proves the induction step, so it is enough to construct $M_1$ and $T$ satisfying (I) and (II). 
	
	By using a specific representation of $\cW$, which is described in Theorem \ref{thm wrapped Fukaya category}, constructing $M_1$ is equivalent to give the following data:
	\begin{itemize}
		\item cochain complex (with zero differentials) $M_1(L_v)$ for all $v \in V(Q)$ in object level, and 
		\item a collection of morphisms in $\Mod k, \{M_1(\overline{f}) | \overline{f} \text{  is a generating morphism in  } \cW^{op}\}$, 
	\end{itemize}
	satisfying the definition of dg-functors. 
	We note that the overline notation $\overline{f}$ is introduced in Notation \ref{notation morphisms}.
	
	\noindent{\bf Construction of $M_1$ (object level):}
	For every $v \in V(Q)$, we set 
	\[M_1(L_v)^t = \begin{cases}
		M(L_v)^{s_1} & \text{  if  } t = s_1, \\
		0 & \text{  otherwise}.
	\end{cases}\]
		
	\noindent{\bf Construction of $M_1$ (morphism level):}
	First, we emphasize that $M_1(L_v)^t \subset M(L_v)^t$ for all $t \in \mathbb{Z}$ and for all $v \in V(Q)$. 
	Thus, for $\overline{f} \in \hom_{\cW^{op}}(L_v,L_w)$, we set $M_1(\overline{f})$ as 
	\[M_1(\overline{f})^t := M(\overline{f})^t|_{M_1(L_v)^t} : M_1(L_v)^t \to M(L_w)^{t + |\overline{f}|} \text{  for all  } t \in \mathbb{Z}.\]
	Below, we show that  
	\begin{gather}
		\label{eqn target of a morphism}
		M(\overline{f})^t\left(M_1(L_v)^t\right) \subset M_1(L_w)^{t+|\overline{f}|}.
	\end{gather}
	If so, one can see $M_1(\overline{f})$ as a linear map from $M_1(L_v)^t$ to $M_1(L_w)^{t+|\overline{f}|}$, and it completes the morphism level construction.
	
	If $t \neq s_1$, then the domain is the zero $k$-module, so Equation \eqref{eqn target of a morphism} holds.
	When $t = s_1$, if $|\overline{f}| <0$, the target $M(L_w)^{s_1 + |\overline{f}|}$ is also the zero $k$-module, since $s_1$ is the minimum of $\supp(M)$. 
	Thus, 
	\[M(\overline{f})^t\left(M_1(L_v)^t\right) = 0 = M_1(L_w)^{t+|\overline{f}|}.\]
	Finally, if $t= s_1$ and $|\overline{f}|=0$ (note that $|\overline{f}| \leq 0$ because of the non-positively graded condition), 
	\[M(\overline{f})^t\left(M_1(L_v)^t\right) \subset  M(L_w)^{t + |\overline{f}|} = M(L_w)^{s_1} = M_1(L_w)^{s_1}.\] 
	
	\noindent{\bf Construction of $M_1$ (dg-functor):}
	We constructed $M_1$ in object and morphism levels, and now we need to show that $M_1$ is a {\em dg-functor}. 
	In other words, we need to show the following:
	\begin{enumerate}
		\item[(i)] $M_1(\overline{f}_2 \circ \overline{f}_1) = M_1(\overline{f}_2) \circ M_1(\overline{f}_1)$ for any composable morphisms $\overline{f}_2$ and $\overline{f}_1$ in $\mathcal{W}^{op}$. 
		\item[(ii)] $M_1 \circ d_{\mathcal{W}^{op}} = d_{\mathrm{HMod}k} \circ M_1$.
		Especially, since $d_{\mathrm{HMod}k} =0$, it is enough to show that $M_1\left(d_{\cW^{op}} \overline{f}\right) = 0$ for any morphism $\overline{f}$ in $\mathcal{W}^{op}$. 
	\end{enumerate}
	
	To prove (i), we first recall that if $|\overline{f}|< 0$, then $M_1(\overline{f})$ should be the zero map. 
	Thus, we need only to check it for the case of $|\overline{f}_1| = |\overline{f}_2|= 0$.
		
	For $\overline{f} \in \hom_{\cW^{op}}^0(L_v, L_w)$, $M_1(\overline{f})$ is given as a collection of linear maps $\{M_1(\overline{f})^t : M_1(L_v)^t \to M_1(L_w)^t\}_{t \in \mathbb{Z}}$ such that 
	\[M_1(\overline{f})^t = \begin{cases}
			 M(\overline{f})^{s_1} & \text{  if  } t = s_1, \\
			 0 & \text{  otherwise}.
		\end{cases}
	\]
 	Thus, it is easy to observe that $M_1(\overline{f}_2 \circ \overline{f}_1)^t = M_1(\overline{f}_2)^t \circ M_1(\overline{f}_1)^t$, for all $t \in \mathbb{Z}$ and $\overline{f}_i$ with $|\overline{f}|_i =0$. 
	
	Now, we prove (ii), i.e., $M_1\left(d_{\cW^{op}} \overline{f}\right) = 0$ for all $\overline{f}$. 
	If $d_{\cW^{op}} \overline{f}$ has degree other than $0$, it naturally holds, since $M_1$ is nonzero only for degree $0$ morphisms. 
	If $d_{\cW^{op}} \overline{f}$ has degree $0$,  $M_1\left(d_{\cW^{op}} \overline{f}\right)^{s_1} = M\left(d_{\cW^{op}} \overline{f}\right)^{s_1} = 0$ and $M_1\left(d_{\cW^{op}} \overline{f}\right)^t = 0$ for all $t \neq s_1$ by the construction of $M_1$.
	Thus, (ii) holds. 
	
	So far, we constructed $M_1$ satisfying (I). 
	Below, we construct $T: M_1 \to M$ satisfying (II).  
	
	\noindent{\bf Construction of $T: M_1 \to M$:}	
	We constructed $M_1$ from a given proper module $M$.
	Now, see $M_1$ and $M$ as $A_\infty$-functors from $\cW^{op}$ to $\Mod k$, we construct a natural transformation $T: M_1 \to M$ of degree $0$. 
	For the preliminaries on natural transformations between $A_\infty$-functors, we refer the reader to the famous textbook \cite{sei08}.
	
	We first need to construct a pre-natural transformation $T = \{T^d\}_{d \geq 0}$, i.e., $T^0 \in \hom_{\Mod k}\left(M_1(L_v),M(L_v)\right)$ for any object $L_v$ of $\cW^{op}$ and a collection of multi-linear maps for all integers $d \geq 1$ 
	\[T^d: \hom_{\cW^{op}}(X_{d-1},X_d) \otimes \dots \otimes \hom_{\cW^{op}}(X_0,X_1) \to \hom_{\Mod k}(M_1(X_0), M(X_d))[-d].\]  
	
	To define $T^0$, we recall that $M_1(L_v) \subset M(L_v)$ for all $v \in V(Q)$. 
	Thus, there exists a natural embedding of $M_1(L_v)$ into $M(L_v)$. 
	We set $T^0(L_v)$ as the natural embedding.  
	
	For all $d \geq 1$, we have to set $T^d$ as the zero map.
	We note that it is the only possible choice because of the degree reason.  
	 
	The next step is to show that the constructed $T$ is a natural transformation, or in other words, that $T$ is a closed morphism in $\Prop \cW$. 
	Note that the differential of $A_\infty$-natural transformation $T$ is involved with $A_\infty$-structures of $\cW^{op}$ or $\Mod k$. 
	Since these $A_\infty$-structures are simple (note that those categories are dg categories and have the trivial higher structures), the computation of the differential of $T$ is feasible, although it is hard to compute in usual. 
	See \cite[Chatper (1d)]{sei08} for details.  
	
	After the simple computation, the differential of $T$ is zero if and only if 
	\begin{gather}
		\label{eqn natural transformation}
		\mu_{\Mod k}^2\left(M(\overline{f}),T^0\right) + (-1)^{-(|\overline{f}|-1)}\mu_{\Mod k}^2\left(T^0, M_1(\overline{f})\right) = 0,
	\end{gather}
	for every homogeneous morphism $\overline{f}$ in $\cW^{op}$. 
	We note that $\mu_{\Mod k}^2$ is the composition of morphisms in $\Mod k$.
	Thus, we simply say $\mu_{\Mod k}^2(M(\overline{f}),T^0) = M(\overline{f}) \circ T^0$ and $\mu_{\Mod k}^2(T^0, M_1(\overline{f})) = T^0 \circ M_1(\overline{f})$. 
	
	To prove Equation \eqref{eqn natural transformation}, we compare $\left(M(\overline{f})\circ T^0\right)^t$ and $\left(T^0 \circ M_1(\overline{f})\right)^t$ for each $t \in \mathbb{Z}$.  
	If $t \neq s_1$, then the domain is $M_1(L_v)^t = 0$, so Equation \eqref{eqn natural transformation} holds. 
	If $t = s_1$ and $|\overline{f}| <0$, then $\left(M(\overline{f})\circ T^0\right)^t$ and $\left(T^0 \circ M_1(\overline{f})\right)^t$ are also both zero maps, and Equation \eqref{eqn natural transformation} holds. 
	Finally, if $t = s_1$ and $|\overline{f}|=0$, then $T^0$ is the identity map. 
	Thus, Equation \eqref{eqn natural transformation} holds.
	
	\noindent{\bf $\mathbf{Cone(T)}$:}
	By the definition of mapping cones in $\Mod k$, 
	$\Cone(T)^{s_1} =0$ and $\Cone(T)^t = M(L_v)^t$ for all $t \neq s_1$. 
	Thus, $\supp\left(\Cone(T)\right) = \supp(M) \setminus \{s_1\}$, i.e., the constructed $T$ satisfies (II). 
\end{proof}

Lemma \ref{lem extra condition 1} divides a random proper module $M$ into simpler modules $M_1, \dots, M_k$ in Equation \eqref{eqn extra condition 1}, each of which is concentrated at one degree.
Note that each $M_i$ is the degree $s_i$ part of $M$ and the {\em order} on $\supp(M)$ helps us to divide $M$ into simpler pieces. 

Similarly, when $M$ is concentrated at a degree, Lemma \ref{lem extra condition 2} claims that one can divide $M$ into simpler modules. 
Similar to the proof of Lemma \ref{lem extra condition 1}, a {\em partial order} on $V(Q)$ helps us to to divide $M$ into simpler pieces. 
We first introduce the partial order on $V(Q)$. 
\begin{dfn}
	\label{dfn equivalence relation 1}
	\mbox{}
	\begin{enumerate}
		\item Let $v, w \in V(Q)$. 
		We define an equivalence relation $v \sim w$ as follows: 
		\[v \sim w \text{  if and only if  } \exists C \in \mathbb{B} \text{  such that  } v, w \in V(C).\]
		Note that $\mathbb{B}$ and $V(C)$ are defined in Definition \ref{dfn Cyc}. 
		Or equivalently, $v \sim w$ if and only if there exist nonzero degree $0$ morphisms $\overline{f} \in \hom_{\cW^{op}}^0(L_v,L_w)$ and $\overline{g} \in \hom_{\cW^{op}}^0(L_w,L_v)$.
		\item Since $\sim$ is an equivalence relation on $V(Q)$, it induces a partition of $V(Q)$. 
		Let $I$ be the index set of the partition, i.e., 
		\[V(Q) = \bigsqcup_{s \in I} V_s(Q),\]
		such that $v \sim w$ if and only if there exists a unique $s \in I$ satisfying $v, w \in V_s(Q)$. 
	\end{enumerate}
\end{dfn} 

Lemma \ref{lem partial order 1} says that the equivalence relation induces a partial order with a nice condition.
\begin{lem}
	\label{lem partial order 1} 
	One can give a partial order $\preceq$ on the index set $I$ so that the following holds: 
	Let $v \in V_s(Q)$ and $w \in V_t(Q)$. 
	If $\hom_{\cW^{op}}^0(L_w,L_v) \neq 0$, then $s \preceq t$. 
\end{lem}
\begin{proof}
	It is trivial from Definition \ref{dfn equivalence relation 1} and from that every generating morphism of $\cW$ is non-positively graded. 
\end{proof}

\begin{rmk}
	\label{rmk partial order}
	We note that the partial order given in Lemma \ref{lem partial order 1} will be used in Section \ref{sec:stab} for constructing a stability condition on $\Prop\cW$.
\end{rmk}

\begin{lem}
	\label{lem extra condition 2}
	Let us assume that $M \in \Prop \cW$ is a module concentrated at degree $0$. 
	\begin{enumerate}
		\item $M$ admits a tower 
		\begin{equation}
			\label{eqn extra condition 2}
			\begin{tikzcd}
				 0 \arrow[r] & \ast \arrow[r] \arrow[d] & \dots  \arrow[r] & M \arrow[d] \\
				& M_1 \arrow[lu, dashed] & \dots           & M_k \arrow[lu, dashed]
			\end{tikzcd},
		\end{equation}
		such that 
		\begin{itemize}
			\item there exists an {\em injective} map $f: \{1, \dots, k\} \to I$ so that $M_i(L_v) = 0$ for every $v \notin V_{f(i)}(Q)$, 
			\item if $f(i) \prec f(j)$, i.e., $f(i) \preceq f(j)$ and $f(i)\neq f(j)$, then $i < j$.  
		\end{itemize}
		\item Thus, it is enough to prove Theorem \ref{thm generation of proper modules} for modules $M$ having a unique $i_M \in I$ such that $M(L_v) = 0$ for every $v \notin V_{i_M}(Q)$.
		Or equivalently, it is enough to prove Theorem \ref{thm generation of proper modules} under the assumption that for all $v, w \in V(Q)$, $v \sim w$. 
	\end{enumerate}
\end{lem}
\begin{proof}
	With slight modifications, the proof of Lemma \ref{lem extra condition 1} proves Lemma \ref{lem extra condition 2}. 
	The modifications we need include the following:
	\begin{itemize}
		\item In the proof of Lemma \ref{lem extra condition 1}, we used $\supp(M)$. 
		Here, we use 
		\[\supp_I(M):= \{ s \in I | \text{  there exists  } v \in V_s(Q) \text{ such that  } M(L_v) \neq 0 \}.\]
		\item For the induction step in the proof of Lemma \ref{lem extra condition 1}, we fix the minimal element $s_1$ of $\supp(M)$. 
		Here, we fix a element $s \in \supp_I(M)$ such that there is no $t \in \supp_I(M)$ satisfying $t \prec s$.
		Note that such a minimal element is not unique and one can choose a minimal element. 
		Then, $f(1) = s$. 
		\item The first half of (2) is trivially true from (1). 
		For the second half, one can consider a sub-quiver $Q' \subset Q$ such that $V(Q') = V_{i_M}(Q)$. 
	\end{itemize} 
\end{proof}

Lemma \ref{lem extra condition 2} is saying that one can generate a given proper module $M$ from a simpler module. 
The key idea is to utilize a partial order on $V(Q)$ given in Definition \ref{dfn equivalence relation 1} and Lemma \ref{lem partial order 1}. 
We note that the given quiver $Q$ determines the equivalence relation in Definition \ref{dfn equivalence relation 1} and we do {\em not} need to use a specific property of a given module $M$, even though we discuss about the way of splitting $M$ into simpler modules. 

Below, we introduce a {\em stronger} equivalence relation on $V(Q)$ that is dependent on the given module $M$. 
The stronger equivalence relation implies Lemma \ref{lem extra condition 3}, which is a stronger version of Lemma \ref{lem extra condition 2}.

\begin{dfn}
	\label{dfn equivalence relation 2}
	Let $M \in \Prop \cW$ be a module concentrated at degree $0$. 
	\begin{enumerate}
		\item Let $v, w \in V(Q)$. 
		We say that $v \sim_M w$ if and only if there exist $\overline{f} \in \hom_{\cW^{op}}^0(L_v,L_w)$ and $\overline{g} \in \hom_{\cW^{op}}^0(L_w,L_v)$ such that 
		\[M(\overline{f}), M(\overline{g}) \neq 0.\]
		\item Since $\sim_M$ is an equivalence relation on $V(Q)$, it induces a partition of $V(Q)$. 
		Let $I_M$ be the index set of the partition, i.e., 
		\[V(Q) = \bigsqcup_{s \in I_M} V^M_s(Q),\]
		such that $v \sim_M w$ if and only if there exists a unique $s \in I_M$ satisfying $v, w \in V^M_s(Q)$. 
	\end{enumerate}
\end{dfn}

\begin{lem}
	\label{lem extra condition 3} 
	Let us assume that $M \in \Prop \cW$ is concentrated at degree $0$. 
	\begin{enumerate}
		\item One can give a partial order $\preceq_M$ on the index set $I_M$ so that the following holds: 
		Let $v \in V^M_s(Q)$ and $w \in V^M_t(Q)$. 
		If $\hom_{\cW^{op}}^0(L_w,L_v) \neq 0$, then $s \preceq_M t$. 
		\item The given module $M$ admits a tower 
		\begin{equation}
			\label{eqn extra condition 3}
			\begin{tikzcd}
				0 \arrow[r] & \ast \arrow[r] \arrow[d] & \dots  \arrow[r] & M \arrow[d] \\
				& M_1 \arrow[lu, dashed] & \dots           & M_k \arrow[lu, dashed]
			\end{tikzcd},
		\end{equation}
		such that 
		\begin{itemize}
			\item there exists an {\em injective} map $f: \{1, \dots, k\} \to I_M$ so that $M_i(L_v) = 0$ for every $v \notin V^M_{f(i)}(Q)$, 
			\item if $f(i) \prec_M f(j)$, i.e., $f(i) \preceq_M f(j)$ and $f(i)\neq f(j)$, then $i < j$.  
		\end{itemize}
		\item Thus, it is enough to prove Theorem \ref{thm generation of proper modules} for modules $M$ such that for all $v, w \in V(Q)$, $v \sim_M w$. 
	\end{enumerate}
\end{lem}
\begin{proof}
	The proofs of Lemmas \ref{lem partial order 1} and \ref{lem extra condition 2} prove this lemma after simple modifications. 
\end{proof}

\begin{rmk}
	\label{rmk two partial orders}
	By Definitions \ref{dfn equivalence relation 1} and \ref{dfn equivalence relation 2}, it is easy to see that if $v \sim_M w$ then $v \sim w$. 
	In this sense, $\sim_M$ is {\em stronger} than $\sim$. 
	Similarly, Lemma \ref{lem extra condition 3} is stronger than Lemmas \ref{lem extra condition 2}, and in order to prove Theorem \ref{thm generation of proper modules}, we only need Lemmas \ref{lem extra condition 1} and \ref{lem extra condition 3} and do not need Lemma \ref{lem extra condition 2}. 
	
	Thus, we could remove Definition \ref{dfn equivalence relation 1}, Lemmas \ref{lem partial order 1} and \ref{lem extra condition 2}. 
	However, we decided to contain them because of the following difference between $\sim$ and $\sim_M$:
	In order to define $\sim_M$, one needs to fix a specific module $M$, differently from $\sim$. 
	It means that $\sim_M$ could be more useful to study a property of the specific module $M$, but in order to study the whole category $\Prop \cW$, $\sim_M$ might not be helpful. 
	However, $\sim$ is defined from the category $\cW$ without using other extra information.
	Thus, it is possible to help us to understand the structure of $\Prop \cW$. 
	One example is Section \ref{sec:stab}: We used $\sim$ in order to construct a stability condition on $\Prop \cW$. 
\end{rmk}

\subsection{Proof of Theorem \ref{thm generation of proper modules}}
\label{subsection proof of Theorem 1}
In Section \ref{subsection proof of Theorem 1}, we prove Theorem \ref{thm generation of proper modules}, which is the main result of Section \ref{section proper modules over wrapped Fukaya categories}.

Let us recall the setting. 
A quadruple $(Q,M,\s,d_e)$ denotes a given set of plumbing data equipped with non-positively graded $\{d_e\}$, or more precisely, satisfying Equation \eqref{eqn non-positively grading condition}. 
As usual, $\cW$ denotes the corresponding wrapped Fukaya category $\cW(P(Q,M,\s);d_e)$ with an explicit quiver representation given in Theorem \ref{thm wrapped Fukaya category}. 
Thanks to Lemmas \ref{lem extra condition 1} and \ref{lem extra condition 3}, we assume that $M \in \Prop \cW$ satisfies the following: 
\begin{itemize}
	\item $M$ is concentrated at degree $0$, and
	\item for all $v, w \in V(Q)$, $v \sim_M w$. 
\end{itemize}

The first step of the proof is to assign $M$ an element of $\mathbb{B}$.
\begin{lem}
	\label{lem assigned cycle} 
	For the given $M$, let
	\[E(M):= \{e \in E(Q)| M(\overline{x}_e) \neq 0\}.\]
	Then, either 
	\begin{itemize}
		\item $E(M) = \varnothing$ or
		\item there exists $C_M \in \mathrm{Cyc}$ such that $E(C_M) = E(M)$. 
	\end{itemize}
	Moreover, if $E(M) = \varnothing$, then $M \in \cC_{(v,0)}$ for a vertex $v \in V(Q)$. 
\end{lem}
\begin{proof}
	Let us assume that $E(M) = \varnothing$. 
	Then, it is easy to observe that $|V(Q)| =1$.
	In other words, $E(M)=\varnothing$ means that $v \sim_M w$ if and only if $v =w$. 
	Since we are assuming that $v \sim_M w$ for all $v, w \in V(Q)$, $|V(Q)|=1$. 
	
	Let $v$ denote the unique vertex of $Q$. 
	Then, $v \in V(Q) \subset \mathbb{B}$. 
	To prove the ``moreover'' part, we need to check that $M$ satisfies the conditions (a)--(d) of Definition \ref{dfn generators}. 
	The conditions (a), (b), and (d) hold trivially and (c) is equivalent to that $E(M) = \varnothing$. 
	
	Now, we consider the case of $E(M) \neq \varnothing$. 
	We want to show that there exists $C_M \in \mathrm{Cyc}$ such that $E(M) = E(C_M)$. 
	Or equivalently, there exist $\gamma_1, \dots, \gamma_m \in \overline{\mathrm{Cyc}}$ such that 
	\begin{itemize}
		\item[(i)] $E(M) = \cup_i E(\gamma_i)$, and 
		\item[(ii)] $\cup_i E(\gamma_i)$ is connected as a subset of the base graph of $Q$. (It means that $\{\gamma_1, \dots, \gamma_m\} \in \mathrm{Cyc}$.)
	\end{itemize} 
	
	If (i) holds, then so does (ii): Since $v \sim_M w$ for all $v, w \in V(Q)$, $E(M)$ is connected. 
	Thus, it is enough to prove (i).
	
	Let $e = w \to v \in E(M)$.
	Then, since $v \sim_M w$, there exists $\overline{f} \in \hom_{\cW^{op}}(L_w, L_v)$ such that $M(\overline{f}) \neq 0$. 
	Note that $\overline{f}$ should be written uniquely as a product of generating morphisms. 
	Let us choose $\overline{f}$ having shortest product among all $\overline{f} \in \hom_{\cW^{op}}(L_w, L_v)$ such that $M(\overline{f}) \neq 0$.
	Then, we have a collection of arrows $\{e_1, \dots, e_m\}$ such that $\overline{x}_{e_i}$ appears in $\overline{f}$, and moreover, there exists $\gamma_e \in \overline{\mathrm{Cyc}}$ such that $\{e, e_1, \dots, e_m\} = E(\gamma_e)$. 
	We point out that $E(\gamma_e) \subset E(M)$. 
	It proves (i) and we can choose $C_M := \{ \gamma_e | e \in E(M)\}$.
\end{proof}

Lemma \ref{lem assigned cycle} proves Theorem \ref{thm generation of proper modules} for $M$ satisfying that $E(M) = \varnothing$. 
Thus, we assume that $E(M) \neq \varnothing$ in the rest of Section \ref{subsection proof of Theorem 1}. 
Under the assumption, we prove Theorem \ref{thm generation of proper modules} by an induction on
\[\sum_{v \in V(Q)} \dim_k M(L_v).\]
We recall that $k$ is a fixed coefficient field, so $\dim_k$ is well-defined.

The base step is the case of $\sum_{v \in V(Q)} \dim_k M(L_v) =1$. 
If so, we have $|V(Q)|=1$ from the assumptions. 
Since $E(M) \neq \varnothing$, $C_M$ consists of loops on the unique vertex. 
Moreover, if $e \in E(M)$, then $M(\overline{x}_e)$ is the nonzero linear map between $M(L_v)$ of dimension $1$. 
It implies that $M(\overline{x}_e)$ is an isomorphism. 
Thus, Theorem \ref{thm generation of proper modules} holds for the base step. 

Now, we prove the induction step, so let us assume the induction hypothesis: Let $N$ be a module such that $\sum_{v \in V(Q)} \dim_k N(L_v) < \ell$.
Then, $N$ is generated by the given generators in Theorem \ref{thm generation of proper modules}. 

First of all, if our given $M$ satisfies that $M \in \cC_{(C_M,0)}$, then there is nothing to prove. 
Thus, let us assume that $M \notin \cC_{(C_M,0)}$.
In other words, $M$ does not satisfy at least one of Definition \ref{dfn generators} (a)--(d).

But, the assumptions on $M$ guarantees that (a)--(c) hold for $M$ and $(C_M,0)$. 
Thus, (d) has not to hold.  
Equivalently, there exists an arrow $e = s \to t \in E(M)$ such that 
\begin{gather}
	\label{eqn induction step}
	\sum_{\overline{f} \in \hom_{\cW^{op}}^0(L_s,L_s)} M(\overline{f}) \big(\mathrm{Im} M(\overline{x}_e)\big) \subsetneq M(L_s).
\end{gather}

Our strategy is to construct another module $M_e \in \Prop \cW$ and a natural transformation $T: M_e \to M$, similar to the proofs of Lemmas \ref{lem extra condition 1}, \ref{lem extra condition 2} and \ref{lem extra condition 3}.

\noindent{\bf Construction of $M_e$ (object level):} 
For each $v \in V(Q)$, we set 
\[M_e(L_v) := \sum_{\overline{g} \in \hom_{\cW^{op}}^0(L_s,L_v)} M(\overline{g})\big(\mathrm{Im} M(\overline{x}_e)\big).\]

\noindent{\bf Construction of $M_e$ (morphism level):} 
Note that $M_e(L_v) \subset M(L_v)$ by definition. 
Thus, for $\overline{h} \in \hom_{\cW^{op}}^0(L_v,L_w)$, we construct $M_e(\overline{h})$ as a restriction of $M(L_v)$.
Then, the following shows that the restriction is a self-map on $M_e(L_v)$. 
\[M(\overline{h}) \big(M_e(L_v)\big) = M(\overline{h}) \left(\sum_{\overline{g} \in \hom_{\cW^{op}}^0(L_s,L_v)} M(\overline{g})\big(\mathrm{Im} M(\overline{x}_e)\big)\right) = \sum_{\overline{g} \in \hom_{\cW^{op}}^0(L_s,L_v)} M(\overline{h} \circ \overline{g})\big(\mathrm{Im} M(\overline{x}_e)\big) \subset M_e(L_w).\]

\noindent{\bf Construction of $M_e$ (dg-functor):}
We need to prove that the constructed $M_e$ is a dg-functor.
It is trivial from the object and morphism level constructions. 

\noindent{\bf Construction of a natural transformation $T: M_e \to M$:}
In order to complete the induction step, we construct a natural transformation $T: M_e \to M$ in $\Prop \cW$, i.e., collection of multi-linear maps $T = \{T^d\}_{d \geq 0}$.

For $d =0$, we set $T^0: M_e(L_v) \to M(L_v)$ as the inclusion of $k$-modules for all $v \in V(Q)$. 
For $d \geq 1$, we set $T^d =0$. 
Then, one can easily see that $T$ is closed in $\Prop \cW$. 
This part is the same as that in the proof of Lemma \ref{lem extra condition 1}, so we skip it.

\noindent{\bf The induction step:}
We are ready to complete the induction step. 
We recall that by the definition of $M_e(L_s)$ (here, we recall that $s$ is the starting point of an arrow $e$) and by Equation \eqref{eqn induction step}, 
\[0 \neq \mathrm{Im}M(\overline{x}_e) \subset M_e(L_s) \subsetneq M(L_s).\]
It implies that 
\[0 < \dim_k M_e(L_s), \dim_k \Cone(T)(L_s) < \dim_k M(L_s)\]
and thus
\begin{gather}
	\label{eqn induction hypothesis}
	0 < \sum_{v \in V(Q)} \dim_k M_e(L_v), \sum_{v \in V(Q)} \dim_k \Cone(T)(L_v) < \ell.
\end{gather}

We note that by the construction, $M_e$ and $\Cone(T)$ are modules concentrated at degree $0$. 
By applying Lemma \ref{lem extra condition 3}, $M_e(T)$ and $\Cone(T)$ are divided into simpler modules, i.e., modules satisfying the assumptions of Section \ref{subsection proof of Theorem 1}, and also satisfying the induction hypothesis (because of Equation \eqref{eqn induction hypothesis}).
It means that the generators generate $M_e$ and $\Cone(T)$. 
Since $M_e$ and $\Cone(T)$ generate $M$, the induction step holds. \qed

\begin{rmk}
	\label{rmk with potential}
	Theorem \ref{thm generation of proper modules} and Corollary \ref{cor generation of ginzburg dg algebra} deal with Ginzburg dg algebras {\em without} potential, or equivalently, the zero potential. 
	The reason why we do {\em not} consider more general case, i.e., Ginzburg algebras associated to quivers {\em with} potentials, is that we only have a geometric model for Ginzburg algebras without potential. 
	However, we expect that the main idea could be generalized for a general case, i.e., a quiver $Q$ with arbitrary potential $W$. 
	 
	Before explaining our expectation, we recall that if a potential $W$ is generic, in other words, the corresponding Ginzburg dg algebra $\Gamma_n(Q,W)$ is Jacobi-finite, then there is a generation result for $\Prop \Gamma_n(Q,W)$. 
	Thus, for the generation result, the problem is the infinite dimensional part of $H^0(\Gamma_n(Q,W))$. 
	In the proof of Theorem \ref{thm generation of proper modules}, we handle the infinite dimensional part by using {\em cycles} of $Q$. 
	Similarly, if we have an arbitrary potential $W$ inducing infinite dimensional $H^0(\Gamma_n(Q,W))$, we expect that ``some'' cycles of $Q$ will correspond to the infinite dimensional part of the zeroth cohomology.
	
	If our hope works, then the key of proving Theorem \ref{thm generation of proper modules} for an arbitrary $W$ is to classify ``some" cycles of $Q$ with respect to the given $W$. 
\end{rmk}
\qed

\section{The immersed compact Fukaya category of plumbings}
\label{section the compact Fukaya category of plumbings}

The purpose of this section is to prove Theorem~\ref{thm immersed module
	correspondence} and its corollary, which together establish equivalences
\[
\cF(P;d_e) \;\simeq\; \Prop\,\cW(P;d_e) \;\simeq\; \mu\mathrm{Sh}(L)
\]
between the immersed compact Fukaya category of a plumbing space $P$, the
category of proper modules over its wrapped Fukaya category, and the category
of microlocal sheaves on the skeleton $L = \bigcup_{v\in V(Q)} M_v$ of $P$.
This generalises the Nadler--Zaslow theorem \cite{nad-zas09} to plumbing
spaces and is dual to the Lagrangian--sheaf correspondence of \cite{gps3}.
The first equivalence is proved by showing that each generating proper module
of Theorem~\ref{thm generation of proper modules} is realised as the Yoneda
image of an object of $\cF(P;d_e)$.
The key symplectic input is the generation theorem of \cite{Jeong-Karabas-Lee26}:
any exact Lagrangian immersion equipped with a bounding cochain supported on
its positive-action double points is generated by the Lagrangian cocores of
the ambient Weinstein manifold.

In Section~\ref{subsection recalled results immersed} we recall the relevant
results from \cite{Jeong-Karabas-Lee26} and use them to define the immersed
compact Fukaya category $\cF(W)$ and establish its basic properties.  In
Section~\ref{subsection proper modules as immersed Lagrangians in plumbings}
we specialise to plumbing spaces, constructing the relevant objects of
$\cF(P;d_e)$ and proving Theorem~\ref{thm immersed module correspondence}
and Corollary~\ref{cor microlocal sheaves}.

\subsection{The immersed compact Fukaya category}
\label{subsection recalled results immersed}
Let $(W, \omega = d\theta)$ be a Weinstein manifold of finite type. An \emph{exact Lagrangian immersion (with cylindrical end)} is a proper
immersion $i\colon L \hookrightarrow W$ admitting a smooth function
$f\colon L \to \R$ with $i^*\theta = df$, whose self-intersection locus
consists of finitely many transverse double points, and whose image is
tangent to the Liouville vector field outside a compact subset of $W$. By
abuse of notation we write $L$ for the image $i(L)$ and sometimes call it
an immersed Lagrangian.

Each double point $x$ with preimage $i^{-1}(x) = \{x^+, x^-\}$ contributes
two generators to the wrapped Floer cochain complex $\CW^*(L,L)$, denoted
$x_+:=(x^+,x^-)$ and $x_-:=(x^-,x^+)$, with actions
\[
a(x_+) = f(x^+) - f(x^-), \qquad a(x_-) = f(x^-) - f(x^+).
\]
By a small perturbation of $f$ (or $L$ if necessary) one can always arrange $f(x^+) > f(x^-)$,
so that $a(x_+) > 0$ and $a(x_-) < 0$. We call $x_+$ a
\emph{positive-action double-point generator}.  A \emph{bounding cochain}
for $L$ is a degree-$1$ element $b \in \CW^1(L,L)$ satisfying the
Maurer--Cartan equation $\sum_{k \geq 0} m_k(b,\ldots,b) = 0$.

\begin{dfn}
	We say a bounding cochain $b$ is \emph{supported on the positive-action
		double points} if $b$ is a linear combination of positive-action
	double-point generators.
\end{dfn}

By \cite{Jeong-Karabas-Lee26}, this is equivalent to equipping
the Legendrian lift $L^+ := \{(i(p),-f(p))\mid p\in L\}\subset W\times\R$ with an augmentation of its Chekanov--Eliashberg
algebra, which is precisely the condition used in \cite{cdrgg17} to define
Floer theory for immersed Lagrangians.

It is shown in \cite{Jeong-Karabas-Lee26} that exact Lagrangian immersions
equipped with bounding cochains supported on positive-action double points
form a well-defined $A_\infty$-category whose morphism spaces are computed
via Hamiltonian perturbed Floer theory.  We recall that \cite{cdrgg17}
defined Floer-theoretic operations for immersed Lagrangians via Hamiltonian
perturbations, but did not verify the coherence conditions required to
assemble these into a well-defined $A_\infty$-category.  This gap is filled
in \cite{Jeong-Karabas-Lee26} by working within Gao's immersed wrapped Fukaya
category $\cW^{\mathrm{im}}(W)$ \cite{gao17}, which is well-defined by construction via pearly tree
trajectories, and showing that on the objects of interest Gao's pearly-tree
$A_\infty$-operations reduce to holomorphic curve counts, thereby agreeing
with the Hamiltonian perturbation approach of \cite{cdrgg17}. Moreover, the
following generation result is established.

\begin{thm}[\cite{Jeong-Karabas-Lee26}]
	\label{thm immersed generation recalled}
	Let $L$ be an exact Lagrangian immersion in $W$ equipped with a bounding
	cochain $b$ supported on its positive-action double points.  Then $(L,b)$
	is generated by the Lagrangian cocores of $W$ in $\cW^{\mathrm{im}}(W)$.
	In particular, $(L,b)$ is an admissible object of the (embedded) wrapped Fukaya
	category $\cW(W)$ in the sense of \cite{cdrgg17} or \cite{gps1}.
\end{thm}

We now define the immersed compact Fukaya category of $W$ that we work with
throughout this section.

\begin{dfn}
	\label{dfn compact Fukaya category}
	The \emph{immersed compact Fukaya category} $\cF(W)$ is the full
	$A_\infty$-subcategory of $\cW^{\mathrm{im}}(W)$ whose objects are
	pairs $(L,b)$ where $L$ is a compact exact Lagrangian immersion and $b$
	is a bounding cochain supported on the positive-action double points of
	$L$.
\end{dfn}

Note that embedded compact exact Lagrangians are objects of $\cF(W)$, taking
$b=0$ as the bounding cochain.

\begin{prop}
	\label{prop compact Fukaya subcategories}
	$\cF(W)$ is a full $A_\infty$-subcategory of $\cW(W)$. Moreover, the
	Yoneda functor restricts to a fully faithful functor
	\[
	\cY\colon \cF(W) \longrightarrow \Prop\,\cW(W).
	\]
\end{prop}

\begin{proof}
	The first statement follows from Theorem~\ref{thm immersed generation
		recalled}: every object $(L,b)$ of $\cF(W)$ is an admissible object of
	$\cW(W)$, so $\cF(W)$ is a full $A_\infty$-subcategory of $\cW(W)$.
	For the second, the Yoneda functor is fully faithful on $\cW(W)$ by the
	Yoneda lemma, so it remains to check that $\cY(L,b)$ is a proper module.
	Since $L$ is compact, the morphism space $\hom_{\cW(W)}(K,L)$ is
	finite-dimensional for every object $K$, hence $\cY(L,b)$ takes values
	in $\Perf\, k$ and is therefore a proper $\cW(W)$-module.
\end{proof}

\subsection{Proper modules as immersed compact Lagrangians in plumbings}
\label{subsection proper modules as immersed Lagrangians in plumbings}

We now specialise to plumbing spaces and prove
Theorem~\ref{thm immersed module correspondence}.  Fix plumbing data
$(Q,M,\s)$ of dimension $2n\geq 6$ and a grading structure
$\{d_e\}_{e\in E(Q)}$ satisfying the non-positivity
condition~\eqref{eqn non-positively grading condition}.  Write $P =
P(Q,M,\s)$ for the associated plumbing space and $\cW = \cW(P;d_e)$ for its wrapped Fukaya category
throughout.  Here $\cF(P;d_e)$ denotes the immersed compact Fukaya category
of $P$ in the sense of Definition~\ref{dfn compact Fukaya category}, where
$\{d_e\}$ plays the same role as in $\cW(P;d_e)$: it determines the gradings
of intersection points between Lagrangians, as described in
Section~\ref{subsubsection grading structure}.

We begin by constructing the objects of $\cF(P;d_e)$ that appear in the
theorem.  Let $I\subset V(Q)$ be a non-empty subset of vertices and for each
$v\in I$ let $\rho_v$ be a local system on $M_v$.  Recall from
Definition~\ref{definition plumbing space} that the core $M_v\subset P$ is
the image of the zero section of $T^*M_v$ under the inclusion $\iota_v\colon
T^*M_v\to P$.  We write $M_v^{\rho_v}$ for $M_v$ equipped with $\rho_v$,
and define the immersed Lagrangian
\[
L \;:=\; \bigcup_{v\in I} M_v^{\rho_v} \;\subset\; P.
\]
We fix a grading on each $M_v$ so that the unique intersection point of
$M_v$ with the cocore $L_v$ has degree $0$, i.e.\ $\hom^0(M_v,L_v)\neq 0$.
For each arrow
$e\colon v\to w\in E(Q)$ with $v,w\in I$, the cores $M_v$ and $M_w$ meet
transversely at a single plumbing point $p_e$. By abuse of
notation, the same point $p_e$ gives two morphisms depending on the direction:
\[
p_e\in\hom^{1-d_e}_\cW(M_w,M_v)
\quad\text{and}\quad
q_e\in\hom^{n-(1-d_e)}_\cW(M_v,M_w),
\]
corresponding to the two ways of viewing the intersection point as a morphism.
The degree of $p_e$ is computed as follows: under
the Yoneda embedding, $p_e$ corresponds to a degree-$k$ module homomorphism
from $\cY(M_w)$ to $\cY(M_v)$, whose nonzero component is by
\cite[Section~(1j)]{sei08} a degree $k-1$ map
\[
\cY(M_w)(L_w)\otimes\hom^*_\cW(L_v,L_w)\longrightarrow\cY(M_v)(L_v).
\]
Since $\cY(M_j)(L_j)=\hom^0(L_j,M_j)$ is concentrated in degree zero for $j\in\{v,w\}$, and
the only nonzero contribution to the map from $\hom^*_\cW(L_v,L_w)$ is from
$x_e\in\hom^{d_e}_\cW(L_v,L_w)$, this reduces to
\[
\hom^0(L_w,M_w)\otimes k\{x_e\}\longrightarrow\hom^0(L_v,M_v),
\]
which forces $k-1=-d_e$, hence $k=1-d_e$. 

In particular, when $d_e=0$, $p_e$ is a degree-$1$ morphism in
$\hom^*_\cW(M_w,M_v)$ and $q_e$ is a degree-$(n-1)\neq 1$ morphism in
$\hom^*_\cW(M_v,M_w)$.
As elements of $\CW^*(L,L)$, $p_e$ and $q_e$ are the two generators
contributed by the double point $p_e$ of $L$ in the sense of
Section~\ref{subsection recalled results immersed}.  Since each $M_v$ is a
zero section, $\theta|_{M_v}=0$ for each $v\in I$, so the potential $f$ of $L$ can be taken
to be globally zero.  By a small perturbation of $L$ and $f$ near each plumbing point, we arrange that $p_e$ has positive
action and $q_e$ has negative action for all $e$ with $d_e=0$.

We now define the bounding cochain.  When $d_e=0$ and $v,w\in I$, the
morphism space
\[
\hom^1_\cW\!\left(M_w^{\rho_w}, M_v^{\rho_v}\right)
\;\simeq\;
\Hom(R_w, R_v)\otimes \hom^1_\cW(M_w, M_v)
\]
contains $\Hom(R_w, R_v)\otimes k\{p_e\} \cong \Hom(R_w,R_v)$ as a
subspace, where $R_v$ and $R_w$ denote the stalks of $\rho_v$ and $\rho_w$.
We fix a choice of $b_e\in\Hom(R_w,R_v)\subset
\hom^1_\cW(M_w^{\rho_w},M_v^{\rho_v})$ for each $e\colon v\to w$ with $d_e=0$ and
$v,w\in I$, and set $b_e=0$ otherwise. The sum
\[
b \;:=\; \sum_{e\in E(Q)} b_e \;\in\; \hom^1(L,L)
\]
is a bounding cochain for $L$ supported on the positive-action double points,
since each $b_e$ is supported on the positive-action double-point generator $p_e$.
It remains to verify the Maurer--Cartan equation
\[\sum_{k\geq 0} m_k(b,\ldots,b) = 0.\]
The $k=0$ term vanishes since $L$ is exact.  For $k\geq 1$, suppose for
contradiction that $m_k(b_{e_1},\ldots,b_{e_k})\neq 0$ for some arrows
$e_1,\ldots,e_k\in E(Q)$.  Then there exists a
holomorphic disk $u$ with boundary on $L$ and input punctures at the
positive-action double-point generators $p_{e_1},\ldots,p_{e_k}$.  The boundary $\partial u$ is a closed loop on $L = \bigcup_{v\in I} M_v$
passing through the plumbing points $p_{e_1},\ldots,p_{e_k}$, and since $u$
is a disk, $[\partial u] = 0$ in $H_1(P;\Z)$.  Using the splitting
$H_1(P;\Z) = H_1(Q;\Z)\oplus\bigoplus_{v\in V(Q)} H_1(M_v;\Z)$, the
projection of $[\partial u]$ onto the $H_1(Q;\Z)$ summand is the closed
edge-path passing through the edges $e_1,\ldots,e_k$, which is nonzero in $H_1(Q;\Z)$.  This contradicts $[\partial u]=0$.  Hence all
$m_k(b,\ldots,b)$ with $k\geq 1$ vanish.
\medskip

We can now state and prove the main theorem of this section.

\begin{thm}
	\label{thm immersed module correspondence}
	Let $(Q,M,\s)$ be plumbing data of dimension $2n\geq 6$ equipped with a
	grading structure $\{d_e\}$ satisfying the non-positivity
	condition~\eqref{eqn non-positively grading condition}.  For any proper module
	$E\in\Prop\,\cW(P;d_e)$ concentrated at degree zero, there exist
	$I\subset V(Q)$, local systems $\{\rho_v\}_{v\in I}$, and a bounding
	cochain $b$ of the form constructed above such that the immersed
	Lagrangian $L = \bigcup_{v\in I} M_v^{\rho_v}$ equipped with $b$ is an
	object of $\cF(P;d_e)$ satisfying
	\[
	\cY(L,b) \;\simeq\; E \quad\text{in }\Prop\,\cW(P;d_e).
	\]
	Consequently, the Yoneda functor restricts to an equivalence
	\[
	\cF(P;d_e)\;\simeq\;\Prop\,\cW(P;d_e).
	\]
\end{thm}

\begin{proof}
	\noindent\textit{Step 1: Constructing $(L,b)$ from $E$.}
	Since $\cW(P;d_e)$ is generated by the cocores $\{L_v\}_{v\in V(Q)}$, the
	module $E$ is determined by the values $E(L_v)$, $E(x_e)$, $E(y_e)$, and
	higher compositions.  Set
	\[
	I := \{v\in V(Q)\mid E(L_v)\not\simeq 0\}.
	\]
	For each $v\in I$, define the local system $\rho_v$ on $M_v$ as follows.
	The stalk of $\rho_v$ is $E(L_v)$.  By Theorem~\ref{thm wrapped Fukaya
		category}, the endomorphism algebra of $L_v$ in $\cW(P;d_e)$ contains
	$C_{-*}(\Omega(M_v\setminus\{\mathrm{pt}\}))$ as a subalgebra.  Restricting
	$E$ to this subalgebra gives a map
	\[
	C_{-*}(\Omega(M_v\setminus\{\mathrm{pt}\}))\longrightarrow
	\hom^*(E(L_v),E(L_v)).
	\]
	Since $E$ is concentrated at degree zero, all components of nonzero degree
	are sent to zero.  Hence $C_{-1}$ maps to zero, which means $C_0$ elements
	are sent to closed elements, and the map factors through
	\[
	H_0(\Omega(M_v\setminus\{\mathrm{pt}\}))
	= \pi_1(M_v\setminus\{\mathrm{pt}\})
	\longrightarrow \Hom^0(E(L_v),E(L_v)).
	\]
	By van Kampen's theorem, $\pi_1(M_v\setminus\{\mathrm{pt}\})\cong\pi_1(M_v)$
	since $n\geq 3$, so this defines a representation of $\pi_1(M_v)$ on
	$E(L_v)$, i.e.\ a local system $\rho_v$ on $M_v$ with stalk $E(L_v)$.
	Define $L := \bigcup_{v\in I} M_v^{\rho_v}$.  For each $e\colon v\to w\in
	E(Q)$ with $d_e=0$ and $v,w\in I$, set
	\[
	b_e := E(x_e)\;\in\;\hom^1\!\left(M_w^{\rho_w}, M_v^{\rho_v}\right),
	\]
	and $b_e=0$ otherwise.  Put $b:=\sum_{e\in E(Q)} b_e$.  By the
	Maurer--Cartan verification above, $(L,b)$ is an object of $\cF(P;d_e)$.
	
	\medskip
	\noindent\textit{Step 2: Verifying $\cY(L,b)\simeq E$.}
	By Theorem~\ref{thm wrapped Fukaya category}, $\cW(P;d_e)$ is generated by
	the cocores $L_v$, the loop algebra generators $C_{-*}(\Omega(M_v\setminus
	\{\mathrm{pt}\}))$, and the morphisms $x_e$, $y_e$, and their higher
	compositions.  It suffices to check agreement on each of these.
	
	\medskip
	\noindent$\cY(L,b)(L_v)$: Since $L_v$ intersects only $M_v$ among the
	cores, with a unique intersection point of degree $0$,
	\begin{align*}
		\cY(L,b)(L_v) = \hom^*(L_v, L)
		\simeq \hom^*(L_v, M_v^{\rho_v})
		\simeq \hom^*(k,E(L_v))\otimes\hom^*(L_v, M_v)
		&\simeq E(L_v)\otimes\hom^*(L_v, M_v)\\
		&\simeq E(L_v),
	\end{align*}
	where the stalk of $\rho_v$ is $E(L_v)$ by construction.
	
	The loop algebra $C_{-*}(\Omega(M_v\setminus\{\mathrm{pt}\}))$ acts on
	$\cY(L,b)(L_v)\simeq E(L_v)$ via the local system $\rho_v$, which was
	constructed precisely so that this action agrees with that of $E$.
	
	\medskip
	\noindent $\cY(L,b)(x_e)$: For $d_e\neq 0$, both $\cY(L,b)(x_e)$ and
	$E(x_e)$ vanish for degree reasons since $E$ is concentrated at degree zero.
	For $d_e=0$,  the Yoneda module evaluated on
	$x_e\in\hom^0(L_v,L_w)$ is given by the $A_\infty$ composition map
	\[
	\hom^*(L_w, L)\xrightarrow{m_2(-\otimes x_e)}\hom^*(L_v, L),
	\]
	which counts rigid holomorphic disks with boundary on $L_v\cup L_w\cup L$
	and corners at $x_e$, intersections of $L_w$ with $L$, and intersections of
	$L_v$ with $L$. Since $L$ intersects $L_w$ only at $M_w$ and $L_v$ only
	at $M_v$, this reduces to counting disks with boundary on $L_v\cup L_w\cup
	M_w\cup M_v$.  The boundary of such a disk must pass through some plumbing points of $L$.
	Since $b$ is supported only on the positive-action generators $p_e$ and not
	on the negative-action generators $q_e$, the boundary can traverse
	each plumbing point only in the positive-action direction and cannot return.  Since $x_e$ is a corner of the disk within the plumbing sector corresponding to $p_e$, any traversal of a plumbing point other than
	$p_e$ would contribute a
	nonzero element to the $H_1(Q;\Z)$ summand of $H_1(P;\Z)$, contradicting
	null-homologousness of the boundary by the same splitting argument as in the
	Maurer--Cartan verification above.  Therefore the boundary passes through $p_e$ exactly once, and the disk is
	contained in the plumbing sector $\Pi_n$ corresponding to $e$.  Since all plumbing sectors are the same standard sector $\Pi_n\simeq\C^n$,
	the count of such holomorphic disks is a universal constant $\lambda$,
	independent of $e$, and it is clearly nonzero (in fact, $\lambda=1$).
	Redefining $b_e\mapsto b_e/\lambda$, each of the $\lambda$ disks
	contributes $b_e/\lambda$ via the bounding cochain value at the corner $p_e$,
	giving
	\[
	\cY(L,b)(x_e) = \lambda \cdot \frac{b_e}{\lambda} = b_e = E(x_e).
	\]

	\medskip
	\noindent$\cY(L,b)(y_e)$: Since $|y_e| = 2-n-d_e < 0$ strictly by
	condition~\eqref{eqn non-positively grading condition}, both
	$\cY(L,b)(y_e)$ and $E(y_e)$ vanish as $E$ is concentrated at degree zero.
	
	\medskip
	\noindent Higher compositions: Since all generators of $\cW(P;d_e)$ have
	nonpositive degree, any higher composition produces a morphism of negative
	degree, so both $\cY(L,b)$ and $E$ evaluate to zero since both are
	concentrated at degree zero.
	
	\medskip
	Hence $\cY(L,b)\simeq E$.
	
	\medskip
	\noindent\textit{Equivalence.}
	By Theorem~\ref{thm generation of proper modules}, every
	$E\in\Prop\,\cW(P;d_e)$ is generated by the modules in
	$\bigcup_{C\in\mathbb{B}}\cC_{(C,0)}$.  Step~2 shows each such module is
	the Yoneda image of an object of $\cF(P;d_e)$.  Since the Yoneda functor
	is fully faithful and maps $\cF(P;d_e)$ into $\Prop\,\cW(P;d_e)$ by
	Proposition~\ref{prop compact Fukaya subcategories}, the induced functor
	$\cY\colon\cF(P;d_e)\to\Prop\,\cW(P;d_e)$ is an equivalence.
\end{proof}

Recall that by \cite{nad16}, the category of (traditional) microlocal sheaves on the
skeleton $\mathbb L$ of a Weinstein manifold of finite type $W$ is equivalent to
the category of proper modules over wrapped microlocal sheaves on $\mathbb L$,
\[
\mu\mathrm{Sh}(\mathbb L) \;\simeq\; \Prop\,\mu\mathrm{Sh}^w(\mathbb L),
\]
and by \cite{gps3} the category of wrapped microlocal sheaves on $\mathbb L$ is equivalent to
the wrapped Fukaya category of $W$, giving
\[
\mu\mathrm{Sh}(\mathbb L) \;\simeq\; \Prop\,\cW(W).
\]
Combined with Theorem~\ref{thm immersed module correspondence}, this gives
the following.

\begin{cor}\label{cor microlocal sheaves}
	Under the non-positivity condition~\eqref{eqn non-positively grading
		condition}, the immersed compact Fukaya category $\cF(P;d_e)$ of the plumbing $P$ is
	equivalent to the category of microlocal sheaves on the skeleton
	$\mathbb L = \bigcup_{v\in V(Q)} M_v$ of $P$, equipped with the appropriate grading
	structure:
	\[
	\mu\mathrm{Sh}(\mathbb L) \;\simeq\; \cF(P;d_e).
	\]
\end{cor}

\begin{rmk}
	This generalises the Nadler--Zaslow theorem \cite[Theorem 4.1.3, Proposition 4.6.1]{nad09} (building on earlier work on \cite{nad-zas09}), which
	establishes the same equivalence for cotangent bundles with embedded
	Lagrangians, to plumbing spaces.  Immersed Lagrangians are necessary when $Q$ has loops or degree zero cycles.  This
	equivalence is dual to the Lagrangian--sheaf correspondence of
	\cite{gps3}, which establishes the wrapped analogue.
\end{rmk}

\section{An alternative proof}
\label{section an alternative proof}

In this section, we provide an alternative proof of Theorem \ref{thm intro compact Lagrangian} that is proven in Section \ref{section the compact Fukaya category of plumbings}. 
In the previous section, we used the techniques from \cite{cdrgg17} in order to prove Theorem \ref{thm intro compact Lagrangian}, but the present section will use other methods from \cite{gps2}. 
This alternative way provides a pictorial explanation of the proof (see \ref{figure cobordism}), but the drawback is that it is related to an ad hoc idea, tailored in the setting of plumbing spaces.

\subsection{Motivation and Assumptions}
\label{subsection motivation and assumptions}

In order to give a pictorial description, we would like to start with a geometric motivation.
We also explain a technical assumption, which we need to apply the techniques from \cite{gps2}.

\subsubsection{Motivation} 
\label{subsubsection motivation}
Let $W$ be a Weinstein manifold equipped with a Weinstein sectorial covering $\{W_i\}$.
Ganatra--Pardon--Shende \cite{gps2} proves that the wrapped Fukaya category of $W$, $\cW(W)$, is equivalent to the gluing of those of $W_i$. 
More precisely, we have $\cW(W) \simeq \mathrm{hocolimit}\left( \cW(W_i)\right)$.

One question, which naturally follows the result of \cite{gps2}, is to ask about the computation of compact Fukaya category of $W$. 
To investigate this question, one could try to understand a compact Lagrangian $L \subset W$ as an object of the glued category, i.e., $\mathrm{hocolimit} \left(\cW(W_i)\right)$. 

From this point of view, one could see $L$ as a gluing of Lagrangians $L \cap W_i \subset W_i$, or equivalently, the corresponding object in $\cW(W_i)$. 
To be more precise, we note that $\cW(W_i)$ is defined to be the wrapped Fukaya category of its {\em convex completion} $\left(\widetilde{W}_i, \widetilde{\Lambda}_i\right)$. 
(See \cite[Section 2.7]{gps1} for the definition of convex completion.)
Geometrically, $L \cap W_i$ also extends to a Lagrangian $\widetilde{L}_i \subset \widetilde{W}_i$. 

Since $L$ is compact, the boundary of $L \cap W_i$ would be a subset of the actual boundary of the Weinstein sector $W_i$. 
Or equivalently, the asymptotic boundary of $\widetilde{L}_i$ is a subset of $\widetilde{\Lambda}_i$.
It means that $\widetilde{L}_i$ is an object of the {\em infinitesimal Fukaya category} of $\left(\widetilde{W}_i, \widetilde{\Lambda}_i\right)$.
(For the definition of infinitesimal Fukaya category, see \cite{Nadler14} and references therein.)

Thus, it would be natural to expect that the object $L$ of $\cF(W)$ is obtained by gluing objects of the infinitesimal Fukaya categories of convex completions of $W_i$.
And, the generators of $\cF(W)$ are also obtained by gluing generators of infinitesimal Fukaya categories, if there exist. 

We apply this idea to plumbing spaces. 
Let us recall that a plumbing space $P$ admits a Weinstein sectorial covering given in Equation \eqref{eqn Weinstein sectorial covering}. 
Using the notation from Section \ref{subsection plumbing space}, each sector in the covering is either a cotangent bundle $T^*\widetilde{M}_v$ or a plumbing sector $\Pi_n$. 

We note that the infinitesimal Fukaya category of a cotangent bundle $T^*\widetilde{M}_v$ is generated by the zero section equipped with local systems. 
And for the plumbing sector $\Pi_n = \left(\mathbb{C}^n, \partial_\infty \left(\mathbb{R}^n \cup i \mathbb{R}^n\right)\right)$, its infinitesimal Fukaya category would be generated by two Lagrangian disks $\mathbb{R}^n$ and $i \mathbb{R}^n$. 

The expected generators of $\cF(P)$ will be introduced in the next subsection. 
Before moving to the next topic, we would like to remark the following: 
In the plumbing sector $\Pi_n$, we have two Lagrangian disks that intersect each other. 
When we glue them to obtain a generator of $\cF(P)$, we may need to glue both Lagrangian disks.
It means that the expected generator could be an immersed Lagrangian. 

The simplest example could be found in the self-plumbing of one cotangent bundle, i.e., a plumbing space whose plumbing pattern is a quiver consisting of one vertex $v$ and one arrow $e = v \to v$. 
Then, the zero section of the cotangent bundle in the plumbing space is an immersed Lagrangian that is expected to be obtained by the above gluing procedure. 

It means that such immersed Lagrangians should be objects of the category that we are interested in.
In other words, the compact Fukaya category we are dealing in the current paper should contain some immersed Lagrangians. 
Thus, we discuss {\em immersed Lagrangian Floer theory} below. 

The concept of immersed Lagrangian Floer theory has been studied by many researchers, for example, see \cite{Akaho-Joyce10, gao17, als-bao18}, etc.
However, to our best knowledge, it is not fully established in the current literature. 
Thus, we assume that the theory is well-defined in below. 
More specific assumptions are stated in Section \ref{subsubsection assumptions}.

\subsubsection{Assumptions}
\label{subsubsection assumptions} 
From above, we need to assume that immersed Lagrangian Floer theory is well-established.
It is explained in Section \ref{subsection recalled results immersed}. 

Since holomorphic curve counts is assumed to define $\cW^{im}(X,\Lambda)$, it is believable that $\cW^{im}(X,\Lambda)$ satisfies basic properties of holomorphic curve counts. 
Such properties are axiomatized in \cite[Secton 2.3]{gps2}, and if a category satisfies these properties, it is called {\em abstract (wrapped) Floer setup}.
Moreover, \cite[Section 2.4]{gps2} proves that the usual wrapped Fukaya category $\cW(X,\Lambda)$ satisfies the properties. 

Thus, it sounds reasonable to believe that wrapped Fukaya category of immersed Lagrangians also is an abstract (wrapped) Floer setup, but we do not prove this for simplicity of the paper. 
Instead, we assume that in the current section.
In other words,
\begin{assumption}
	\label{assumption 2}
	$\cW^{im}(X,\Lambda)$ is an abstract (wrapped) Floer setup.
\end{assumption}

\subsection{Immersed Lagrangians in plumbings}
\label{subsection immersed Lagrangians in plumbings}
We now adapt the idea in Section \ref{subsection motivation and assumptions} to plumbing spaces $P$. 
We first fix a collection of immersed Lagrangians that are, by arguments in Section \ref{subsection motivation and assumptions}, expected to generate the subcategory of $\cW(P)$, consisting of compact Lagrangians. 

Again, we let $P$ denote a plumbing space corresponding to plumbing data $(Q,M,\s)$ equipped with a non-positively graded $\{d_e\}$. 
Let us fix a sub-quiver $I$ of $V(Q)$, i.e., $I$ is a quiver such that $V(I) \subset V(Q), E(I) \subset E(Q)$. 
Then, we can consider a Lagrangian $L$ defined as 
\[L := \cup_{v \in V(I)} M_v.\]
We note that $L$ is an immersed Lagrangian if there exists an arrow $e = v \to w \in E(Q)$ such that $v, w \in V(I)$. 
Moreover, immersing points of $L$ are plumbing points, so the set of immersing points of $L$ can be identified with the following set 
\[\left\{e = v \to w \in E(Q) | v, w \in V(I)\right\}.\]

Before going further, we specify a grading on $M_v$. 
Recall that $M_v$ is the zero section of cotangent bundle $T^*M_v$, of which $P$ consists. 
Then, its cotangent fiber corresponds to a cocore generator $L_v$ of $\cW(P;d_e)$. 
We fix a grading on $M_v$ so that the intersection point between the zero section $M_v$ and a cotangent fiber $L_v$ has degree $0$, i.e., 
\[\hom^0\left(M_v,L_v\right) \neq 0, \text{ or equivalently,  } \hom^n\left(L_v,M_v\right) \neq 0,\]
in the wrapped Fukaya category of $T^*M_v$. 

Inside $T^*M_v$, $M_v$ {\em equipped with a local system} could be seen as an object of its Fukaya category. 
Let $\rho_v$ denote a local system on $M_v$ and $(M_v,\rho_v)$ denote the corresponding object in the Fukaya category. 
For convenience, let $M_V^{\rho_v}$ denote the Lagrangian $M_v$ equipped with an actual local system $\rho_v$. 

Then, we have a (possibly) immersed Lagrangian submanifold $L = \oplus_{v \in V(I)} M_v$ and its local system $\rho = \oplus_{v \in V(I)} \rho_v$ inside the plumbing space $P$, since $P$ is a combination of cotangent bundles $T^*M_v$ for $v \in V(Q)$. 

Since $(L, \rho)$ is an immersed Lagrangian, we could consider bounding cochains on it.
In the rest of the section, we restrict our attention to some specific bounding cochains that are constructed below: 
First, we fix the zero bounding cochain.
Then, 
\[\left(L, \rho, 0\right) = \oplus_{v \in V(I)} M_v^{\rho_v},\]
in the usual wrapped Fukaya category $\cW(P;d_e)$. 
Thus, There exists a morphism space 
\[\hom^*_\cW\left(\left(L,\rho,0\right),\left(L,\rho,0\right)\right).\]

Now, we recall that for every $e = v \to w \in E(Q)$, we have a plumbing point $p_e \in M_v \cap M_w$ corresponding to $e$. 
Then, there exists a generator in $\hom^{1-d_e}_\cW\left(M_w,M_v\right)$ corresponding to $p_e$. 
Especially, when $e \in E(I)$ is of degree $0$, i.e., $d_e=0$, $p_e \in \hom_\cW^1(M_w,M_v)$. 
Let $b_e$ denote a choice of morphism in $\hom^1_\cW\left(M_w^{\rho_w},M_v^{\rho_v}\right)$, which is represented by a tensor product with the plumbing point $p_e$ in 
\[\hom^1_\cW\left(M_w^{\rho_w},M_v^{\rho_v}\right) = \hom^*\left(E_v, E_w\right) \otimes \hom_\cW^{1-*}\left(M_v, M_w\right),\]
where $E_v$ and $E_w$ denote the stalk of $\rho_v$ and $\rho_w$. 
(Here, we are seeing a local system as a locally constant sheaf, so the word ``stalk'' makes sense.)
Now, we consider the bounding cochains $b$ which can be written as
\[b = \oplus_{e \in E(I), d_e =0} b_e \in \hom^1_\cW\left((L,\rho,0),(L,\rho,0)\right).\]
We note that one can easily check that such a $b$ satisfies the Maurer-Cartan equation.

\begin{rmk}
	\label{rmk immersing point for a loop}
	We note that $M_v$ could be an immersed Lagrangian if $Q$ has a loop, i.e., length $1$ cycle. 
	Thus, rigorously, we should prove that $M_v^{\rho_v}$, equipped with the zero bounding cochain, is an object of $\cW(P;d_e)$ for all $v \in V(Q)$, before discussing $\hom_{\cW}^1 \left(M_v^{\rho_v},M_w^{\rho_w}\right)$. 
	One could prove it by employing the proof of Lemma \ref{lem immersed Lagrangian}, and we skip the proof.
\end{rmk}

Now, we can state and prove Lemma \ref{lem immersed Lagrangian}.
\begin{lem}
	\label{lem immersed Lagrangian}
	In the above setting and under Assumption \ref{assumption 2}, a Lagrangian equipped with a local system and a bounding cochain $(L,\rho,b)$, constructed as above, is an admissible object of the wrapped Fukaya category $\cW(P;d_e)$, or equivalently, is generated by cocore generators of $\cW(P;d_e)$.
\end{lem}
\begin{proof}
	We fix a sub-quiver $I$ of $Q$.
	Also, we fix an immersed Lagrangian equipped with a local system and a bounding cochain constructed as above, i.e.,
	\[\left(L = \cup_{v\in V(I)} M_v, \rho = \oplus_{v\in V(I)} \rho_v, b = \oplus_{e \in E(I), d_e =0} b_e\right).\]
	In order to prove that the fixed immersed Lagrangian is generated by cocores, we use (the proof of) \cite[Lemma 1.37]{gps2}.
	
	More precisely, we first choose some Lagrangians and an exact Lagrangian cobordism $C$ in the symplectization of $\partial_\infty P$ such that 
	\begin{itemize}
		\item the chosen Lagrangians are generated by cocores,
		\item the negative end of $C$ is equal to the union of asymptotic boundaries of the chosen Lagrangians, and
		\item the gluing of the chosen Lagrangian and $C$ is our starting immersed Lagrangian $L$.
	\end{itemize} 
	Then, the proof of \cite[Proposition 1.37]{gps2} (and that of \cite[Lemma 4.1]{gps2}), together with Assumption \ref{assumption 2}, proves that $L$ is equivalent to a twisted complex consisting of the chosen Lagrangians in $\cW^{im}(P)$. 
	It implies that $L$ is also generated by cocores, since the chosen Lagrangians are.  
	
	To construct Lagrangians and cobordism $C$ above, we consider the Weinstein sectorial covering in Equation \eqref{eqn Weinstein sectorial covering}, i.e., 
	\[P=\bigcup_{v\in V(Q)} T^*\widetilde M_v \,\cup\, \bigcup_{e\in E(Q)} \Pi_n.\] 
	Then, for $v \in V(Q)$ and $e = t \to h \in E(Q)$, we define Lagrangians $X_v, X_e$, and $Y_{e,t}, Y_{e,h}$ in $T^*\widetilde{M}_v$, $\Pi_n$, and their intersections, as follows. 
	Note that Figure \ref{figure cobordism} provides a conceptual picture of the chosen Lagrangians and cobordisms. 
	\vskip0.1in
	
	\noindent{\em For $v \in V(Q)$}:
	First, if $v \notin V(I)$, then we set $X_v$ an empty set. 
	 
	Now, let us assume that $v \in V(I) \subset V(Q)$. 
	We consider the zero section $\widetilde M_v$ of the Weinstein sector $T^*\widetilde M_v$. 
	After taking convex completion of $T^*\widetilde M_v$ to a Weinstein manifold with stop as described in \cite[Section 2.8]{gps1}, $\widetilde M_v$ becomes asymptotic to the stop along its boundary, hence $\widetilde{M}_v$ an object of the infinitesimal wrapped Fukaya category of the convex completion. 
	Let $X_v$ be a perturbation of $\widetilde M_v$ by a small negative Reeb flow inside the completion.
	Then, as described in an example of \cite[Section 1.2]{gps1}, $X_v$ is disjoint from the stop. 
	Hence, after reversing the convex completion, i.e., after removing a small neighborhood of the stop, $X_v$ is a Lagrangian submanifold in $T^*\widetilde M_v$ which does not touch the actual boundary of the Weinstein sector.
	Thus, $X_v$ is an admissible object of the wrapped Fukaya category $\cW(T^*\widetilde M_v)$.
	
	Moreover, the inclusion of Weinstein sectors $T^*\widetilde M_v\hookrightarrow P$ induces a functor $\cW(T^*\widetilde M_v)\to\cW(P;d_e)$ by \cite{gps1}, which sends Lagrangians to their images under the inclusion. 
	Therefore, we can consider $X_v$ as an object of $\cW(P;d_e)$.
	
	Let us recall that the given local system $\rho_v$ on $M_v$ can be defined to be a representation of the first fundamental group $\pi_1(M_v)$. 
	Since $\widetilde{M}_v$ is obtained by removing $n$-dimensional open balls from $M_v$, when $n \geq 3$, $\widetilde{M}_v$ and $M_v$ have the same first fundamental group. 
	(It is an easy application of the Van Kampen theorem.)
	Thus, a local system $\rho_v$ can be seen as that on $\widetilde{M}_v$, and even on $X_v$.
	
	Finally, a Lagrangian equipped with a local system and the zero bounding cochain $\left(X_v, \rho_v, 0\right)$ is an admissible object of $\cW(P;d_e)$. 
	\vskip0.1in
	
	\noindent{\em For $e \in E(Q)$}:
	Let $\Pi_n$ be a Weinstein sector in the covering, corresponding to $e = t \to h \in E(Q)$. 
	We first set notation. 
	
	We recall that the convex completion of $\Pi_n$ is the Weinstein manifold with stop $(\C^n,\partial_{\infty}(\R^n\sqcup i\R^n))$. 
	Let $C_t$ denote $\R^n\subset\C^n$, and let $C_h$ denote $i\R^n\subset\C^n$. 
	They are asymptotic to the stop, hence objects of the infinitesimal wrapped Fukaya category of $(\C^n,\partial_{\infty}(\R^n\sqcup i\R^n))$. 
	As we discussed before, the perturbation $C_t'$ (resp.\ $C_h'$) of $C_t$ (resp.\ $C_h$) by a small negative Reeb flow inside the completion makes it an object of the wrapped Fukaya category $\cW(\Pi_n)$.
	
	Recall that, for each $e = t\to h\in E(I)$, the inclusion of Weinstein sectors $\Pi_n\hookrightarrow P$ induces a functor $\cW(\Pi_n)\to\cW(P;d_e)$, which sends Lagrangians $C_t'$ and $C_h'$ to their images under the inclusion. 
	Call again the images $C_t'$ and $C_h'$, respectively. 
	Then, we can consider $C_t'$ and $C_h'$ as objects of $\cW(P;d_e)$.
	
	We note that they are topologically disks, they could not have a non-trivial local system. 
	Thus, to specify a local system on them, it is enough to fix the rank of the local system. 
	Set $\rho_t'$ (resp.\ $\rho_h'$) be the trivial local system on $C'_t$ (resp.\ $C'_h$), whose rank is the same as the rank of $\rho_t$ (resp.\ $\rho_h$).
	Recall that $\rho_t$ and $\rho_h$ are the given local systems on $M_t$ and $M_h$, which are parts of $L$. 
	Note that if $t$ (resp.\ $h$) $\notin V(I)$, we could think $\rho_t'$ (resp.\ $\rho_h'$) as a zero dimensional representation.  
	
	We also note that the morphism space from $C_h'$ and $C_t'$ is one-dimensional in $\cW(\Pi_n)$. 
	When we assume that $C_t'$ and $C_h'$ are graded as subsets of $M_v$ and $M_w$, there exists a nonzero degree $(1-d_e)$ morphism from $C_t'$ to $C_h'$.
	Moreover, the morphism space $\hom_\cW^*\left((C_h',\rho_h'), (C_t',\rho_t')\right)$ is equivalent to the subspace of $\hom_\cW^*\left(M_h^{\rho_h}, M_t^{\rho_t}\right)$, which is generated by the plumbing point corresponding to $e$. 
	
	Now, we choose a Lagrangian $X_e$ by considering each of the following cases separately.
	\begin{enumerate}
		\item[(i)] $e \in E(I)$;
		\item[(ii)] $e = t \to h \notin E(I)$, but $\{t, h\} \cap V(I) \neq \varnothing$;
		\item[(iii)] Otherwise. 
	\end{enumerate}
	
	The simplest case is (iii). 
	We just set $X_e$ the empty set. 
	
	For (ii), our choice is dependent on the intersection $\{t,h\} \cap V(I)$. 
	Based on the intersection, we set 
	\[X_e = \begin{cases}
		 (C_t', \rho_t', 0) &\text{  if  } \{t,h\} \cap V(I) = \{t\}, \\
		 (C_h', \rho_h', 0) &\text{  if  } \{t,h\} \cap V(I) = \{h\}, \\
		 (C_t' \cup C_h', \rho_t' \oplus \rho_h', 0) &\text{  if  } \{t,h\} \cap V(I) = \{t, h\}.
	\end{cases}\]
	We note that the first two cases do not need a bounding cochain, since the corresponding Lagrangians are embedded.
	For the last case, we need to specify a bounding cochain since $C_t' \cup C_h'$ has a double point.
	We choose the zero bounding cochain (since $e \notin E(I)$). 
	
	For (i), the Lagrangian submanifold and the local system on it are defined as the same as the last case of (ii), i.e., 
	\[C_t' \cup C_h' \text{  and  } \rho_t' \oplus \rho_h'.\]
	To choose a bounding cochain on it, we consider $d_e$. 
	To get a nonzero bounding cochain, we need a nonzero degree $1$ morphism between $C_t'$ and $C_h'$. 
	It only happens when $d_e = 0$, since $d_e$ satisfies \ref{eqn non-positively grading condition}.

	Thus, if $d_e \neq 0$, then $X_e$ should be 
	\[X_e = (C_t' \cup C_h', \rho_t' \oplus \rho_h', 0).\]
	If $d_e =0$, we recall that $\hom_\cW^1\left((C_h',\rho_h'), (C_t',\rho_t')\right)$ is equivalent to the subspace of $\hom_\cW^1\left(M_h^{\rho_h}, M_t^{\rho_t}\right)$.
	Thus, the given bounding cochain $b_e$ of $L$ can be seen as an element of $\hom_\cW^1\left((C_h',\rho_h'), (C_t',\rho_t')\right)$.
	And, we set 
	\[X_e = \left(C_t' \cup C_h', \rho_t' \oplus \rho_h', b_e\right),\]
	which is equivalent to the following twisted complex, in $\cW(P;d_e)$, 
	\[\left((C_t',\rho_t') \oplus (C_h',\rho_h'), \begin{bmatrix}
		0 & 0 \\ b_e & 0
	\end{bmatrix}\right).\]
	\vskip0.1in
	
	\noindent{\em For the pair $(e,v) \in E(Q) \times V(Q)$, where $v$ is one of the end points of $e$}:
	For this case, two Weinstein sectors $T^*\widetilde{M}_v$ and $\Pi_n$ corresponding to $e$ intersect along their boundary.
	By slightly thickening it, we assume that their intersection is another Weinstein sector of codimension $0$ in $P$, which is equivalent to 
	\[T^*[-1,1] \times T^*S^{n-1}.\]
	
	In the intersection, if $v \in V(I)$ (and so that $X_v$ is not empty), we fix a Lagrangian $T^*_0[-1,1] \times S^{n-1}$, where the first factor is a cotangent fiber of $T^*[-1,1]$ and the second factor $S^{n-1}$ is the zero section of $T^*S^{n-1}$. 
	Note that the fixed Lagrangian does not touch the actual boundary of the intersection Weinstein sector. 
	We set $Y_{e,v}$ as the fixed Lagrangian in the thickened intersection of $T^*\widetilde{M}_v$ (resp.\ $T^*\widetilde{M}_w$) and $\Pi_n$.
	Note that $Y_{e,v}$ is homeomorphic to $[-1,1] \times S^{n-1}$ with $n \geq 3$, they cannot have a nontrivial local system. 
	Thus, similar to the case of $C_t'$ and $C_h'$ above, fixing a rank of a local system is enough to specify a local system on it. 
	We set $\rho_{e,v}$ a local system on $Y_{e,v}$ whose rank is the same as that of $\rho_v$.
	And since $Y_{e,v}$ is embedded in the intersection and also in $P$, we do not specify a bounding cochain of it. 
	Then, $\left(Y_{e,v}, \rho_{e,v}, 0\right)$ is an admissible object in $\cW(P;d_e)$.
	
	If $v \notin V(I)$, we simply choose $Y_{e,v} = \varnothing$, i.e., the empty set.
	\vskip0.1in

	\noindent{\em The cobordism $C$:}
	Now, from the construction of $Y_{e,v}$, one can see that if $v \in V(I)$, then $Y_{e,v}$ has two boundary components such that each of them is a Legendrian sphere in $\partial_\infty P$. 
	Moreover, one can also see that the positive Reeb flow identifies a boundary component of $X_v$ (resp.\ $X_e$) with one of two boundary components of $Y_{e,v}$. 
	Thus, it is easy to construct a cobordism $C$ connecting $X_v$ (resp.\ $X_e$) and $Y_{e,v}$ for all pair $(e,v) \in E(Q) \times V(Q)$ where $v$ is one endpoint of $e$. 
	
	Figure \ref{figure cobordism} is a conceptual picture describing the Lagrangians constructed above, near $\Pi_n$ corresponding to $e = v \to w \in E(Q)$.
	\begin{figure}[h]
		\centering
\begingroup%
  \makeatletter%
  \providecommand\color[2][]{%
    \errmessage{(Inkscape) Color is used for the text in Inkscape, but the package 'color.sty' is not loaded}%
    \renewcommand\color[2][]{}%
  }%
  \providecommand\transparent[1]{%
    \errmessage{(Inkscape) Transparency is used (non-zero) for the text in Inkscape, but the package 'transparent.sty' is not loaded}%
    \renewcommand\transparent[1]{}%
  }%
  \providecommand\rotatebox[2]{#2}%
  \newcommand*\fsize{\dimexpr\f@size pt\relax}%
  \newcommand*\lineheight[1]{\fontsize{\fsize}{#1\fsize}\selectfont}%
  \ifx\svgwidth\undefined%
    \setlength{\unitlength}{226.77165354bp}%
    \ifx\svgscale\undefined%
      \relax%
    \else%
      \setlength{\unitlength}{\unitlength * \real{\svgscale}}%
    \fi%
  \else%
    \setlength{\unitlength}{\svgwidth}%
  \fi%
  \global\let\svgwidth\undefined%
  \global\let\svgscale\undefined%
  \makeatother%
  \begin{picture}(1,1.00062504)%
    \lineheight{1}%
    \setlength\tabcolsep{0pt}%
    \put(0,0){\includegraphics[width=\unitlength,page=1]{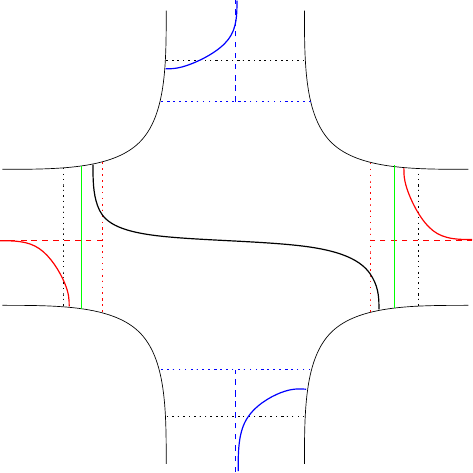}}%
    \put(0.36021423,0.51298237){\color[rgb]{0,0,0}\makebox(0,0)[lt]{\lineheight{1.25}\smash{\begin{tabular}[t]{l}$C_{e,v}'$\end{tabular}}}}%
    \put(0.10300157,0.59871998){\color[rgb]{0,0,0}\makebox(0,0)[lt]{\lineheight{1.25}\smash{\begin{tabular}[t]{l}$Y_{e,v}$\end{tabular}}}}%
    \put(0.04892095,0.40877829){\color[rgb]{0,0,0}\makebox(0,0)[lt]{\lineheight{1.25}\smash{\begin{tabular}[t]{l}$X_v$\end{tabular}}}}%
    \put(0.50578775,0.96475632){\color[rgb]{0,0,0}\makebox(0,0)[lt]{\lineheight{1.25}\smash{\begin{tabular}[t]{l}$X_w$\end{tabular}}}}%
    \put(0,0){\includegraphics[width=\unitlength,page=2]{cobordism.pdf}}%
  \end{picture}%
\endgroup%
		
		\caption{This figure describes Lagrangians constructed above, near a $\Pi_n$ corresponding to an arrow $e = v \to w \in E(Q)$.
			Especially, the black dotted lines are boundaries of (the extended) $\Pi_n$ and the red and blue dotted lines are boundaries of (the extended) $T^*\widetilde{M}_v$ and $T^*\widetilde{M}_w$. 
			The red and blue dashed lines correspond to $\widetilde{M}_v$ and $\widetilde{M}_w$, respectively.
			Since the arrows point Reeb directions along the asymptotic boundary, the red and blue thick lines describe $X_v$ and $X_w$. 
			Similarly, the center black thick line corresponds to $C_{e,w}'$.
			And, finally, if we assume that $v \in V(I)$, the green corresponds to $Y_{e,v}$, which is homeomorphic to $T^*_0[-1,1] \times S^{n-1}$.
			Finally, the purple part describes the cobordism connecting $\widetilde{M}_v'$ and $C_{e,v}'$.}
		\label{figure cobordism}
	\end{figure} 
	Note that Figure \ref{figure cobordism} describes the case of $n=1$. 
	Thus, the intersection of two Weinstein sectors $T^*\widetilde{M}_v$ and $\Pi_n$ is homeomorphic to $T^*[-1,1] \times T^*S^0$, i.e., two copies of $T^*[-1,1]$, which are bounded by black and red dotted lines. 
	The cobordism $C$ connecting $X_v$ and $Y_{e,v}$ are purple curves connecting red and green lines in the picture. 
	We note that that part of $C$ can be seen as a product $C_0 \times S^0 \subset T^*[-1,1] \times T^*S^0$. 
	For a general $n$, we construct a cobordism $C$ as a product of $C_0 \times S^{n-1} \subset T^*[-1,1] \times T^*S^{n-1}$ in the intersections of Weinstein sectors. 
	
	As we spoke at the beginning of the proof, one can easily see that the fixed immersed Lagrangian $\left(L = \cup_{v \in V(I)} M_v, \rho=\oplus_{v \in V(I)} \rho_v, b = \oplus_{e \in E(I), d_e =0} b_e\right)$ is Hamiltonian equivalent to 
	\[\sqcup_{v \in V(Q)} X_v \#^C \sqcup_{e \in E(Q)} X_e \#^C \sqcup_{(e, v)} Y_{e,v},\]
	where $\#^C$ denotes the gluing via cobordism $C$.
	Then, the proof of \cite[Proposition 1.37]{gps2} proves that the fixed immersed Lagrangian is equivalent to the twisted complex consisting of $X_v, X_e, Y_{e,v}$ in $\cW^{im}$.
	And, since all of $X_v, X_e, Y_{e,v}$ are generated by cocores in $\cW(P;d_e)$, so is the fixed Lagrangian.  
\end{proof}

We note that Lemma \ref{lem immersed Lagrangian} is a special case of Theorem \ref{thm immersed generation recalled}, but the proof is different from that of Theorem \ref{thm immersed generation recalled}. 
In \cite{Jeong-Karabas-Lee26}, Theorem \ref{thm immersed generation recalled} is proven by applying the techniques from \cite{cdrgg17}, but the above proof of Lemma \ref{lem immersed Lagrangian} employees \cite{gps2}, under Assumption \ref{assumption 2}. 

\begin{rmk}
	\label{rmk twisted complex}
	In the above proof of Lemma \ref{lem immersed Lagrangian}, we said that $(L,\rho,b)$ would be equivalent to a twisted complex consisting of $X_v, X_e, Y_{e,v}$. 
	The proof of \cite[Proposition 1.37]{gps2} specifies the twisted complex.
	More specifically, 
	\[\left(L,\rho,b\right) \simeq\left( \left(\oplus_v X_v\right) \bigoplus \left(\oplus_e X_e \right)\right) \stackrel{f}{\to} \left( \oplus_{e,v} Y_{e,v}	\right),\]
	where $f$ is given by the Reeb flow identifying $\partial_\infty X_v \cup \partial_\infty X_e$ and $\partial_\infty Y_{e,v}$. 
\end{rmk}

Now, we define our {\em compact Fukaya category} of a plumbing space $P$ as follows: 
\begin{dfn}
	\label{def compact Fukaya category}
	Let $P = P(Q,M,\s)$ be a plumbing space equipped with a non-positively grading $\{d_e\}$. 
	Then, the {\em compact Fukaya category} of $P$, denoted by $\cF(P;d_e)$, is the subcategory of $\cW(P;d_e)$ generated by all compact, possibly immersed Lagrangians that are admissible in $\cW(P;d_e)$. 
\end{dfn}

Note that Definition \ref{def compact Fukaya category} is equivalent to Definition \ref{dfn compact Fukaya category}, if $W$ in Definition \ref{dfn compact Fukaya category} is a plumbing space $P$.

\subsection{Proof of Theorem \ref{thm intro compact Lagrangian}}
\label{subsection proof of compact Fukaya category generation theorem}
Let $P = P(Q,M,\s)$ be a plumbing space equipped with a non-positively grading $\{d_e\}$.
For convenience, we set 
\[\cW := \cW(P;d_e) \text{  and  } \cF:= \cF(P;d_e).\]

In this subsection, we reprove our main result of Section \ref{section the compact Fukaya category of plumbings}.
To prove Theorem \ref{thm intro compact Lagrangian}, we recall that the Yoneda functor sends $\cW(P;d_e)$ to $\Mod \cW$. 
And the Yoneda image of the subcategory $\cF$ is contained in the category of proper modules, i.e., $\Prop \cW$. 
Thus, for any generator $E$ of $\Prop \cW$ given in Theorem \ref{thm generation of proper modules}, if one can find an immersed Lagrangian $\left(L, \rho, b\right) \in \cF$ such that $\cY_\cW(L,\rho,b) = E$, then Theorem \ref{thm intro compact Lagrangian} holds. 

We note that in Theorem \ref{thm generation of proper modules}, $E$ is in $\mathcal{C}_{(C_0,0)}$ where $C_0 \in \mathbb{B}$. 
By definition, $C_0$ can be seen as a sub-quiver of $Q$.
Thus, by applying Section \ref{subsection immersed Lagrangians in plumbings}, one can construct a collection of (possibly immersed) Lagrangians equipped with local systems and bounding cochains. 
\begin{dfn}
	\label{def collection of generating Lagrangians}
	For $C_0 \in \mathbb{B}$, let $\mathcal{L}_{C_0}$ denote the collection of Lagrangians equipped with local systems and bounding cochains $\left(L, \rho, b\right)$ that are constructed in Section \ref{subsection immersed Lagrangians in plumbings} for a fixed sub-quiver $C_0$ of $Q$. 
\end{dfn}

Now, we can restate Theorem \ref{thm intro compact Lagrangian} and (re)prove it. 

\begin{thm}[Theorem \ref{thm intro compact Lagrangian}]
	\label{thm compact Fukaya}
	Let $P$ be a plumbing space corresponding to plumbing data $(Q,M,\s)$ equipped with non-positively graded $\{d_e\}$. 
	Then, the compact Fukaya category of $P$, $\cF(P;d_e)$, is generated by 
	\[\bigsqcup_{C_0 \in \mathbb{B}} \mathcal{L}_{C_0}.\]
\end{thm}

\begin{rmk}\label{rmk cores generate compact Lagrangians}
	Theorem \ref{thm compact Fukaya} recovers several well-known generation results in symplectic topology:
	For cotangent bundles, the compact Fukaya category is generated by the zero section equipped with local systems, as implied by, for example, \cite{nad-zas09,nad09}.
	For plumbings of $T^*S^n$ with $n\geq 6$ along trees, \cite[Theorem 1.1]{abo-smi12} showed that the compact Fukaya category is generated by Lagrangian cores equipped with local systems.
	Theorem \ref{thm compact Fukaya} extends these examples to plumbings along arbitrary quivers satisfying the non-positivity condition.
	More precisely, in the presence of degree-$0$ cycles, additional generators appear, namely unions of Lagrangian cores equipped with bounding cochains determined by the degree-$0$ cycles.
\end{rmk}

\begin{proof}[Proof of Theorem \ref{thm compact Fukaya}]
	It is enough to show that for a randomly given $E \in \cC_{(C,0)}$, we can construct $(L^E,\rho^E,b^E) \in \mathcal{L}_{C_0}$ such that $\cY(L^E,\rho^E,b^E) = E$. 
	
	It is easy to construct Lagrangian submanifold $L$ (possibly immersed) as 
	\[L^E = \bigcup_{v \in V(I_C)} M_v.\]
	
	We also note that $E$ is concentrated at degree $0$. 
	Thus, for every $v \in V(I_C)$ and a proper $i$, $E(\bar{\alpha}_v^i): E(L_v) \to E(L_v)$ can be a nonzero linear map only if the degree of $\bar{\alpha}_v^i$ is zero. 
	(See Notation \ref{notation morphisms} for the notation $\bar{\alpha}_v^i$.) 
	Since $\alpha_v^i$ is a generating morphisms of the based loop space of $M_v \ \text{pt}$, we can see that the nonzero linear maps $E(\bar{\alpha}_v^i)$ defines a local system on $M_v$. 
	Let $\rho_v^E$ denote the constructed local system on $M_v$. 
	Then, we set $\rho^E := \oplus_{v \in V(I_C)} \rho_v^E$. 
	
	We point out that from the definition, one has 
	\[\hom_\cW^*\left(M_v^{\rho_v^E}, M_w^{\rho_w^E}\right) = \hom^0\left(E(L_v), E(L_w)\right) \otimes \hom_\cW^*\left(M_v, M_w\right),\]
	for any $v, w \in V(Q)$. 
	
	Now, for an arrow $e = t \to h \in E(Q)$, let $p_e$ denote the plumbing point in $P$, which is corresponding to the arrow $e$.
	Then, as a morphism from $M_h$ to $M_t$, $p_e$ is of degree $(1-d_e)$. 
	If $e = t \to h \in E(C) = E(I_C)$ with $d_e =0$, then 
	\[E(\bar{x}_e) \in \hom^0\left(E(L_h), E(L_t)\right).\]
	We set 
	\[b_e^E:= E(\bar{x}_e) \otimes p_e \in \hom^0\left(E(L_h), E(L_t)\right) \otimes \hom_\cW^1\left(M_h, M_t\right) = \hom_\cW^1\left(M_h^{\rho_h^E}, M_t^{\rho_t^E}\right).\]
	And, let $b^E := \oplus_{e \in E(I_C), d_e=0} b_e^E$. 
	Then, it is easy to show that $b^E$ satisfies the Maurer-Cartan equation, since (higher-)products between plumbing points $p_e$ vanish.
	In other words, $b^E$ could be a bounding cochain.
	
	We constructed an immersed Lagrangian with a local system and a bounding cochain $\left(L^E, \rho^E, b^E\right)$.
	Thanks to Lemma \ref{lem immersed Lagrangian}, the constructed $\left(L^E, \rho^E, b^E\right)$ is an admissible object in $\cW$. 
	Thus, it is also an object of $\cF$. 
	
	Moreover, the proof of Lemma \ref{lem immersed Lagrangian} shows that the constructed Lagrangian is equivalent to a twisted complex consisting of $X_v, X_e$, and $Y_{e,v}$.
	See Remark \ref{rmk twisted complex}.
	An easy computation shows that the Yoneda image of the twisted complex is equivalent to the starting proper module $E$.
\end{proof}

\section{Applications and examples}
\label{section applications and examples}

In the section, we discuss applications of our main results, especially, Theorems \ref{thm intro generation} and \ref{thm intro compact Lagrangian}. 
We discuss mainly two topics, {\em Bridgeland stability conditions} and {\em cluster category}, which are related to the categories of our interests, i.e., generalized Ginzburg dg algebras in representation theory and Fukaya categories of plumbings in symplectic topology. 
We also give example computations for previous sections in the last subsection. 

Before starting the section, we note that in Section \ref{section applications and examples}, we use different notations from the previous sections. 
The reason of the notational change is that we would like to use similar notations with previous works in representation theory, for example \cite{kel-yan11, iya-yan18}.
(On the other hand, in the previous sections, we used the word {\em wrapped Fukaya category} and $\cW$ denoting it, in order to emphasize the connection to symplectic topology.)
We first set the notations we will use in Section \ref{section applications and examples}.

Let $\mathcal{G}$ be a generalized Ginzburg dg algebra which comes from wrapped Fukaya categories of plumbing spaces, introduced in Theorem \ref{thm wrapped Fukaya category}. 
In this section, we consider the derived perfect module category over $\mathcal{G}$ (which is equivalent to the wrapped Fukaya category $\cW$ of the corresponding plumbing space, by Theorem \ref{thm wrapped Fukaya category}) and the derived proper module category over $\mathcal{G}$ (which is equivalent to $\Prop \cW$). 
We denote them by $\mathrm{per}(\mathcal{G})$ and $D_{\mathrm{fd}}(\mathcal{G})$, respectively.

In Section \ref{sec:stab}, we construct specific stability conditions on the derived proper module category $D_{\mathrm{fd}}(\mathcal{G})$. 
Thanks to Theorem \ref{thm compact Fukaya}, it is equivalent to construct a stability condition on the compact Fukaya category of a plumbing space. 

In Section \ref{sec:CYtri}, we prove that if $\mathcal{G}$ is Jacobi-finite (see Definition \ref{dfn Jacobi-finite}), $\mathrm{per}(\mathcal{G}), D_{\mathrm{fd}}(\mathcal{G})$ and $\mathcal{G}$ form a Calabi--Yau triple. 
Again, thanks to Theorems \ref{thm wrapped Fukaya category} and \ref{thm compact Fukaya}, for some plumbing spaces (satisfying the Jacobi-finite condition), the wrapped Fukaya category, the compact Fukaya category, and the endomorphism algebra of the selected generator (given by the union of cocores) for a Calabi--Yau triple. 
This is a generalization of the result in \cite{Bae-Jeong-Kim25}, which showed the same statement for the special case when the underlying graph of $Q$ is a tree and $M_{v} = S^{n}$ $(n \geq 3)$ for all $v$.

\subsection{Stability conditions}
\label{sec:stab}

In this section, we construct explicit stability conditions on the derived proper module category $D_{\mathrm{fd}}(\mathcal{G})$. 
We note that Bridgeland \cite{bri07} defined the notion of stability conditions, and Kontsevich and Soibelman \cite{Kontsevich-Soibelman08} introduced a modified version. 
We start the section by recalling the original definition of Bridgeland \cite{bri07}, then we construct an explicit stability condition with respect to the original definition. 
After that, we also introduce the modified definition, and we show that the constructed stability condition also satisfies the modified definition. 
We end the section by briefly mentioning Joyce's conjecture \cite{joy15}. 
We note that the conjecture is about stability conditions on compact Fukaya category. 
Since our category $D_{\mathrm{fd}}(\mathcal{G})$ can be seen as the compact Fukaya category of plumbing spaces, the plumbings could be a test ground for studying Joyce's conjecture. 

Here is the original definition of \cite{bri07}.
\begin{dfn} \label{dfn:tristab}
	Let $\mathcal{T}$ be a triangulated category. A {\em Bridgeland stability condition} $\sigma$ on $\mathcal{T}$ is a pair $(Z, \mathcal{P})$ satisfying the following:
	\begin{enumerate}
		\item[(i)] $Z : K(\mathcal{T}) \to \mathcal{C}$ is a group homomorphism where $K(\mathcal{T})$ is the $K$-group of $\mathcal{T}$.
		\item[(ii)] $\mathcal{P}$ is a collection of full subcategories $\left\{\mathcal{P}(\phi)\right\}_{\phi \in \mathbb{R}}$ of $\mathcal{T}$.
		\item[(iii)] Given a nonzero object $X \in \mathcal{P}(\phi)$, $Z(X) = m(X) e^{i \pi \phi}$ for some $m(X) \in \mathbb{R}_{>0}$.
		\item[(iv)] $\mathcal{P}(\phi+1) = \mathcal{P}(\phi)[1]$ for all $\phi \in \mathbb{R}$.
		\item[(v)] For $X_{1} \in \mathcal{P}(\phi_{1})$ and $X_{2} \in \mathcal{P}(\phi_{2})$ where $\phi_{1} > \phi_{2}$, $\Hom_{\mathcal{T}}(X_{1},X_{2}) = 0$.
		\item[(vi)] For any nonzero object $X \in \mathcal{P}(\phi)$, there exists a collection of exact triangles
		\begin{equation}
			\begin{tikzcd}
				\label{eqn tower 2}
				X_0 = 0 \arrow[r] & X_1 \arrow[r] \arrow[d] & X_2 \arrow[r] \arrow[d] & \dots \arrow[r] & X_{\ell-1} \arrow[r] \arrow[d] & X_\ell = X \arrow[d] \\
				& Y_{1} \arrow[lu, dashed] & Y_2 \arrow[lu, dashed] & \dots & Y_{\ell-1} \arrow[lu, dashed] & Y_{\ell} \arrow[lu, dashed]
			\end{tikzcd}
		\end{equation}
		where $Y_{i} \in \mathcal{P}(\phi_{i})$ for all $1 \leq i \leq l$ such that $\phi_{1}>\phi_{2}>\dots>\phi_{l}$.
	\end{enumerate}
\end{dfn}

Bridgeland \cite{bri07} provided a construction of stability condition on a triangulated category $\mathcal{T}$, if $\mathcal{T}$ admits a {\em bounded t-structure}.
To introduce the construction, we define the following:
\begin{dfn}
	\label{dfn t-structure}
	Let $\mathcal{A}$ be an abelian category. 
	\begin{enumerate}
		\item A {\em stability function} on $\mathcal{A}$ is a homomorphism $Z:K(\mathcal{A}) \to \mathbb{C}$ such that $Z(E)$ lies in the strict upper half plane $H \subset \mathbb{C}$ for all nonzero object $E \in \mathcal{A}$.
		\item The number $\phi(E) := \frac{1}{\pi} \mathrm{arg} Z(E) \in (0,1]$ is called the {\em phase} of $E \neq 0$.
		\item A nonzero object $E \in \mathcal{A}$ is called {\em semi-stable} if for any nonzero subobject $F \subset E$, the inequality $\phi(F) \leq \phi(E)$ holds.
		\item For a nonzero object $E \in \mathcal{A}$, a {\em Harder--Narasimhan filtration} of $E$ is a finite chain of subobjects
		\[0 = E_{0} \subset E_{1} \subset \dots \subset E_{n-1} \subset E_{n} = E,\]
		such that any quotient $F_{i}:=E_{i}/E_{i-1}$ is semi-stable and the inequality $\phi(F_{1}) > \dots > \phi(F_{n})$ holds.
		\item A stability function $Z$ is said to have the {\em Harder--Narasimhan property} if every nonzero object has a Harder--Narasimhan filtration.
	\end{enumerate}
\end{dfn}

We also recall the definition of $t$-structure.

\begin{dfn}
	Let $\mathcal{T}$ be a triangulated category.
	\begin{enumerate}
		\item A {\em t-structure} is a pair of two subcategories $(\mathcal{T}^{\leq 0},\mathcal{T}^{\geq 0})$ satisfying the following properties.
		\begin{enumerate}
			\item[(i)] For $X \in \mathcal{T}^{\leq 0}$ and $Y \in \mathcal{T}^{\geq 0}$, $\mathrm{Hom}_{\mathcal{T}}(X[1],Y) = 0$.
			\item[(ii)] For $X \in \mathcal{T}^{\leq 0}$ and $Y \in \mathcal{T}^{\geq 0}$, $X[1] \in \mathcal{T}^{\leq 0}$ and $Y[-1] \in \mathcal{T}^{\geq 0}$
			\item[(iii)] For every object $Z \in \mathcal{T}$, there exists a triangle $X \to Z \to Y \to X[1]$ such that $X \in \mathcal{T}^{\leq 0}$ and $Y \in \mathcal{T}^{\geq 0}$.
		\end{enumerate}
		\item A $t$-structure $(\mathcal{T}^{\leq 0},\mathcal{T}^{\geq 0})$ is called {\em bounded} if for every object $X \in \mathcal{T}$, there exists an integer $m>0$ such that $X[m] \in \mathcal{T}^{\leq 0}$ and $X[-m] \in \mathcal{T}^{\geq 0}$.
	\end{enumerate}
\end{dfn}

We note that in \cite{Beuilinson-Bernstein-Deligne82}, it is known that the heart of a t-structure is an abelian category. 
Now, we can state the construction of Bridgeland.

\begin{thm}[{\cite[Theorem 5.3]{bri07}}]\label{thm:stabequiv}
	To give a stability condition on a triangulated category $\mathcal{T}$ is equivalent to giving a bounded t-structure on $\mathcal{T}$ and a stability function on its heart with the Harder--Narasimhan property.
\end{thm}

We use Theorem \ref{thm:stabequiv} to construct a stability condition on the derived proper module category $D_{\mathrm{fd}}(\mathcal{G})$. 
As the first step, we recall that Lemma \ref{lem extra condition 1} gives a bounded t-structure on $D_{\mathrm{fd}}(\mathcal{G})$. 
To be more precise, let us recall Lemma \ref{lem extra condition 1} (1), i.e., that every $M \in D_{\mathrm{fd}}(\mathcal{G})$ admits a tower 
\begin{equation}
	\label{eqn t-structure}
	\begin{tikzcd}
		0 \arrow[r] & \ast \arrow[r] \arrow[d] & \dots  \arrow[r] & M \arrow[d] \\
		& M_1 \arrow[lu, dashed] & \dots           & M_k \arrow[lu, dashed]
	\end{tikzcd},
\end{equation}
such that 
\begin{itemize}
	\item $M_i$ is concentrated at degree $s_i$ and 
	\item $s_1 < s_2 < \dots < s_k$. 
\end{itemize}
Thus, the following defines a bounded t-structure of $D_{\mathrm{fd}}(\mathcal{G})$.
\begin{gather*}
	D_{\mathrm{fd}}(\mathcal{G})^{\leq 0} = \{ M \in D_{\mathrm{fd}}(\mathcal{G}) | \text{  In Equation \eqref{eqn t-structure}, $M_i$ is concentrated at $s_i \leq 0$ for all $i$} \}, \\
	D_{\mathrm{fd}}(\mathcal{G})^{\geq 0} = \{ M \in D_{\mathrm{fd}}(\mathcal{G}) | \text{  In Equation \eqref{eqn t-structure}, $M_i$ is concentrated at $s_i \geq 0$ for all $i$} \}.
\end{gather*}
We remark that the given t-structure is the standard t-structure on the derived category of a dg category, given in \cite[Proposition 3.1.1]{GLV22}.
Moreover, the heart of the above bounded t-structure consists of proper modules concentrated at the degree $0$. 

\begin{thm} \label{thm:stab}
	There exists a Bridgeland stability condition on $D_{\mathrm{fd}}(\mathcal{G})$.
\end{thm}
\begin{proof}
	Let $\mathcal{H}$ be the heart of the above bounded t-structure on $D_{\mathrm{fd}}(\mathcal{G})$, i.e., $\mathcal{H} = D_{\mathrm{fd}}(\mathcal{G})^{\leq 0} \cap D_{\mathrm{fd}}(\mathcal{G})^{\geq 0}$.
	We would like to find a stability function with the Harder--Narasimhan property on the abelian category $\mathcal{H}$.
	Then, by Theorem \ref{thm:stabequiv}, we get a stability condition on $D_{\mathrm{fd}}(\mathcal{G})$.
	
	To define a stability function $Z$, we use the partition of $V(Q)$(, where $Q$ is the quiver associated to $\mathcal{G}$) defined in Definition \ref{dfn equivalence relation 1}.
	We recall that two vertices $v, w \in V(Q)$ are equivalent, i.e., $v \sim w$ if and only if there exist nonzero degree $0$ morphisms $\overline{f} \in \hom_{\cW^{op}}^0(L_v,L_w)$ and $\overline{g} \in \hom_{\cW^{op}}^0(L_w,L_v)$.
	Then, the equivalence relation defines a partition of $V(Q)$ with index set $I$, i.e., 
	\[V(Q)= \sqcup_{s \in I} V_s(Q).\]
	We also recall that $I$ could have a partial order $\preceq$ by Lemma \ref{lem partial order 1}.
	
	Let $I = \{ s_1, \dots, s_l \}$. 
	For simplicity, we assume that $s_i \preceq s_j$, then $i < j$.  
	Now, we assign a complex number $z_i$ to $s_i$ for all $i$ such that 
	\begin{itemize}
		\item if $s_i \preceq s_j$ then $\mathrm{arg}(z_i) \geq \mathrm{arg}(z_j)$, and 
		\item $0 < \mathrm{arg}(z_i) \leq \pi$.
	\end{itemize}
	
	We finally define our stability function.
	For $M \in \mathcal{H}$, we set 
	\begin{gather}
		\label{eqn central charge}
		Z(M) = \sum_{s_i \in I} \left(\sum_{v \in V_{s_i}(Q)} \dim M(L_v)\right) z_i.
	\end{gather}
	We note that, as mentioned in Remark \ref{rmk HMod}, we are assuming that $M$ is of the form $\oplus_d H^d(M)[-d]$, and since $M$ is concentrated at the degree $0$, $M = H^0(M)$. 
	Thus, two equivalent proper modules in $\mathcal{H}$ have the same function value. 
	Moreover, it is trivial that $Z$ can be seen as a linear map on the Grothendieck group $K(\mathcal{H})$.
	Lastly, the Harder--Narasimhan property of $Z$ follows from Lemma \ref{lem extra condition 2}. 
	More precisely, the tower in Equation \eqref{eqn extra condition 2} can be rearranged as the Harder--Narasimhan filtration with respect to the choice of $z_i$. 
\end{proof}

As mentioned at the beginning of the section, \cite{Kontsevich-Soibelman08} provides a modified version of Definition \ref{dfn:tristab}.
We first introduce the modified version and we show that the constructed stability conditions in Theorem \ref{thm:stab} satisfy the definition of \cite{Kontsevich-Soibelman08}.
We also discuss its relation to Joyce's conjecture \cite{joy15} at the end of Section \ref{sec:stab}.
\begin{dfn}
	\label{dfn modified stability condition}
	Let $\mathcal{T}$ be a triangulated category, $\Gamma$ be a finite rank free abelian group, and $\mathrm{cl}: K(\mathcal{T}) \to \Gamma$ be a group homomorphism. 
	A {\em stability condition} $\sigma = \left(Z_\Gamma, \mathcal{P}\right)$ consists of 
	\begin{itemize}
		\item a group homomorphism $Z_\Gamma: \Gamma \to \mathbb{C}$, and 
		\item a collection of full additive subcategories $\mathcal{P}= \{\mathcal{P}(\phi)\}_{\phi \in \mathbb{R}}$ of $\mathcal{T}$,
	\end{itemize}
	such that $(Z:=Z_\Gamma \circ \mathrm{cl}, \mathcal{P})$ is a Bridgeland stability condition defined in Definition \ref{dfn:tristab}.
\end{dfn}

The main difference between Definitions \ref{dfn:tristab} and \ref{dfn modified stability condition} is that the second one is using a fixed free abelian group $\Gamma$, which does not appear in the first one. 
It would be a natural question to ask a choice of $\Gamma$.

In symplectic topology, if one believes the conjecture of Joyce, there exists a natural choice of $\Gamma$. 
To be more precise, let us recall that Joyce \cite{joy15} conjectured that a holomorphic volume form $\Omega$ on a symplectic manifold $X$ defines a stability condition in its Fukaya category.
Especially, special Lagrangians with respect to $\Omega$ would correspond to the semistable objects, and the central charge function is defined as follows: 
\[Z(L) = \int_L \Omega.\]

We recall that $\Omega$ is a holomorphic $n$-form where $n$ is the complex dimension of the symplectic manifold $X$. 
Thus, one can see $\int \Omega$ as a group homomorphism from the middle homology $H_n(X;\mathbb{Z})$ to $\mathbb{C}$. 
It implies that $H_n(X;\mathbb{Z})$ would be a natural choice for $\Gamma$ in Definition \ref{dfn modified stability condition}.

In our case, the corresponding symplectic manifold is a plumbing space. 
And, using the Mayer--Vietoris sequence, the above argument suggests to choose $\Gamma = H_n(P;\mathbb{Z}) = \mathbb{Z}^{|V(Q)|}$, with a natural function $\mathrm{cl}$. 
Then, one can easily construct a stability condition in the sense of Definition \ref{dfn modified stability condition}. 

\begin{cor}
	\label{cor modified stability condition}
	The exists a stability condition on $D_{\mathrm{fd}}(\mathcal{G})$.
\end{cor}
\begin{proof}
	It is easy to check that $Z$ defined in \eqref{eqn central charge} factors through $\mathrm{cl}:K(\mathcal{H}) \to \Gamma$. 
\end{proof}

\subsection{Calabi--Yau triples}
\label{sec:CYtri}

In \cite{iya-yan18}, they defined the notion of Calabi--Yau triples. 
In order to introduce the definition, we need to define the following terminologies:
Let $\mathcal{T}$ be a triangulated category.
A split-generator $M \in \mathcal{T}$ is called a {\em silting object} if it satisfies $\mathrm{Hom}_\mathcal{T}(M,M[d])=0$ for all $d>0$. 
Here, $\mathrm{Hom}_\mathcal{T}$ stands for a morphism space in the triangulated category $\mathcal{T}$. 
(In particular, when $\mathcal{T}$ is the derived category of a dg category $\mathcal{D}$, $\mathrm{Hom}_\mathcal{T}$ is the same as $\mathrm{Hom}^{0}_\mathcal{D}$.)
For any object $M \in \mathcal{T}$, we define two subcategories:
\begin{gather*}
	\mathcal{T}_M^{\leq 0} \coloneqq \{ X \in \mathcal{T} \,|\, \mathrm{Hom}_\mathcal{T}(M,X[d])=0 \text{ for all } d>0 \},\\
	\mathcal{T}_M^{\geq 0} \coloneqq \{ X \in \mathcal{T} \,|\, \mathrm{Hom}_\mathcal{T}(M,X[d])=0 \text{ for all } d<0 \}.
\end{gather*}

\begin{dfn}
	Let $\mathcal{T}$ be a split-closed triangulated category, $M \in \mathcal{T}$ be an object and $\mathcal{U} \subset \mathcal{T}$ be a thick subcategory.
	\begin{enumerate}
		\item The triple $(\mathcal{T},\mathcal{U},M)$ is called an {\em ST-triple} if it satisfies the following conditions:
		\begin{enumerate}
			\item[(i)] $M$ is a silting object and $\mathrm{Hom}_\mathcal{T}(M,X)$ is finite dimensional for every $X \in \mathcal{U}$.
			\item[(ii)] $(\mathcal{T}_M^{\leq 0},\mathcal{T}_M^{\geq 0})$ is a t-structure on $\mathcal{T}$.
			\item[(iii)] $\mathcal{T}_M^{\geq 0} \subset \mathcal{U}$ and $(\mathcal{U}_M^{\leq 0} = \mathcal{T}_M^{\leq 0} \cap \mathcal{U} ,\mathcal{U}_M^{\geq 0} = \mathcal{T}_M^{\geq 0} \cap \mathcal{U})$ is a bounded t-structure on $\mathcal{U}$.
		\end{enumerate}
		\item An ST-triple $(\mathcal{T},\mathcal{U},M)$ is called a {\em $n$-Calabi--Yau triple} ($n \geq 2$) if it is a ST-triple satisfying the following extra condition:
		\begin{enumerate}
			\item[(iv)] $(\mathcal{T},\mathcal{U})$ is {\em relative $n$-Calabi--Yau}: for $X \in \mathcal{U}, Y \in \mathcal{T}$,
			$$D\mathrm{Hom}_{\mathcal{T}}(X,Y) \simeq \mathrm{Hom}_{\mathcal{T}}(Y,X[n]).$$
		\end{enumerate}
	\end{enumerate}

\end{dfn}

The following is a standard example.

\begin{exa}[{\cite[Lemma 4.15]{amy19}}]\label{exa:CY}
	Let $R$ be a dg algebra satisfying the following conditions:
	\begin{enumerate}
		\item $H^d(R)=0$ for any positive integer $d>0$.
		\item $H^0(R)$ is finite-dimensional.
		\item $\mathrm{per}(R) \supseteq D_{\mathrm{fd}}(R)$.
	\end{enumerate}
	Then $(\mathrm{per}(R), D_{\mathrm{fd}}(R),R)$ is an ST-triple.
	If $R$ is also $n$-Calabi--Yau, then $(\mathrm{per}(R), D_{\mathrm{fd}}(R))$ is relative $n$-Calabi--Yau (see e.g., \cite[Corollary 2.4.2]{Bae-Jeong-Kim25}) and hence $(\mathrm{per}(R), D_{\mathrm{fd}}(R),R)$ is a $n$-Calabi--Yau triple.
\end{exa}

We note the following facts: 
It is well-known that a wrapped Fukaya category $\mathcal{W}$ is a (homologically) smooth and (non-compact) Calabi--Yau $A_{\infty}$-category, see \cite{gan13}. 
By definition, it implies that a generalized Ginzburg dg algebra $\mathcal{G}$ is also homologically smooth and $n$-Calabi--Yau where $2n$ is the dimension of a plumbing space. 
In particular, $\mathrm{per}(\mathcal{G}) \supseteq D_{\mathrm{fd}}(\mathcal{G})$ holds since $\mathcal{G}$ is homologically smooth, see \cite{kel08}. 
Thus, if $\mathcal{G}$ is Jacobi-finite (see Definition \ref{dfn Jacobi-finite}), then Example \ref{exa:CY} implies Theorem \ref{thm:CY}.

\begin{dfn}
	\label{dfn Jacobi-finite} 
	A generalized Ginzburg algebra $\mathcal{G}$ is {\em Jacobi-finite} if $H^0(\mathcal{G})$ is finite-dimensional.
\end{dfn}

\begin{thm} \label{thm:CY}
	Let $\mathcal{G}$ be a Jacobi-finite, generalized Ginzburg dg algebra associated to a non-positively graded plumbing data of dimension $n \geq 3$. Then, $(\mathrm{per}(\mathcal{G}), D_{\mathrm{fd}}(\mathcal{G}),\mathcal{G})$ is a $n$-Calabi--Yau triple.
\end{thm}
\begin{proof}
	We need to check the condition (1) in Example \ref{exa:CY} only. 
	It is trivial since the plumbing data is non-positively graded.
\end{proof}

Let $(Q, M, \s)$ be a plumbing data equipped with a non-positively grading $\{d_e\}$.
Now, we investigate a geometric condition which implies the Jacobi-finiteness of the corresponding generalized Ginzburg dg algebra $\mathcal{R}$. 

Thanks to Theorem \ref{thm wrapped Fukaya category}, the degree $0$ part of $\mathcal{G}$ is dependent on the following two factors:
\begin{itemize}
	\item $C_{0}(\Omega_p (M_v\setminus\text{pt}))$ (that determines $\alpha_v$ morphisms in Notation \ref{notation morphisms}), and
	\item the degree $0$ cycles of $Q$. 
\end{itemize}
Thus, Corollary \ref{cor CY triple} is straightforward. 
\begin{cor} 
	\label{cor CY triple}
	Let $(Q, M, \mathrm{sgn}, d_{e})$ be a plumbing data and $P$ be the resulting plumbing space of dimension $2n$.
	Let $L_v$ denote the cotangent fiber of $T^*M_v \subset P$.
	If the following two hold, 
	\begin{itemize}
		\item $|\pi_1(M_v)| < \infty$, and
		\item $Q$ has no degree $0$ cycle,
	\end{itemize}
	$\left(\cW(P;d_e), \cF(P;d_e), \oplus_v L_v\right)$ is a Calabi--Yau triple.
\end{cor}
\begin{proof}
	We note that if $|\pi_{1}(M_{v})| < \infty$, then there are only finitely many morphisms $\alpha_v$ (from $ C_{0}(\Omega_p (M_v\setminus\text{pt}))$).
	Thus, the corresponding generalized Ginzburg dg algebra is Jacobi-finite, and Theorem \ref{thm:CY} completes the proof.
\end{proof}

As another corollary, we can construct infinitely many generalized cluster categories.

\begin{cor}
	Under the setting of Theorem \ref{thm:CY}, the Verdier quotient $\mathrm{per}(\mathcal{G})/D_{\mathrm{fd}}(\mathcal{G})$ is a generalized cluster category.
\end{cor}
\begin{proof}
	This follows from \cite[Section 5.3]{iya-yan18}.
\end{proof}
We note that the Verdier quotient $\mathrm{per}(\mathcal{G})/D_{\mathrm{fd}}(\mathcal{G})$ can be seen as a quotient $\cW(P;d_e) / \cF(P;d_e)$, where $P$ is the plumbing space satisfying the conditions of Corollary \ref{cor CY triple}.
If $\cW(P;d_e)$ and $\cF(P;d_e)$ are Koszul dual to each other, then \cite[Corollary 1.4]{Ganatra-Gao-Venkatesh22} provides a geometric model of the cluster category, which is the Rabinowitz Fukaya category of the plumbing space.

\subsection{Example}
\label{subsection example}
We end the article by giving an example.

We first fix our field $\mathbb{C}$.
Then, we fix plumbing data with grading $(Q,M,\s, d_e)$ for the example. 
The quiver $Q$ is the following: 
\[\begin{tikzcd}
	& z & \\
	x \ar[rr,"a", bend left=10] \ar[ur, "c"]  & & y \ar[ll,"b",bend left=10] \ar[ul, "d"'].
\end{tikzcd}\]
For all vertices $x, y, z$, we assign spheres $S^n$ with $n \geq 3$. 
For all arrows $a, b, c, d$, we assign degree $0$ and the positive signs. 

The corresponding algebra is the usual Ginzburg dg algebra associated to $Q$:
We note that the all assigned manifolds are the sphere, so there is no $\alpha_v$ morphisms. 
Let $\Gamma$ denote the associated Ginzburg dg algebra. 
We also note that the quiver $Q$ has a cycle $C_0:=ab$ that makes $\Gamma$ to be non Jacobi-finite. 

We recall the notation for the generating collection in Theorem \ref{thm generation of proper modules}:
The collection is 
\[\cup_{C \in \mathbb{B}} \cC_{(C,0)}.\]
It is easy to check that $\mathbb{B}$ for this example is $\{x, y, z, C_0\}$. 

It is also easy to check that each of $\cC_{(x,0)}, \cC_{(y,0)}, \cC_{(z,0)}$ has a unique simple module, given as follows:
\[M_x = \begin{tikzcd}
	& 0 \ar[dr] \ar[dl]& \\
	\mathbb{C} \ar[rr, bend right=10]  & & 0 \ar[ll,bend right=10] 
\end{tikzcd},
M_y = \begin{tikzcd}
	& 0 \ar[dr] \ar[dl]& \\
	0 \ar[rr, bend right=10]  & & \mathbb{C} \ar[ll,bend right=10]
\end{tikzcd},
M_z = \begin{tikzcd}
	& \mathbb{C} \ar[dr] \ar[dl]& \\
	0 \ar[rr, bend right=10]  & & 0 \ar[ll,bend right=10]
\end{tikzcd}.\]

Similarly, it is easy to check that $\cC_{(C_0,0)}$ has infinitely many simple modules. 
For example, for each $\lambda \in \mathbb{C}^*$, the following $M_\lambda$ is a simple module in $\cC_{(C_0,0)}$:
\[M_\lambda:=\begin{tikzcd}
	& 0 \ar[dr] \ar[dl]& \\
	\mathbb{C} \ar[rr, "\lambda"', bend right=10]  & & \mathbb{C} \ar[ll, "1"', bend right=10].
\end{tikzcd}\]
Also, for a fixed $r \in \mathbb{N}$, let $\phi_{r, \lambda}$ be the following $r$-by-$r$ matrix
\[\phi_{r,\lambda} = \begin{bmatrix}
	\lambda & 1 & 0 & \dots & 0 \\
	0 & \lambda & 1 & \dots & 0 \\ 
	0 & 0 & \lambda & \dots & 0 \\
	\vdots & \vdots & \vdots & \ddots & \vdots \\
	0 & 0 & 0 & \dots & \lambda
\end{bmatrix}.\]
Then, the following $M_{\lambda,r}$ is also an indecomposable module:
\[M_{r,\lambda} = \begin{tikzcd}
	& 0 \ar[dr] \ar[dl]& \\
	\mathbb{C}^r \ar[rr, "\phi_{r,\lambda}"', bend right=10]  & & \mathbb{C}^r \ar[ll, "Id_r"', bend right=10]
\end{tikzcd}.\] 
We note that $M_\lambda = M_{1, \lambda}$.
Also, since every square matrix admits the Jordan normal form, one can easily see that indecomposable modules in $\cC_{(C_0,0)}$ are $\{M_{r,\lambda} | r \in \mathbb{N}, \lambda \in \mathbb{C}^*\}$.
It implies that to generate $\Prop \Gamma$, we need infinitely many generators. 

\begin{rmk}
	\label{rmk indecomposable modules in an example}
	We note that a random module in $\cC_{(C_0,0)}$ is given as a direct sum of indecomposable modules $\left\{M_{r,\lambda} | r \in \mathbb{N}, \lambda \in \mathbb{C}^*\right\}$.
\end{rmk}

From the viewpoint of symplectic topology or Theorem \ref{thm compact Fukaya}, the above generators correspond to Lagrangians in the corresponding plumbing space. 
We note that the simple modules $M_x, M_y, M_z$ correspond to Lagrangian spheres.
Let us denote these Lagrangian sphere $L_x, L_y, L_z$. 

Similarly, for each pair $(r, \lambda) \in \mathbb{N} \times \mathbb{C}^*$, one can see $M_{r, \lambda}$ as an immersed Lagrangian with a local system and a bounding cochain as follows: 
Let $\rho_{x,r}$ (resp.\ $\rho_{y,r}$) be the trivial local system of rank $r$ on $L_x$ (resp.\ $L_y$).
Since $L_x$ and $L_y$ have two intersection points, denoted as $p$ and $q$, and since one can assume that $p \in \Hom^1_\cW(L_x,L_y)$ and $q \in \Hom^1_\cW(L_y,L_x)$,  
\[\Hom_\cW^1\left((L_x, \rho_{x,r}), (L_y, \rho_{y,r})\right) = \Hom\left(\mathbb{C}^r, \mathbb{C}^r\right) \otimes \langle p \rangle, \Hom_\cW^1\left((L_y, \rho_{y,r}),(L_x, \rho_{x,r})\right) = \Hom\left(\mathbb{C}^r, \mathbb{C}^r\right) \otimes \langle q \rangle.\]
Then, $M_{r,\lambda}$ is the following immersed Lagrangian with a local system and a bounding cochain 
\[\left( L_x \cup L_y, \rho_{x,r} \oplus \rho_{y,r}, \begin{bmatrix}
	0 & \phi_{r,\lambda} \otimes p \\ Id_r \otimes q & 0
\end{bmatrix}\right).\]

On the other hand, by performing Lagrangian surgeries corresponding to the bounding cochain, one can see $M_{r, \lambda}$ as a Lagrangian homeomorphic to $S^1 \times S^{n-1}$ with a nontrivial local system of rank $r$.  
It can be seen as a higher dimensional analogy of a result in Keating \cite[Sections 1 and 2]{Keating15}.

The proof of Theorem \ref{thm:stab} constructs stability conditions on $\Prop \Gamma$.
Thus, the space of stability condition is not empty. 
Moreover, since its Grothendieck group has infinite rank, thanks to \cite[Theorem 1.2]{bri07}, one can see that the space the space of Bridgeland stability conditions is of infinite dimensional.

Also, Corollary \ref{cor modified stability condition} gives a stability condition of $\Prop \Gamma$ in terms of Definition \ref{dfn modified stability condition}. 
We point out that, as mentioned at the end of Section \ref{sec:stab}, the stability conditions in Joyce's conjecture \cite{joy15} would be the modified version, if the free abelian group in Definition \ref{dfn modified stability condition} is provided as the (torsion free part of the) middle degree homology of the symplectic manifold. 
Since Corollary \ref{cor modified stability condition} constructs such stability conditions, we could expect that the constructed stability conditions would correspond to holomorphic volume forms on the plumbing space, as Joyce \cite{joy15} conjectured. 

Using the above example, we briefly introduce a possible strategy of finding a corresponding holomorphic volume form $\Omega$ such that the stability condition which is related to $\Omega$ by \cite{joy15} is the one constructed by Corollary \ref{cor modified stability condition}.
To do that, let us recall the following facts. 

We first recall that $T^*S^n$ could be seen as a complex submanifold of $\mathbb{C}^{n+1}$. 
More precisely, 
\[T^*S^n = \left\{(z_0, \dots, z_n) \in \mathbb{C}^{n+1} | z_0^2 + \dots z_n^2 =1 \right\}.\]
Then, one can restrict the standard holomorphic volume form $\Omega = dz_0 \wedge \dots \wedge d z_n$ on $\mathbb{C}^{n+1}$ to $T^*S^n$. 
The restriction gives a holomorphic volume form on $T^*S^n$ such that the zero section is a special Lagrangian of phase $0$. 
Let $\Omega_{std}$ denote the holomorphic volume form on $T^*S^n$. 

Secondly, let us fix a stability condition $\sigma$ which is constructed in the proof of Corollary \ref{cor modified stability condition}.
We recall that $\sigma$ is constructed by assigning $z_1 \in \mathbb{C}$ (resp.\ $z_2 \in \mathbb{C}$) to the vertices $x, y$ (resp.\ $z$) such that $\pi \geq \mathrm{arg}(z_1) > \mathrm{arg}(z_2) > 0$. 
Then, the central charge $Z$ satisfies that 
\[Z(L_x) = Z(L_y) = z_1, Z(M_{r,\lambda}) = 2rz_1, Z(L_z) = z_2.\]

On a cotangent bundle $T^*L_x = T^*S^n$, one could define a holomorphic volume form $z_1 \Omega_{std}=: \Omega_x$.
Moreover, according to this holomorphic volume form, the zero section $L_x$ is a special Lagrangian whose central charge value is $z_1$. 
Similarly, on $T^*L_y$ and $T^*L_z$, $\Omega_y:=z_1 \Omega_{std}$ and $\Omega_z:= z_2 \Omega_{std}$ are holomorphic volume forms such that the zero sections are special Lagrangians with central charge values $z_1$ and $z_2$. 

Now, we recall that the plumbing space is obtained by identifying $T^*L_x, T^*L_y, T^*L_z$ in the plumbing regions. 
Thus, if one can identifying $\Omega_x, \Omega_y, \Omega_z$ in the identified regions, one could construct a holomorphic volume form $\Omega$ on the plumbing space, which corresponds to the starting stability condition $\sigma$. 
In this article, we do not prove that the holomorphic volume forms can be identified in the regions, but we expect one can easily prove (or disprove) it since the identified regions are simple: The regions are topologically $\mathbb{D}^n \times \mathbb{D}^n$.

If the above strategy works, plumbing spaces could be a test ground to study {\em Thomas--Yau type conjecture} of \cite{joy15}, i.e., \cite[Conjecture 3.34]{joy15}, which conjectured about the {\em Lagrangian Mean curvature Flow} of a random, admissible Lagrangian in a Calabi--Yau manifold. 
More precisely, for an arbitrary Lagrangian $L$ in a Calabi--Yau manifold, i.e., a symplectic manifold with a compatible holomorphic volume form, Thomas--Yau type conjecture asks if the Lagrangian Mean Curvature Flow of $L$ converges to a combination of special Lagrangians (possibly with singularities). 
If we have a stability condition corresponding to the given holomorphic volume form, then the expected special Lagrangians should correspond to the components of the Harder--Narasimhan tower of $L$. 

\bibliographystyle{amsalpha}
\bibliography{proper_module_and_compact_Fukaya.bib}

\end{document}